\documentclass{article}
\usepackage[T1]{fontenc}
\usepackage{amsthm, fullpage}
\usepackage{amsmath,enumerate}
\usepackage{amssymb}
\usepackage{amsfonts}
\usepackage{eucal}
\usepackage[all]{xy}
\usepackage{pslatex}
\usepackage{hyperref}
\providecommand{\scr}{\mathcal}
\usepackage{mathrsfs}

\newtheorem{prop}{Proposition}[subsection]

\newtheorem{coro}[prop]{Corollary}
\newtheorem{lemm}[prop]{Lemma}
\newtheorem{lem}[prop]{Lemma}
\newtheorem*{lemm*}{Lemma}

\theoremstyle{definition}

\newtheorem{empt}[prop]{}
\newtheorem{dfn}[prop]{Definition}
\newtheorem{rem}[prop]{Remark} 
\newtheorem{ntn}[prop]{Notation} 
 
\newtheorem*{rem*}{Remark}

\theoremstyle{thm}
\newtheorem{thm}[prop]{Theorem}
\newtheorem*{thm*}{Theorem}
\newtheorem*{lem*}{Lemma}

\newtheorem*{cor*}{Corollary}
\newtheorem*{prop*}{Proposition}

\theoremstyle{dfn}
\newtheorem*{dfn*}{Definition}

\numberwithin{equation}{prop}

\newcommand{\riso}{ \overset{\sim}{\longrightarrow}\, }

\newcommand{\Spec}{\mathrm{Spec}\,}
\newcommand{\Spf}{\mathrm{Spf}\,}

\newcommand{\gr}{\mathrm{gr}}

\newcommand{\FF}{{\mathcal{F}}}

\newcommand{\E}{{\mathcal{E}}}
\newcommand{\G}{{\mathcal{G}}}
\renewcommand{\H}{{\mathcal{H}}}
\newcommand{\M}{{\mathcal{M}}}
\newcommand{\NN}{{\mathcal{N}}}
\newcommand{\D}{{\mathcal{D}}}

\newcommand{\PP}{{\mathcal{P}}}
\newcommand{\QQ}{{\mathcal{Q}}}
\renewcommand{\O}{{\mathcal{O}}}

\newcommand{\V}{\mathcal{V}}

\newcommand{\Y}{\mathcal{Y}}

\newcommand{\X}{\mathfrak{X}}

\newcommand{\U}{\mathfrak{U}}

\newcommand{\A}{\mathbb{A}}
\renewcommand{\P}{\mathbb{P}}
\newcommand{\F}{\mathbb{F}}
\newcommand{\C}{\mathbb{C}}
\newcommand{\DD}{\mathbb{D}}
\renewcommand{\L}{\mathbb{L}}
\newcommand{\R}{\mathbb{R}}
\newcommand{\Q}{\mathbb{Q}}
\newcommand{\Z}{\mathbb{Z}}
\newcommand{\N}{\mathbb{N}}

\newcommand{\hdag}{  \phantom{}{^{\dag} }    }
\DeclareMathOperator{\sym}{Sym}
\DeclareMathOperator{\supp}{Supp}
\DeclareMathOperator{\red}{red}
\DeclareMathOperator{\coker}{coker}

\begin{document}

\title{Betti number estimates in $p$-adic cohomology}
\author{Daniel Caro} 

\date{}

\maketitle

\begin{abstract}
In the framework of Berthelot's theory of arithmetic $\D$-modules, 
we prove the $p$-adic analogue of Betti number estimates and we give some standard
applications. 

\end{abstract}

\tableofcontents

\bigskip

\section*{Introduction}
Let $k$ be a perfect field of characteristic $p$ and $l$ be a prime number different from $p$.
When $k$ is algebraically closed, in the framework of Grothendieck's  $l$-adic etale cohomology of $k$-varieties,
Bernstein, Beilinson and Deligne in their famous paper on perverse sheaves,
more precisely in \cite[4.5.1]{BBD} (or see the $p$-adic translation here in \ref{BBD4.5.1}), established some Betti number estimates. 
The goal of this paper is to get the same estimates in the context of Berthelot's arithmetic $\D$-modules. 
We recall that this theory of Berthelot gives a $p$-adic cohomology stable under six operations (see \cite{caro-Tsuzuki})
and admitting a theory of weights (see \cite{Abe-Caro-weights}) analogous to that of Deligne in the $l$-adic side (see \cite{deligne-weil-II}). 
This allows us to consider Berthelot's theory as a right $p$-adic analogue
of Grothendieck's $l$-adic etale cohomology. 
By trying to translate the proof of Betti number estimates in \cite[4.5.1]{BBD} in the framework of arithmetic $\D$-modules, 
two specific problems appear.
The first one is that we do not have a notion of local acyclicity in the theory of arithmetic $\D$-modules. 
We replace the use of this notion by another one that we might call 
``relative generic $\O$-coherence''.
The goal of the first chapter is to prove this property. 
The proof of this relative generic $\O$-coherence 
uses the precise description of the characteristic variety of a unipotent overconvergent $F$-isocrystal 
(see \cite{Caro-Lagrangianity}). 
Berthelot's characteristic variety of a holonomic arithmetic $\D$-module endowed with a Frobenius structure. 
The second emerging problem when we follow the original $l$-adic proof of Betti number estimates is that we still do not have vanishing cycles theory as nice as in the $l$-adic framework 
(so far, following \cite{Abe-Caro-BEq} we only have a $p$-adic analogue of Beilinson's unipotent nearby cycles and vanishing cycles). 
Here, we replace successfully in the original proof on Betti number estimates 
the use of vanishing cycles by that of some Fourier transform and of Abe-Marmora formula (\cite[4.1.6.(i)]{AbeMarmora})
relating the irregularity of an isocrystal with the rank of its Fourier transform. 
We conclude this paper by the remark that these Betti number estimates 
allow us to state that the results of \cite[chapters 4 and 5]{BBD} are still valid (except 
\cite[5.4.7--8]{BBD} because the translation is not clear so far).

\subsection*{Acknowledgment}
I would like to thank Ambrus P\'al for his invitation at the Imperial College of London.
During this visit, one problem that we studied was to get 
a $p$-adic analogue of Gabber's purity theorem published in \cite{BBD}. We noticed that
the $p$-adic analogue of this proof was not obvious since contrary to the $l$-adic case we did not have $p$-adic vanishing cycles.
I would like to thank him for the motivation he inspired to overcome this problem.

\subsection*{Convention, notation of the paper}
Let $\V$ be a complete discrete valued ring of mixed characteristic $(0,p)$, 
$K$ its field of fractions,  
$k$ its residue field which is supposed to be perfect, $\pi$ be a uniformizer of $\V$. 
Let $F _k \colon k \to k$ be the Frobenius map given by $x \mapsto x ^p$. 
When we deal with Frobenius structures, 
we suppose that there exists  a lifting $\sigma _{0}\colon \V \to \V$ of the Frobenius map $F _k$ that we will fix. 
A $k$-variety is a separated reduced scheme of finite type over $k$.
We say that a $k$-variety $X$ is realizable if there exists an immersion of the form $X \hookrightarrow \PP$, 
where $\PP$ is a proper smooth formal scheme over $\V$.
In this paper, $k$-varieties will always be supposed realizable. 
For any $k$-variety $X$, we denote by $ p _{X} \colon X \to \Spec k $ the canonical morphism.
We will denote formal schemes by curly or gothic letters and the corresponding straight roman letter will
mean the special fiber (e.g. if $\X$ is a formal scheme over $\V$, then $X$ is the $k$-variety equal to the special fiber of $\X$).
The  underlying topological space of a $k$-variety $X$ is denoted by $|X|$.
When $M$ is a $\V$-module, we denote by 
$\widehat{M}$ its $p$-adic completion and we set
$M _\Q := M \otimes _{\V} K$.
By default, a module will mean a left module. 
Moreover, is $f \colon \PP ' \to \PP$ is a morphism of formal schemes over $\V$, we denote by $\L f ^{*}$ the functor defined by putting 
$\L f ^{*} (\M) =  \O _{\PP', \Q} \otimes ^{\L} _{f ^{-1} \O _{\PP, \Q}} f ^{-1} \M$, for any bounded below complex $\M$ of $\O _{\PP, \Q}$-modules.
When $f $ is flat, we remove $\L$ in the notation. 

If $T \to S$ is a morphism of schemes and $f\colon X \to Y$ is a $S$-morphism, 
then we denote by $f _T \colon X _T \to Y _T$ or simply by 
$f  \colon X _T \to Y _T$ the base change of $f$ by $T \to S$.

Concerning the cohomological operations of the theory of arithmetic $\D$-modules of Berthelot, we follow
the usual notation (for instance, see the beginning of \cite{Abe-Caro-weights}).
More precisely, let $S$ be a noetherian scheme such that $p$ is nilpotent in $\O _S$. 
Let $f \colon X \to Y$  be morphism of quasi-compact smooth $S$-schemes. 
If $f$ is smooth, then the extraordinary pull-back of level $m$ by $f$ has the factorisation
$f ^{! ^{(m)}}\colon 
D ^\mathrm{b} _{\mathrm{coh}} (\D ^{(m)} _{Y/S})
\to 
D ^\mathrm{b} _{\mathrm{coh}} (\D ^{(m)} _{X/S})$
  (see \cite[2.2.4]{Beintro2}).
If $f$ is proper, then 
the pushforward of level $m$ by $f$ has the factorisation
$f _{+ ^{(m)}}\colon 
D ^\mathrm{b} _{\mathrm{coh}} (\D ^{(m)} _{X/S})
\to 
D ^\mathrm{b} _{\mathrm{coh}} (\D ^{(m)} _{Y/S})$
  (see \cite[2.4.4]{Beintro2}).
  When there is no ambiguity with the basis $S$, we remove ``$/S$'' is the notation.

Let $f\colon \PP \to \QQ$ be a morphism of quasi-compact smooth formal $\V$-schemes. 
If $f$ is smooth, then we have the  extraordinary pull-back of level $m$ by $f$ of the form
$f ^{! ^{(m)}}\colon 
D ^\mathrm{b} _{\mathrm{coh}} (\widehat{\D} ^{(m)} _{\QQ})
\to 
D ^\mathrm{b} _{\mathrm{coh}} (\widehat{\D} ^{(m)} _{\PP})$
and 
$f ^{! ^{(m)}}\colon 
D ^\mathrm{b} _{\mathrm{coh}} (\widehat{\D} ^{(m)} _{\QQ,\Q})
\to 
D ^\mathrm{b} _{\mathrm{coh}} (\widehat{\D} ^{(m)} _{\PP,\Q})$
(see \cite[3.4.6]{Beintro2}),
and we have the extraordinary pull-back by $f$ of the form 
$f ^{! }
\colon 
D ^\mathrm{b} _{\mathrm{coh}} (\D ^{\dag} _{\QQ,\Q})
\to 
D ^\mathrm{b} _{\mathrm{coh}} (\D ^{\dag} _{\PP,\Q})$
(see \cite[4.3.4]{Beintro2}).
If $f$ is proper, then we have the push-forward of level $m$ of the form
$f _{+ ^{(m)}}\colon 
D ^\mathrm{b} _{\mathrm{coh}} (\widehat{\D} ^{(m)} _{\PP})
\to 
D ^\mathrm{b} _{\mathrm{coh}} (\widehat{\D} ^{(m)} _{\QQ})$
and
$f  _{+ ^{(m)}}\colon 
D ^\mathrm{b} _{\mathrm{coh}} (\widehat{\D} ^{(m)} _{\PP,\Q})
\to 
D ^\mathrm{b} _{\mathrm{coh}} (\widehat{\D} ^{(m)} _{\QQ,\Q})$
(see \cite[3.5.3]{Beintro2}),
and the push-forward by $f$ of the form 
$f  _{+}
\colon 
D ^\mathrm{b} _{\mathrm{coh}} (\D _{\PP,\Q})
\to 
D ^\mathrm{b} _{\mathrm{coh}} (\D _{\QQ,\Q})$
(see \cite[4.3.8]{Beintro2}).

Let $a \colon X \to Y$ be a morphism of (realizable) $k$-varieties.
By definition, there exist immersions
$\iota \colon X \hookrightarrow \PP$ and 
$\iota ' \colon Y \hookrightarrow \QQ$ where $\PP$ and $\QQ$ are proper smooth formal schemes over $\V$.
Replacing $\PP$ by $\PP \times \QQ$, we can suppose that 
there exist a (proper) smooth morphism of formal $\V$-schemes of the form
$f \colon \PP \to \QQ$ such that $f \circ \iota = \iota ' \circ a$. 
By definition, $D ^{\mathrm{b}} _{\mathrm{ovhol}} (X,\PP/K)$
is the full subcategory of $D ^\mathrm{b} _{\mathrm{ovhol}} (\D ^\dag _{\PP,\Q})$ (the derived category of overholonomic complexes
of $\D ^\dag _{\PP,\Q}$-modules)
of the objects $\E$ 
such that there exists an isomorphism of the form
$\R \underline{\Gamma} ^\dag _X (\E) \riso \E$ (see \cite[1.1.6]{Abe-Caro-weights}). 
Since this category $D ^{\mathrm{b}} _{\mathrm{ovhol}} (X,\PP/K)$
does not depend on the choice of $\iota$,
we simply denote it by $D ^{\mathrm{b}} _{\mathrm{ovhol}} (X/K)$
(see Definition \cite[1.1.5]{Abe-Caro-weights}).
The extraordinary pull-back by $a$ is by definition 
$\R \underline{\Gamma} ^\dag _X \circ f ^!
\colon 
 D ^{\mathrm{b}} _{\mathrm{ovhol}} (Y,\QQ/K)
 \to 
 D ^{\mathrm{b}} _{\mathrm{ovhol}} (X,\PP/K)$, 
 which is simply denoted by 
$ a ^! \colon 
D ^{\mathrm{b}} _{\mathrm{ovhol}} (Y/K)
 \to 
 D ^{\mathrm{b}} _{\mathrm{ovhol}} (X/K)$ (again, we chech that this does not depend on 
 $\iota$, $\iota'$ and $f$).
The pushforward by $a$ is by definition 
$f _+
\colon 
 D ^{\mathrm{b}} _{\mathrm{ovhol}} (X,\PP/K)
 \to 
 D ^{\mathrm{b}} _{\mathrm{ovhol}} (Y,\QQ/K)$, 
 which is simply denoted by 
$ a _+ \colon 
D ^{\mathrm{b}} _{\mathrm{ovhol}} (X/K)
 \to 
 D ^{\mathrm{b}} _{\mathrm{ovhol}} (Y/K)$.
 We have also the dual functor
 $\DD _X := \R \underline{\Gamma} ^\dag _X \circ \DD _{\PP}\colon 
D ^{\mathrm{b}} _{\mathrm{ovhol}} (X/K)
\to 
D ^{\mathrm{b}} _{\mathrm{ovhol}} (X/K)$.
Then we get 
$a _! := \DD _Y \circ a _+ \circ \DD _X$
and 
$a ^+ := \DD _X \circ a ^! \circ \DD _Y$.
There is a canonical t-structure on 
$D ^{\mathrm{b}} _{\mathrm{ovhol}} (X/K)$ defined as follows: 
if $\U$ is an open set of $\PP$ so that $X$ is closed in $\U$
then $D ^{\leq n} _{\mathrm{ovhol}} (X/K)$ 
(resp. $D ^{\geq n} _{\mathrm{ovhol}} (X/K)$) 
is the subcategory of $D ^{\mathrm{b}} _{\mathrm{ovhol}} (X/K)$
of complexes $\E$ such $\E |\U \in D ^{\leq n} _{\mathrm{ovhol}} (\D ^\dag _{\U,\Q})$
(resp. $D ^{\geq n} _{\mathrm{ovhol}} (\D ^\dag _{\U,\Q})$), 
where the t-structure on 
$D ^\mathrm{b} _{\mathrm{ovhol}} (\D ^\dag _{\U,\Q})$ is the obvious one.
The heart of this t-structure is denoted by 
$\mathrm{Ovhol} (X/K)$ 
(see Definition \cite[1.2.6]{Abe-Caro-weights}).

Suppose $X$ smooth. Following \cite[1.2.14]{Abe-Caro-weights}, 
we have a full subcategory $D ^{\mathrm{b}} _{\mathrm{isoc}} (X /K)$
of $D ^{\mathrm{b}} _{\mathrm{ovhol}} (X /K)$ whose cohomological spaces (for the above t-structure)
belong to  $\mathrm{Isoc} ^{\dag\dag} (X /K)$ (the category of overconvergent isocrystals on $X/K$). 
Recall that $\mathrm{Isoc} ^{\dag\dag} (X /K)$ is equivalent to the category of overconvergent isocrysals on $X/K$
denoted by $\mathrm{Isoc} ^{\dag} (X /K)$.

If $j \colon U \hookrightarrow X$ is an open immersion of (realizable) varieties, 
the functor $j ^{!}\colon D ^{\mathrm{b}} _{\mathrm{ovhol}} (X/K) \to 
D ^{\mathrm{b}} _{\mathrm{ovhol}} (U/K)$
(or the functor 
$j ^! \colon \mathrm{Ovhol} (X/K) \to \mathrm{Ovhol} (U/K)$)
will simply be denoted by $| U$ (in other papers, to avoid confusion, it was sometimes denoted 
by $\Vert U$ but, here, there is no such risk since we do not work ``partially'').

\medskip 

Let $s$ be a positive integer and $\sigma = \sigma _0 ^{s}\colon \V \to \V$ the corresponding lifting of the $s$th power of the Frobenius map
$F ^s _k\colon k\to k$.
If  $X$ is a $k$-variety (resp. $\PP$ is a smooth formal $\V$-scheme) 
then we denote by $X ^\sigma$ (resp.  $\PP ^{\sigma}$) the corresponding $k$-scheme of finite type (resp. smooth formal $\V$-scheme) induced by the
base change by $F ^s _k$ (resp. $\sigma$). 
We will denote by $F ^s _{X/k} \colon X \to X ^{\sigma}$
the corresponding relative Frobenius
which is a morphism of $k$-schemes. 
Notice, when $X$ is $k$-smooth, $F ^s _{X/k}$ is a morphism of smooth $k$-varieties.
When it exists (e.g. when $\PP$ is affine), we will denote by 
$F ^s _{\PP/\V} \colon \PP \to \PP ^{\sigma}$ 
a morphism of smooth formal $\V$-schemes 
which is a lifting of $F ^s _{X/k} \colon X \to X ^{\sigma}$.
The functor 
$D ^\mathrm{b} _{\mathrm{coh}} (\D ^\dag _{\PP,\Q})
\to D ^\mathrm{b} _{\mathrm{coh}} (\D ^\dag _{\PP ^\sigma,\Q})$
induced by the isomorphism $\PP ^\sigma \riso \PP$
is denoted by $\E \mapsto \E ^\sigma$.
We have the functor
$(F ^s _{P/k}) ^! \colon D ^\mathrm{b} _{\mathrm{coh}} (\D ^\dag _{\PP ^\sigma,\Q})
\to D ^\mathrm{b} _{\mathrm{coh}} (\D ^\dag _{\PP ,\Q})$.
Recall that when $F ^s _{P/k}$
has a lifting  $F ^s _{\PP/\V}\colon \PP \to \PP ^{\sigma}$ then we have $(F ^s _{P/k}) ^! = (F ^s _{\PP/\V}) ^!$
(in general, even if the lifting $F ^s _{\PP/\V}$ is not unique we can glue these functors: e.g. see \cite[2.1]{Be2}).
Finally, we get the functor 
$F ^* \colon D ^\mathrm{b} _{\mathrm{coh}} (\D ^\dag _{\PP,\Q}) 
\to 
D ^\mathrm{b} _{\mathrm{coh}} (\D ^\dag _{\PP,\Q})$
which is defined for any 
$\E \in D ^\mathrm{b} _{\mathrm{coh}} (\D ^\dag _{\PP,\Q})$
by setting
$F ^* (\E) := 
 (F ^s _{P/k}) ^!(\E ^\sigma)$.
 The derived category of overholonomic $F$-complexes of $\D ^\dag _{\PP,\Q}$-modules,
 denoted by 
 $F\text{-}D ^\mathrm{b} _{\mathrm{ovhol}} (\D ^\dag _{\PP,\Q})$,
 is the category whose objects are the data of 
 an object $\E$ of $D ^\mathrm{b} _{\mathrm{ovhol}} (\D ^\dag _{\PP,\Q})$
 endowed with a Frobenius structure, i.e. 
 an isomorphism $\phi$ of $D ^\mathrm{b} _{\mathrm{ovhol}} (\D ^\dag _{\PP,\Q})$   of the form
$\phi \colon F ^* \E \riso \E$.
We get similarly the category
$F\text{-}D ^{\mathrm{b}} _{\mathrm{ovhol}} (X/K)$ of overholonomic $F$-complexes on $X/K$.
When $X$ is smooth, we define similarly the categories of $F$-objects
$F\text{-}D ^{\mathrm{b}} _{\mathrm{isoc}} (X /K)$ 
and 
$F\text{-}\mathrm{Isoc} ^{\dag\dag} (X /K)$ (see  \cite[1.2.14]{Abe-Caro-weights}).

\section{Relative generic $\O$-coherence}

\subsection{Preliminaries on cotangent spaces}

\begin{ntn}
\label{stab-T_X X}
Let $X$ be a smooth $k$-variety. 
For any quasi-coherent $\O _X$-module $\E$,
we denote by $\sym (\E) $ the symetric algebra of $\E$  
and by $\mathbb{V} (\E): = \Spec (\sym (\E) )$ endowed with its canonical projection
$\mathbb{V} (\E) \to \Spec \sym (\O _X) = X$.
We denote by $\Omega ^{1} _X $ the sheaf of differential form of $X/\Spec (k)$ (we skip $k$ in the notation),
and $\mathcal{T} _X$ the tangent space of $X/\Spec (k)$, i.e. the $\O _X$-dual of $\Omega ^{1} _X $.
We denote by $T ^{*} X:= \mathbb{V} (\mathcal{T} _X)$ the cotangent space of $X$
and $\pi _X \colon T ^{*} X \to X$ the canonical projection.
Recall that from \cite[1.7.9]{EGAII}, there is a canonical bijection between
sections of $\pi _X$ and $\Gamma (X, \Omega ^{1} _X)$.
We denote by $T ^* _X X$ the section corresponding to the zero section of 
$\Gamma (X, \Omega ^{1} _X)$.
If $t _1, \dots, t _d$ are local coordinates of $X$, we get local coordinates
$t _1 , \dots, t _d, \xi _1, \dots, \xi _d$ of $T ^{*} X$, where $\xi _i$ is the element associated with
$\partial _i$, the derivation with respect to $t _i$. 
Is this case, $T ^* _X X= V ( \xi _1, \dots, \xi _d)$ is the closed subvariety of 
$T ^* X$ defined by $\xi _1 =0, \dots, \xi _d =0$.

Let $f \colon X \to Y$ be a morphism of smooth $k$-varieties. 
Using the equality \cite[1.7.11.(iv)]{EGAII}
we get the last one
$X \times _{Y} T ^* Y =
X \times _{Y} \mathbb{V} (\mathcal{T}_Y) 
=\mathbb{V} (f ^*\mathcal{T} _Y)$.
The morphism $f ^* \Omega ^{1} _Y \to \Omega ^{1} _X $ induced by $f$
yields by duality 
$\mathcal{T} _X \to f ^* \mathcal{T} _Y$ and then by functoriality
$\mathbb{V} (f ^*\mathcal{T} _X)
\to \mathbb{V} (\mathcal{T} _Y) = T ^* Y$.
By composition, we get the morphism
denoted by
$\rho _{f}\colon  X \times _{Y} T ^* Y \to T ^{*} X $.
We will write by 
$\varpi _f \colon X \times _{Y} T ^* Y \to T ^*Y$ the base change of $f$ under
$\pi _Y$ (instead of $f _{T ^* _Y}\colon X \times _{Y} T ^* Y \to T ^*Y$ which seems too heavy).

We denote by $\mathscr{T} _f$ the function from the set of subvarieties of $T ^{*} X$ to the set 
of subvarieties of $T ^* Y$ defined by posing, for any subvariety $V$ of $T ^{*} X$, 
$\mathscr{T} _{f} (V) :=  \varpi _{f} ( \rho _{f} ^{-1} (V))$.
If $f$ is an open immersion,
then $\rho _f$ is an isomorphism. In that case,
$\mathscr{T} _{f} := \varpi _f \circ \rho _{f} ^{-1} \colon T ^* X \to T ^* P$ 
is an open immersion and this is compatible with the above definition 
of $\mathscr{T} _{f}$.
The application $\mathscr{T} \colon f \mapsto \mathscr{T} _f$ is transitive (with respect to the composition), 
i.e. we have the equality 
$\mathscr{T} _ g \circ \mathscr{T} _f  = \mathscr{T} _{g\circ f}$ for any $g \colon Y \to Z$
 (e.g. look at the bottom of the diagram \ref{lem-incluT*strat-diag} where $f$ and $u$ are replaced respectively by $g$ and $f$).

 We define the $k$-variety $T ^{*} _{X} Y$ (recall a $k$-variety is a separated reduced scheme of finite type over $k$
from our convention) 
by setting $T ^{*} _{X} Y := \rho _{f} ^{-1} ( T ^* _X X)$. 
When $f$ is an immersion, 
$T ^{*} _{X} Y$ is viewed as a subvariety of $T ^* Y$ via 
$T ^{*} _{X} Y \subset X \times _Y T ^* Y \overset{\varpi _f}{\hookrightarrow} T ^* Y$,
i.e. we will simply denote
$\varpi _f (T ^{*} _{X} Y)$ by $T ^{*} _{X} Y$.

\end{ntn}

\begin{lem}
\label{lem-incluT*strat}
Let $u\colon Z \to X$ and $f\colon X \to P$ 
be two morphisms of smooth $k$-varieties. 
\begin{enumerate}
\item We have the equality
$(Z \times _{X}\rho _{f} )^{-1} ( T ^{*} _Z X) = T ^{*} _{Z} P$.
When $u$ is an immersion, this might be written of the form
$\rho _{f} ^{-1} ( T ^{*} _Z X) = T ^{*} _{Z} P$
or
$\mathscr{T} _{f} (T ^{*} _Z X) =\varpi _f (T ^{*} _Z P)$.
If $u $ and $f$ are immersions,  this might be written
$\mathscr{T} _{f} (T ^{*} _Z X) =T ^{*} _Z P$.
Finally, when $u $ is an immersion and $f$ is an open immersion, 
we might identify $T ^{*} _Z X $ and $T ^{*} _Z P$.
\item When $u$ is an immersion (resp. an open immersion), we have the inclusion
$Z \times _{X} T ^* _X P \subset  T ^{*} _Z P$
(resp. the equality $Z \times _{X} T ^* _X P = T ^{*} _Z P$)
 in $Z \times _{P}T ^{*} P$.

\end{enumerate}

\end{lem}

\begin{proof}
1) First, let us prove part 1) of the Lemma. 
The composition 
$(f \circ u ) ^* \Omega _P 
\riso 
u ^* \circ  f ^* \Omega _P
\to u ^* \Omega _X \to \Omega _Z$
is the canonical one. 
 Indeed, since this is local, then we reduce to the case where varieties are affine
and then this is checked by an easy computation.
This implies that the composition
$Z \times _{P}T ^{*} P
\overset{Z \times _{X}\rho _{f}}{\longrightarrow}
Z \times _{X}T ^{*} X
\overset{\rho _{u}}{\longrightarrow}
T ^{*} Z$
is equal to $\rho _{f\circ u}$.
Consider the following diagram
\begin{equation}
\label{lem-incluT*strat-diag}
\xymatrix{
{T ^{*} _{Z} Z} 
\ar@{^{(}->}[d] ^-{}
   \ar@{}[dr]|\square
& 
{T ^{*} _{Z} X} 
\ar[l] ^-{}
\ar@{^{(}->}[d] ^-{}
   \ar@{}[dr]|\square
& 
{T ^{*} _{Z} P} 
\ar[l] ^-{}
\ar@{^{(}->}[d] ^-{}
\\
{T ^{*} Z} 
& 
{Z \times _{X}T ^{*} X} 
   \ar@{}[dr]|\square
\ar[l] ^-{\rho _u}
\ar[d] ^-{\varpi _u}
& 
{Z \times _{P}T ^{*} P} 
\ar[l] ^-{Z \times _{X}\rho _{f}}
\ar[d] ^-{}
\ar[r] ^-{\varpi _{f}} 
   \ar@{}[dr]|\square
&
{Z \times _{X}T ^* P}
\ar[d] ^-{u} 
\\ 
& 
{T ^{*} X} 
& 
{X \times _{P}T ^{*} P} 
\ar[l] ^-{\rho _f} 
\ar[r] ^-{\varpi _f} 
&
{T ^* P,}
}
\end{equation}
where the upper left square and the composition of both
upper squares are by definition cartesian (for the second case, use 
$\rho _{f \circ u}= \rho _u \circ (Z \times _{X}\rho _{f})$).
This yields the cartesianity of the upper right square. 
Hence, we get the equality 
$(Z \times _{X}\rho _{f} )^{-1} ( T ^{*} _Z X) = T ^{*} _{Z} P$.
When $u$ is an immersion, this yields 
$\mathscr{T} _{f} (T ^{*} _Z X) =\varpi _f (T ^{*} _Z P)$.
The other assertions of 1) are obvious. 

2) Now, let us check part 2) of the Lemma.
Since
$(Z \times _{X}\rho _{f} )^{-1} (Z \times _{X} T ^* _X X ) 
= Z \times _{X} \rho _{f} ^{-1} (T ^* _X X ) = 
 Z \times _{X} T ^{*} _{X} P$ 
and 
$(Z \times _{X}\rho _{f} )^{-1} ( T ^{*} _Z X) = T ^{*} _{Z} P$ (this is the first part of the Lemma),
we reduce to check the inclusion 
$Z \times _{X} T ^* _X X \subset T ^{*} _Z X$ 
(resp. equality $Z \times _{X} T ^* _X X = T ^{*} _Z X$) of subvarieties of $T ^* X$.

i) First, suppose that $u$ is an open immersion. 
In that case $\rho _u \colon Z \times _X T ^* X\to T ^* Z$ is an isomorphism
and we check easily the desired equality
$\rho _u ^{-1} (  T ^* _Z Z ) =  Z \times _X T ^* _X X$ by coming back to the definition of $T ^* _X X$
and $T ^* _Z Z $. 

ii) Suppose now that $u$ is only an immersion. 
Let $\overline{Z}$ be the closure of $Z$ in $X$. 
From the part 2.i) of the proof, the respective case of the part 2) of the Lemma is satisfied. 
Hence, since $Z \to \overline{Z}$ is an open immersion, 
we get $Z \times _{\overline{Z}} T ^* _{\overline{Z}} X = T ^* _{Z} X $.
Using this latter equality, we reduce to check the inclusion 
$\overline{Z} \times _{X} T ^* _X X \subset  T ^{*} _{\overline{Z}} X$.
In other words, we can suppose $Z = \overline{Z}$.
Moreover, from the part 1) and the respective case of the part 2) of the Lemma, the check is local in $X$.
Hence, we can suppose that
$X$ has local coordinates $t _1, \dots, t _d$ such that 
$\overline{t} _1,\dots, \overline{t} _r$, the global section of $\O _Z$ induced by 
$t _1, \dots , t _r$, are local coordinates of $Z$.
We get local coordinates
$t _1 , \dots, t _d, \xi _1, \dots, \xi _d$ of $T ^{*} X$, where $\xi _i$ is the element associated with
$\partial _i$, the derivation with respect to $t _i$. 
We get also local coordinates
$\overline{t} _1,\dots, \overline{t} _r, \overline{\xi} _1,\dots, \overline{\xi} _r$ of $T ^* Z$,
where $ \overline{\xi} _i$ is
the element associated with
$\overline{\partial} _i$, the derivation with respect to $\overline{t} _i$. 
Then, $\rho _u \colon Z \otimes _X T ^* X \to T ^* Z$ is smooth 
and $1 \otimes \xi _{r+1}, \dots, 1 \otimes  \xi _d$ are local coordinates relatively to $\rho _u$
(in fact they induce an isomorphism of the form
$Z \otimes _X T ^* X \riso \A ^{d-r} _{T ^* Z}$)
and the image of 
$\overline{\xi} _1\dots, \overline{\xi} _r$ via $\rho _u$ 
are $1 \otimes \xi _{1}, \dots, 1 \otimes  \xi _r$.
Hence we get
$\rho _u ^{-1} (T ^* _Z Z)= T ^* _Z X = V ( 1 \otimes \xi _{1}, \dots, 1 \otimes  \xi _r)$
and 
$Z \times _{X} T ^* _X X =  V ( 1 \otimes \xi _{1}, \dots, 1 \otimes  \xi _d)$ as subvarieties of 
$Z \times _{X} T ^* X$.
Hence, 
$Z \times _{X} T ^* X \subset T ^* _Z X $.
\end{proof}

\begin{lem}
\label{prop-incluT*strat}
Let $f\colon X \to P$ be a morphism of smooth $k$-varieties
and 
$(X_i)_{1\leq i\leq r}$ be a family of smooth subvarieties of $X$ (resp. open subvarieties of $X$) such that 
$X \subset \cup _{i=1} ^{r} X _i$.
Then, $T ^{*} _{X} P \subset \cup _{i=1} ^{r} T ^{*} _{X _i} P$ 
(resp. $T ^{*} _{X} P = \cup _{i=1} ^{r} T ^{*} _{X _i} P$).
\end{lem}

\begin{proof}
From the first equality of Lemma \ref{lem-incluT*strat}, we reduce to the case $X=P$. 
Using the second part of Lemma \ref{lem-incluT*strat},
we get the inclusions
$X _i \times _{X} T ^* _X X \subset T ^{*} _{X _i} X$ (resp. the equalities $X _i \times _{X} T ^* _X X = T ^{*} _{X _i} X$), 
which yields the desired result when $X=P$.
\end{proof}

\begin{prop}
\label{lem-stab-T_X X}
Let $f \colon X \to Y$ be a smooth morphism of smooth varieties. 
Let $B$ be a smooth subvariety of $Y$ and 
$A:= f ^{-1} (B)$.
\begin{enumerate}

\item The morphism $\rho _{f}$ is a closed immersion. 
If $f$ is étale then $\rho _{f}$ is an isomorphism,

\item 
\label{stab-T_X X2} 

We have
$\rho _f ^{-1} (T ^* _A X) 
\overset{\ref{lem-incluT*strat}.1}{=}
T ^* _A Y
=
\varpi _{f} ^{-1} (T ^{*} _{B} Y)$, 
$\rho _f (\varpi _{f} ^{-1} (T ^{*} _{B} Y))
\subset 
T ^* _A X$
 and
$\mathscr{T} _{f}  ( T ^* _A X)  
\subset T ^{*} _{B} Y$.

\item When $f$ is surjective, we have the equality
$\mathscr{T} _{f} (T ^* _A X)
=
 T ^{*} _{B} Y$.
\end{enumerate}

\end{prop}

\begin{proof}
1) From \cite[17.11.1]{EGAIV4},
the canonical morphism 
$f ^* \Omega _{Y} \to \Omega _{X}$ is injective. 
By duality we get the surjection 
$\mathcal{T} _X \to f ^* \mathcal{T} _Y$.
Hence, from \cite[1.7.11.(iv) and (v)]{EGAII},
the morphism $\rho _{f}\colon X \times _{Y} T ^* Y \to T ^* X$ is a closed immersion.
When $f$ is etale, 
from \cite[17.11.2]{EGAIV4}, 
the canonical morphism 
$f ^* \Omega _{Y} \to \Omega _{X}$ is an isomorphism and 
then so is $\rho _f$.

2) a) 
In this step, we check the equality
$T ^* _X Y
=
\varpi _{f} ^{-1} (T ^{*} _{Y} Y)$
($= X \times _{Y}T ^{*} _{Y} Y$).

i) From the respective case of \ref{lem-incluT*strat}.2, this is local in $X$. 
Moreover, this is local in $Y$. Indeed, 
let $V$ be an open set of $Y$. We put $U:= f ^{-1} (V)$
and $f _V\colon U \to V$ the induced morphism. 
On one hand 
$(U \times _{Y} T ^* Y )\cap \varpi _{f} ^{-1} (T ^{*} _{Y} Y)
=
U \times _{Y} T ^* _Y Y 
=
U \times _{V} (V \times _Y T ^* _Y Y) 
\overset{\ref{lem-incluT*strat}.2}{=}
U \times _{V}  T ^* _V Y 
\overset{\ref{lem-incluT*strat}.1}{=}
U \times _{V}  T ^* _V V 
=
\varpi _{f _V} ^{-1} (T ^{*} _{V} V)$
and on the other hand
$(U \times _{Y} T ^* Y )\cap
T ^* _X Y
=
U \times _{X} T ^* _X Y
\overset{\ref{lem-incluT*strat}.2}{=}
T ^* _U Y
\overset{\ref{lem-incluT*strat}.1}{=}
T ^* _U V$, which give the localness in $Y$.

ii) Let $g \colon Y \to Z$ be another smooth morphism of varieties.
If this equality is satisfied for $f$ and $g$, i.e. 
if $T ^* _X Y
=
\varpi _{f} ^{-1} (T ^{*} _{Y} Y)$ 
and 
$T ^* _Y Z
=
\varpi _{g} ^{-1} (T ^{*} _{Z} Z)$,
then we get the squares of the following diagram : 
\begin{equation}
\notag
\xymatrix{
{T ^{*} _{X} X} 
\ar@{^{(}->}[d] ^-{}
   \ar@{}[dr]|\square
& 
{X \times _{Y}T ^{*} _Y Y} 
\ar[l] ^-{}
\ar@{^{(}->}[d] ^-{}
   \ar@{}[dr]|\square
& 
{X \times _{Z}T ^{*} _Z Z} 
\ar[l] ^-{}
\ar@{^{(}->}[d] ^-{}
\\
{T ^{*} X} 
& 
{X \times _{Y}T ^{*} Y} 
\ar[l] _-{\rho _g}
& 
{X \times _{Z}T ^{*} Z} 
\ar[l] _-{\rho _f}
\ar@/^0,5cm/[ll] ^-{\rho _{g\circ f}}
}
\end{equation}
are cartesian.
Hence, the rectangle is also cartesian, 
i.e. $T ^* _X Z
= X \times _{Z}T ^{*} _{Z} Z$, 
which is the desired equality.

iii) Since from the step i) the check of equality 
$T ^* _X Y
=
\varpi _{f} ^{-1} (T ^{*} _{Y} Y)
= X \times _{Y}T ^{*} _{Y} Y$ 
is local in $X$ and $Y$,
then we can suppose that $Y$ has local coordinates $t _1,\dots, t _d$ 
and that there exists an etale morphism of the form
$X \to \A ^{n} _Y$ whose composition with the projection $\A ^{n} _Y \to Y$ gives $f$. 
Since from the step ii) the equality 
$T ^* _X Y
=
\varpi _{f} ^{-1} (T ^{*} _{Y} Y)
= X \times _{Y}T ^{*} _{Y} Y$ is transitive with respect to the composition,
we reduce to the case where $n=0$ or $X = \A ^{n} _Y $.
Suppose $n=0$, i.e. $f$ is etale. 
Let 
$t '_1,\dots, t '_d$ 
be the local coordinates of $X$ induced by
$t _1 , \dots, t _d$. 
We get local coordinates
$t _1 , \dots, t _d, \xi _1, \dots, \xi _d$ of $T ^{*} Y$, where $\xi _i$ is the element associated with
the derivation with respect to $t _i$
and 
local coordinates
$t '_1 , \dots, t '_d, \xi '_1, \dots, \xi '_d$ of $T ^{*} X$, where $\xi '_i$ is the element associated with
the derivation with respect to $t '_i$. 
The isomorphism 
$\rho _f \colon X \times _Y T ^* Y \riso T  ^* X$ 
sends $\xi '_i$ to $1 \otimes \xi _i$. 
Since $T _X ^* X = V (\xi ' _1 ,\dots , \xi ' _d)$
then $T _X ^* Y:= \rho _f ^{-1}(T _X ^* X) =
V (1 \otimes \xi  _1 ,\dots , 1 \otimes \xi  _d)
= X \times _Y V (\xi  _1 ,\dots , \xi  _d)
=X \times _Y T _Y ^* Y$.
Suppose now that $X = \A ^{n} _Y $.
Let $t _{d+1},\dots, t _{d+n}$ be the coordinates of $\A ^n _k$ 
and let $\xi _{d+1},\dots, \xi _{d+n}$ be the element of $T ^* \A ^n _k$ associated  respectively 
with the derivation with respect to $t _{d+1},\dots, t _{d+n}$. 
We get local coordinates
$t _1 , \dots, t _{d+n}, \xi _1, \dots, \xi _{d+n}$ of $T ^{*} X$.
In that case the morphism $\rho _f$ is a closed immersion of the form
$\rho _f \colon \A ^n \times T ^* Y \hookrightarrow T ^* X$ so that
$\A ^n \times T ^* Y = V (\xi _{d+1},\dots, \xi _{d+n})$ in $T ^* X$.
Since $T ^* _X X = V (\xi _{1},\dots, \xi _{d+n})$ in $T ^* X$
then $\rho ^{-1} _f(T ^* _X X)$ is the closed subvariety of 
$\A ^n \times T ^* Y$ defined by $\xi _1 = 0,\dots, \xi _d =0$, 
i.e.  $\rho ^{-1} _f(T ^* _X X)= \A ^n \times T ^* _Y Y $
(recall $T ^* _Y Y  = V (\xi _{1},\dots, \xi _{d})$ in $T ^* Y$),
which is the desired equality.

b) Let $i\colon B \hookrightarrow Y$ be the 
structural immersion. By applying 
\ref{lem-incluT*strat}.1 to the case of $A \to B \to Y$, 
we get $(A \times _{B} \rho _i ) ^{-1} (T ^{*} _A B) =T ^{*} _A Y$.
Moreover, $(A \times _{B} \rho _i ) ^{-1} (A \times _{B} T ^{*} _B B )  = A \times _{B} T ^{*} _B Y$. 
From part a), we have $T ^{*} _A B = A \times _{B} T ^{*} _B B $.
This implies the first equality 
$T ^{*} _A Y =A \times _{B} T ^{*} _B Y = X \times _{Y} T ^{*} _B Y =\varpi _{f} ^{-1} (T ^{*} _{B} Y)$.
Hence, we have checked the equality
$\rho _f ^{-1} (T ^* _A X) 
=
\varpi _{f} ^{-1} (T ^{*} _{B} Y)$.
By applying $\rho _f $ to this equality, 
we get 
$\rho _f (\varpi _{f} ^{-1} (T ^{*} _{B} Y))
=
\rho _f ( \rho _f ^{-1} (T ^* _A X) )
\subset 
T ^* _A X$.
By applying this time 
$\varpi _f$ to this equality, we get
$\mathscr{T} _{f}  ( T ^* _A X)  =
\varpi _{f} 
( \rho _f ^{-1} (T ^* _A X) )
=
\varpi _{f} (\varpi _{f} ^{-1} (T ^{*} _{B} Y))
\subset T ^{*} _{B} Y$.
When $f$ is surjective, 
then so is $\varpi _f$ (see \cite[3.5.2.(ii)]{EGAI}).
Hence, the latter inclusion is in fact an equality.

\end{proof}

\subsection{Inverse and direct images of complexes of arithmetic $\D$-modules and characteristic varieties}

\begin{empt}
[Characteristic variety and characteristic cycle (of level $0$)]
Let $\X$ be a smooth $\V$-formal scheme, $X $ be the reduction of $\X$ modulo $\pi $ (recall $\pi$ is a uniformizer of $\V$).
Let $m\in \N$ be an integer. 
Let us recall Berthelot's definition of characteristic varieties (of level $m$) as explained in \cite[5.2]{Beintro2}. 
\begin{enumerate}
\item Let $\G$ be a coherent 
$\D _{X} ^{(0)}$-module. 
Berthelot checked the equality
$T ^*  X := \Spec \gr \D _{X} ^{(0)} $, 
where $\D _{X} ^{(0)}$ is filtered by the order.
Since $\pi _X \colon T ^* X \to X$ is affine, 
the functor $\pi _{X*}$ induces an equivalence from the category
of (quasi-)coherent $\O _{T ^* X}$-modules to that 
of (quasi-)coherent $\gr \D _{X} ^{(0)} $-modules.
We denote by $\sim$ a quasi-inverse functor. 
Even if it is tempting to identify both categories,
we will try to distingue them to avoid confusion.  
Choose a good filtration $(\G _n) _{n\in \N}$,
i.e. a filtration such that $\gr \G$ is a coherent $\gr \D _{X} ^{(0)}$-module
 (see the definition \cite[5.2.3]{Beintro2}).
 We denote by $\widetilde{\gr}$ the composition of $\gr$ with $\sim$.
So, $\widetilde{\gr} \G$ is a coherent $\O _{T ^* X}$-module.
The characteristic variety of level $0$ of $\G$, denoted by 
$\mathrm{Car} ^{(0)} (\G)$ is by definition the support of 
$\widetilde{\gr}  \G$ in $T ^{*} X $ which is viewed  canonically as a  subvariety of $T ^{*} X $.
Berthelot checked that this is well defined (i.e. that this is independent of the choice of the good filtration). 
Moreover, he defined the characteristic cycle associated with $\G$ 
that we will denote (we add $(0)$ to avoid confusion) by
$Z\mathrm{Car} ^{(0)} (\G)$ (for a detailed definition see 
\cite[5.4.1]{Beintro2}).

\item Let  $\FF$ be a coherent 
$\widehat{\D} _{\X} ^{(0)}$-module. 
The characteristic variety $\mathrm{Car} ^{(0)} (\FF)$ of level $0$ of $\FF$ 
 is by definition the 
characteristic variety of level $0$ of $\FF/\pi \FF$
as coherent 
$\D _{X} ^{(0)}$-module, i.e. 
$\mathrm{Car} ^{(0)} (\FF):= \mathrm{Car} ^{(0)} (\FF/\pi \FF)$.
Similarly, we define the
characteristic cycle $Z\mathrm{Car} ^{(0)} (\FF)$ of level $0$ of $\FF$ 
by setting 
$Z\mathrm{Car} ^{(0)} (\FF):= Z\mathrm{Car} ^{(0)} (\FF/\pi \FF)$.

\item Let $\E$ be a coherent 
$\widehat{\D} _{\X,\Q} ^{(0)}$-module. 
Choose a  coherent 
$\widehat{\D} _{\X} ^{(0)}$-module $\overset{\circ}{\E}$ without $p$-torsion such
that there exists an isomorphism of $\widehat{\D} _{\X,\Q} ^{(0)}$-modules of the form
$\overset{\circ}{\E} _\Q \riso \E$.
The characteristic variety of level $0$ of $\E$ denoted by 
$\mathrm{Car} ^{(0)} (\E)$ is by definition that 
of $\overset{\circ}{\E}$ as coherent $\widehat{\D} _{\X} ^{(0)}$-module, i.e., 
 $\mathrm{Car} ^{(0)} (\E):= \mathrm{Car} ^{(0)} (\overset{\circ}{\E}/\pi \overset{\circ}{\E})$.
 Berthelot checked that this is well defined. 
Similarly, we define the
characteristic cycle $Z\mathrm{Car} ^{(0)} (\E)$ of level $0$ of $\E$ 
by setting 
 $Z\mathrm{Car} ^{(0)} (\E):= Z\mathrm{Car} ^{(0)} (\overset{\circ}{\E}/\pi \overset{\circ}{\E})$.

\item Let $(\NN,\phi)$ be a coherent $F\text{-}\D ^{\dag} _{\X,\Q}$-module, 
i.e. a coherent $\D ^{\dag} _{\X,\Q}$-module $\NN$ and an isomorphism of $\D ^{\dag} _{\X,\Q}$-modules $\phi$  of the form
$\phi \colon F ^* \NN \riso \NN$.
Then there exists a (unique up to isomorphism) coherent $\widehat{\D} _{\X,\Q} ^{(0)}$-module
$\NN ^{(0)}$ and an isomorphism 
$\phi ^{(0)}\colon \widehat{\D} _{\X,\Q} ^{(s)} \otimes _{\widehat{\D} _{\X,\Q} ^{(0)}}\NN ^{(0)} 
\riso F ^* \NN ^{(0)}$
which induces canonically $\phi$.
Then, 
the characteristic variety of $\NN$ denoted by 
$\mathrm{Car} (\NN)$
is by definition
the  characteristic variety of level $0$ of $\NN ^{(0)}$,
i.e., 
$\mathrm{Car} (\NN ):=  \mathrm{Car} ^{(0)} (\NN ^{(0)})$.
Finally, the characteristic cycle of $\NN$ denoted by 
$Z\mathrm{Car} (\NN)$
is by definition
the  characteristic variety of level $0$ of $\NN ^{(0)}$.

\item Let $\overline{\E} \in D ^{\mathrm{b}} _{\mathrm{coh}} (\D ^{(0)} _{X})$
and let 
$(\E ,\phi) \in F\text{-}D ^{\mathrm{b}} _{\mathrm{coh}} (\D ^{\dag} _{\X,\Q})$.
By definition, we define the characteristic variety of these complexes by setting
$\mathrm{Car} ^{(0)} (\overline{\E}) := \cup _r \mathrm{Car} ^{(0)} (\mathcal{H} ^r (\overline{\E}))$
and
$\mathrm{Car}  (\E) := \cup _r \mathrm{Car} ^{(0)} (\mathcal{H} ^r (\E))$.
\end{enumerate}

\end{empt}

\begin{lem}
\label{lem-supp}
Let $u \colon V \to W$ be a morphism of $k$-varieties.

\begin{enumerate}
\item If $u$ is flat then for any $\O _{W}$-module $\NN$
we have $\supp (u ^* \NN)= u ^{-1} (\supp \NN)$.

\item If $u$ is finite, then for any coherent $\O _{V}$-module $\M$ we have
$\supp (u _{*} (\M)) = u ( \supp (\M))$.

\end{enumerate}

\end{lem}

\begin{proof}
Since a flat morphism of local rings is faithfully flat,
we get the first assertion. Let $\M $ be a coherent $\O _{V}$-module. 
From \cite[5.2.2]{EGAI}, $\supp \M$ is a closed subset of $V$.
Since $u$ is closed, by using \cite[3.4.6]{EGAI}, we get the inclusion
$\supp (u _{*} (\M)) \subset u ( \supp (\M))$.
Set 
$W ' := W \setminus \supp (u _* \M) $, 
$V' := u ^{-1} (W')$ and $u ' \colon V ' \to W'$ the morphism induced by $u$.
Since $u$ is finite, then $u _* \M$ is a coherent $\O _W$-module. 
Hence $W'$ is an open subset of $W$.
Since $u ' _{*} ( \M | V') = u _{*} ( \M)  | W' = 0$,
since $u '$ is affine and $\M | V'$ is a quasi-coherent $\O _{V'}$-module, 
this yields $\M | V'=0$. Hence, $\supp \M \subset V \setminus V' = u ^{-1} (\supp (u _* \M) )$, 
which is equivalent to the inclusion 
$
u ( \supp (\M))
\subset 
\supp (u _{*} (\M)) $.

\end{proof}

\begin{prop}
\label{Carf!f_+finet}
Let $f \colon X \to Y$ be an étale morphism of integral smooth $k$-varieties. 
Let 
$\overline{\E} \in D ^{\mathrm{b}} _{\mathrm{coh}} (\D ^{(0)} _{X})$,
$\overline{\FF} \in D ^{\mathrm{b}} _{\mathrm{coh}} (\D ^{(0)} _{Y})$.

\begin{enumerate}
\item We have the equality
$| \mathrm{Car} ^{(0)} ( f ^{! ^{(0)}} (\overline{\FF}) ) | 
=
\rho _f  ( \varpi _f  ^{-1}
(| \mathrm{Car} ^{(0)} (\overline{\FF})  |))$.

\item If $f$ is moreover finite, then
$| \mathrm{Car} ^{(0)} ( f _{+ ^{(0)}} (\overline{\E}) ) | 
=
\varpi _f \circ \rho _f ^{-1} 
(| \mathrm{Car} ^{(0)} (\overline{\E})  |)
=:\mathscr{T} _f  
(| \mathrm{Car} ^{(0)} (\overline{\E})  |)$.

\item If $f$ is moreover finite and surjective of degree $d$
and if $\overline{\FF} $
is a coherent $\D ^{(0)} _{Y}$-module, then 
\begin{gather}
\label{Carf!f_+finet22}
Z \mathrm{Car} ^{(0)} ( f _{+ ^{(0)}} f ^{! ^{(0)}} (\overline{\FF}) ) 
=
d Z\mathrm{Car} ^{(0)} (\overline{\FF}) .
\end{gather}
\end{enumerate}
\end{prop}

\begin{proof}
Since $f$ is etale, $\rho _f$ is an isomorphism
(this is equivalent to say that the canonical morphism
$\gr \, \D _{X} ^{(0)} \to f ^{*} \gr \, \D _{Y} ^{(0)}$ is an isomorphism).
We get the etale morphism of $k$-varieties 
$\mathscr{T} _f   := \varpi _f \circ \rho _f ^{-1}\colon T ^* X \to T ^* Y$ which is 
included in the cartesian square:
\begin{equation}
\notag
\xymatrix{
{T ^* X} 
\ar[r] ^-{\mathscr{T} _f}
  \ar@{}[dr]|\square
  \ar[d] ^-{\pi _X}
& 
{T ^* Y } 
  \ar[d] ^-{\pi _Y}
\\ 
{X} 
\ar[r] ^-{f}
& 
{Y } 
}
\end{equation}
To check the first assertion we can suppose that 
$\overline{\FF}$ is a coherent
$\D ^{(0)} _{Y}$-module.
Let $(\overline{\FF} _n)$ be a good filtration of $\overline{\FF}$.
Then $( f ^*\overline{\FF} _n)$ be a good filtration of $ f ^{! ^{(0)}} (\overline{\FF}) $ 
(which is equal as $\O _X$-module to $f ^* \overline{\FF}$).
With this filtration, we check
$\widetilde{gr} (f ^{! ^{(0)}} (\overline{\FF}) )
\riso 
\mathscr{T} _f ^* 
(\widetilde{gr} (\overline{\FF}))$
(remark that $\pi _{T ^* X, *} \circ \mathscr{T} _f ^* \circ \sim $ is isomorphic to 
$\gr \, \D _{X} ^{(0)}  
\otimes  _{f ^{-1} \gr \, \D _{Y} ^{(0)} } f ^{-1} (-)$
as functor from the category of 
quasi-coherent 
$\gr \, \D _{Y} ^{(0)} $-modules
to that of 
quasi-coherent $\gr \, \D _{X} ^{(0)}$-modules).
From Lemma \ref{lem-supp}.1,
since $\mathscr{T} _f$ is flat we get
$\supp \mathscr{T} _f ^* 
(\widetilde{gr} (\overline{\FF}))
=
\mathscr{T} _f ^{-1} 
(\supp (\widetilde{gr} (\overline{\FF})))$.
Since $ \mathscr{T} _f ^{-1} (\supp (\widetilde{gr} (\overline{\FF})))=\rho _f ( \varpi _f  ^{-1}(\supp (\widetilde{gr} (\overline{\FF}))))$, 
then we obtain the first equality of the proposition.

Suppose now that $f$ is finite and etale. 
To check the second assertion we can suppose that 
$\overline{\E}$ is a coherent
$\D ^{(0)} _{X}$-module.
Let $(\overline{\E} _n)$ be a good filtration of $\overline{\E}$.
Then $( f _*\overline{\E} _n)$ be a good filtration of $ f _{+ ^{(0)}} (\overline{\E}) $ 
(which is isomorphic to $f _{*} (\overline{\E}) $ as $\O _Y$-module.
With this filtration, we check
$\widetilde{gr} (f _{+ ^{(0)}} (\overline{\E}) )
\riso 
\mathscr{T} _{f *} 
(\widetilde{gr} (\overline{\E}))$.
From Lemma \ref{lem-supp}.2,
since $\mathscr{T} _{f } $ is finite, 
we get 
$\supp (\mathscr{T} _{f *} 
(\widetilde{gr} (\overline{\E})))
=
\mathscr{T} _{f } (\supp 
(\widetilde{gr} (\overline{\E})))$.
Hence, we get the second equality of the proposition.

Suppose now that $f$ is finite and etale and surjective of degree $d$.
From the first and the second equality of the proposition, we get
$|\mathrm{Car} ^{(0)} ( f _{+ ^{(0)}} f ^{! ^{(0)}} (\overline{\FF}) )| 
=
\mathscr{T} _f  (\mathscr{T} _f  ^{-1}(|\mathrm{Car} ^{(0)} (\overline{\FF}) |) )$.
Since $f$ is surjective, then so is $\mathscr{T} _f $.
Hence 
$|\mathrm{Car} ^{(0)} ( f _{+ ^{(0)}} f ^{! ^{(0)}} (\overline{\FF}) )| 
=
|\mathrm{Car} ^{(0)} (\overline{\FF}) | $.
By unicity of the structure of reduced subscheme of $T ^* Y$ attached to a closed subspace (of $T ^* Y$), we have in fact the equality 
$\mathrm{Car} ^{(0)} ( f _{+ ^{(0)}} f ^{! ^{(0)}} (\overline{\FF}) ) 
=
\mathrm{Car} ^{(0)} (\overline{\FF}) $
as subvariety.
It remains to compute the multiplicity (see Berthelot's definition of characteristic cycles of \cite[5.4.1]{Beintro2}).
Let $(\overline{\FF} _n)$ be a good filtration of $\overline{\FF}$.
From what we have already checked above in the proof,
we have the good filtration $( f _* f ^*\overline{\FF} _n)$ of $f _{+ ^{(0)}}  f ^{! ^{(0)}} (\overline{\FF}) $ 
and with this filtration, 
$\widetilde{gr} (f _{+ ^{(0)}}  f ^{! ^{(0)}} (\overline{\FF}) )
\riso 
\mathscr{T} _{f *} \circ   \mathscr{T} _f ^* 
(\widetilde{gr} (\overline{\FF}))$.
Then, we obtain the desired computation using 
Lemma \cite[A.1.3]{Fulton-Intersection} and the following fact :
if $\phi \colon A \to B$ is an étale morphism, $\mathfrak{Q}$ a prime ideal of $B$, 
$\mathfrak{P}:= \phi ^{-1} (\mathfrak{Q})$, 
if $M$ is an $A _{\mathfrak{P}}$-module then $B _{\mathfrak{Q}} \otimes _{A _{\mathfrak{P}}} M$ has the same length as $B _{\mathfrak{Q}}$-module than $M$ as $A _{\mathfrak{P}}$-module.
\end{proof}

Every results of \cite[2]{Laumon-TransfCanon} concerning 
the extraordinary inverse and direct images of complexes of filtered arithmetic $\D$-modules are still valid for arithmetic $\D$-modules at level $0$ without new arguments for the check.
For the reader convenience, via the following two propositions we translate in the context of arithmetic $\D$-modules of level $0$ the corollaries 
\cite[2.5.1 and 2.5.2]{Laumon-TransfCanon}  that we will need below to check \ref{dag-carf+redu} which will be an ingredient
of the proof of 
\ref{lem-loc-acyc}.

\begin{prop}
[Laumon]
Let $f \colon X \to Y$ be a morphism of smooth $k$-varieties. 
For any 
$\overline{\FF} \in D ^{\mathrm{b}} _{\mathrm{coh}} (\D ^{(0)} _{Y})$ such that 
the restriction $\rho _f | \varpi _{f} ^{-1}( | \mathrm{Car} ^{(0)} (\overline{\FF}) |) $ is proper, we have 
$f ^{! ^{(0)}} (\overline{\FF}) \in  D ^{\mathrm{b}} _{\mathrm{coh}} (\D ^{(0)} _{X})$ and 
$| \mathrm{Car} ^{(0)} (f ^{! ^{(0)}}\overline{\FF} ) | \subset \rho _{f} ( \varpi _{f} ^{-1} 
( | \mathrm{Car} ^{(0)} (\overline{\FF}) |))$.
\end{prop}

\begin{prop}
[Laumon]
\label{carf+redu}
Let $f \colon X \to Y$ be a morphism of smooth $k$-varieties. 
For any 
$\overline{\E} \in D ^{\mathrm{b}} _{\mathrm{coh}} (\D ^{(0)} _{X})$ such that 
the restriction $\varpi _f | \rho _{f} ^{-1}( | \mathrm{Car} ^{(0)} (\overline{\E}) |) $ is proper, we have 
$f _{+ ^{(0)}} (\overline{\E}) \in  D ^{\mathrm{b}} _{\mathrm{coh}} (\D ^{(0)} _{X})$ and 
$| \mathrm{Car} ^{(0)} ( f _{+ ^{(0)}}\overline{\E} ) | \subset \varpi _{f} ( \rho _{f} ^{-1} 
( | \mathrm{Car} ^{(0)} (\overline{\E}) |))=:\mathscr{T} _f ( | \mathrm{Car} ^{(0)} (\overline{\E}) |)$.
\end{prop}

\begin{empt}
\label{nota-formalCar-finite-etal}
Let $f\colon \PP ' \to \PP$ be a proper morphism of smooth formal $\V$-schemes. 
Let $(\E ',\phi)$ be a coherent $F\text{-}\D ^{\dag} _{\PP',\Q}$-module.
From the equivalence of categories of \cite[4.5.4]{Be2}, there exist (unique up to isomorphism) 
a coherent $\widehat{\D} ^{(0)} _{\PP ',\Q}$-module 
$\FF ^{\prime (0)}$ and an isomorphism 
$\phi ^{(0)}\colon \widehat{\D} ^{(s)} _{\PP ',\Q} \otimes _{\widehat{\D} ^{(0)} _{\PP ',\Q}} 
\FF ^{\prime (0)}
\riso 
F ^{*}\FF ^{\prime (0)}$
which induced  $(\E ',\phi)$ by extension.
Fix an integer $i \in \Z$. From the isomorphisms
\begin{gather}
\widehat{\D} ^{(s)} _{\PP,\Q} \otimes _{\widehat{\D} ^{(0)} _{\PP,\Q}}  
\H ^{i} f  _{+ ^{(0)}} (\FF ^{\prime (0)})
\underset{\cite[3.5.3.1]{Beintro2}}{\riso} 
\H ^{i} f  _{+ ^{(s)}}
(\widehat{\D} ^{(s)} _{\PP ',\Q} \otimes _{\widehat{\D} ^{(0)} _{\PP ',\Q}} 
\FF ^{\prime (0)})
\\
\underset{\phi ^{(0)}}{\riso} 
\H ^{i} f  _{+ ^{(s)}}
(F ^{*}\FF ^{\prime (0)})
\underset{\cite[3.5.4.1]{Beintro2}}{\riso} 
F ^{*} \H ^{i} f  _{+ ^{(0)}}
(\FF ^{\prime (0)}),
\end{gather}
we get 
$\mathrm{Car} (\H ^{i} f _{+} (\E'))
=
\mathrm{Car} ^{(0)} (\H ^{i} f  _{+ ^{(0)}} (\FF ^{\prime (0)}))$.
Choose a coherent $\widehat{\D} ^{(0)} _{\PP '}$-module without $p$-torsion 
$\E ^{\prime (0)}$ such that 
$\E ^{\prime (0)} _\Q \riso \FF ^{\prime (0)}$.
Since 
$\H ^{i} f  _{+ ^{(0)}} (\FF ^{\prime (0)})
\riso 
(\H ^{i} f  _{+ ^{(0)}} (\E ^{\prime (0)})) _\Q$, 
then 
by putting 
$\G ^{(0)}$ as equal to the quotient of $\H ^{i} f  _{+ ^{(0)}} (\E ^{\prime (0)})$ by its $p$-torsion part, 
$\mathrm{Car} ^{(0)} (\H ^{i} f  _{+ ^{(0)}} (\FF ^{\prime (0)}))
:=
\mathrm{Car} ^{(0)} (\G ^{(0)})
\subset 
\mathrm{Car} ^{(0)} (\H ^{i} f  _{+ ^{(0)}} (\E ^{\prime (0)}))$ (for the equality, see the definition \cite[5.2.5]{Beintro2}).
Hence, 
$$|\mathrm{Car} (\H ^{i} f _{+} (\E'))|
\subset 
|\mathrm{Car} ^{(0)} (\H ^{i} f  _{+ ^{(0)}} (\E ^{\prime (0)}))|
:=
|\mathrm{Car} ^{(0)} (k \otimes _{\V} \H ^{i} f _{+ ^{(0)}} (\E ^{\prime (0)}))|.$$

By using a spectral sequence (the result is given in the beginning of the proof of \cite[I.5.8]{virrion}), 
we obtain the monomorphism
$k \otimes _{\V} \H ^{i} f _{+ ^{(0)}} (\E ^{\prime (0)})
\hookrightarrow 
\H ^{i} (k \otimes ^{\L}_{\V} f _{+ ^{(0)}} (\E ^{\prime (0)}))$.
Hence, 
$|\mathrm{Car} ^{(0)}(k \otimes _{\V} \H ^{i} f _{+ ^{(0)}} (\E ^{\prime (0)}))|
\subset
|\mathrm{Car} ^{(0)}(  \H ^{i} (k \otimes ^{\L}_{\V} f _{+ ^{(0)}} (\E ^{\prime (0)})) )|$.
We have 
$k \otimes ^{\L} _{\V} f _{+ ^{(0)}} (\E ^{\prime (0)})  
\riso 
f  _{+ ^{(0)}} (\overline{\E} ^{\prime (0)})$,
where  $\overline{\E} ^{\prime (0)}: = k \otimes _{\V} \E ^{\prime (0)}
\riso
k \otimes ^{\L} _{\V} \E ^{\prime (0)}$.
Finally we get
\begin{equation}
\label{0todagf_+}
|\mathrm{Car} (\H ^{i} f _{+ } (\E '))|
\subset 
|\mathrm{Car} ^{(0)} ( \H ^{i} f _{+ ^{(0)}} (\overline{\E} ^{\prime (0)}))|.
\end{equation}
From \ref{carf+redu}, since $f$ is proper then 
$|\mathrm{Car} ^{(0)} ( \H ^{i} f _{+ ^{(0)}} (\overline{\E} ^{\prime (0)}))|
\subset
\mathscr{T} _{f}
(|\mathrm{Car} ^{(0)} (\overline{\E} ^{\prime (0)})|)$.
By Berthelot's definition of the characteristic variety of $\E'$, we have
$\mathrm{Car} (\E ')
=
\mathrm{Car} ^{(0)} (\overline{\E} ^{\prime (0)})$.
Hence, by using the  spectral sequence 
$E _2 ^{r,s} =\mathcal{H} ^{r} f _+ ( \mathcal{H} ^s (\E) )
\Rightarrow 
\mathcal{H} ^{n} f _+ ( \E ) $
and the beginning of the remark \cite[3.7]{Caro-Lagrangianity}
we check the following proposition. 
\end{empt}

\begin{prop}\label{dag-carf+redu}
Let $f\colon \PP ' \to \PP$ be a proper morphism of smooth formal $\V$-schemes. 
Let $(\E ',\phi) \in F\text{-}D ^{\mathrm{b}} _{\mathrm{coh}} (\D ^{\dag} _{\PP',\Q})$.
We have the inclusion
$$|\mathrm{Car} ( f _{+ } (\E '))|
\subset 
\mathscr{T} _{f}
(|\mathrm{Car} (\E ')|).$$
\end{prop}

\begin{empt}
\label{Carf!+finetformal}
Let $f\colon \PP ' \to \PP$ be a finite étale surjective morphism of smooth formal $\V$-schemes. 

\begin{enumerate}
\item Let $(\E ',\phi)$ be a coherent $F\text{-}\D ^{\dag} _{\PP',\Q}$-module.
With the notation \ref{nota-formalCar-finite-etal},
since $f _{+ ^{(0)}}= f _*$ (and then preserves the property of $p$-torsion freeness) and $f _+ = f  _*$, the inclusion 
\ref{0todagf_+} is an equality. In fact, with the usual notation of characteristic cycles (see \cite[5.4]{Beintro2}),
we get the equality 
\begin{equation}
\label{1281}
Z \mathrm{Car} (f _{+ } (\E '))
=
Z\mathrm{Car} ^{(0)} ( f _{+ ^{(0)}} (\overline{\E} ^{\prime (0)})).
\end{equation}

\item Moreover, 
let $(\E ,\phi)$ be a coherent $F\text{-}\D ^{\dag} _{\PP,\Q}$-module.
From the equivalence of categories of \cite[4.5.4]{Be2}, there exist (unique up to isomorphism) 
a coherent $\widehat{\D} ^{(0)} _{\PP ,\Q}$-module 
$\FF ^{(0)}$ and an isomorphism 
$\phi ^{(0)}\colon \widehat{\D} ^{(s)} _{\PP ,\Q} \otimes _{\widehat{\D} ^{(0)} _{\PP ,\Q}} 
\FF ^{(0)}
\riso 
F ^{*}\FF ^{(0)}$
which induces  $(\E ,\phi)$ by extension.
Choose a coherent $\widehat{\D} ^{(0)} _{\PP}$-module without $p$-torsion 
$\E ^{ (0)}$ such that 
$\FF ^{ (0)} \riso \E ^{ (0)} _\Q$.
Since $\D ^{\dag} _{\PP',\Q}  = f ^*\D ^{\dag} _{\PP,\Q}= 
f ^{! } \D ^{\dag} _{\PP,\Q}$
and 
$\widehat{\D} ^{(0)} _{\PP'} = f ^* \widehat{\D} ^{(0)} _{\PP}= 
f ^{! ^{(0)}} \widehat{\D} ^{(0)} _{\PP}$, 
we check that 
$f ^{! ^{(0)}} \E ^{(0)}= f ^{*} \E ^{(0)}$ has no $p$-torsion and 
that 
$ f ^{!} (\E )\riso 
\D ^{\dag} _{\PP',\Q}
\otimes _{\widehat{\D} ^{(0)} _{\PP '}} 
f ^{! ^{(0)}} \E ^{(0)}$. Moreover, 
putting 
$\overline{\E} ^{ (0)}: = k \otimes _{\V} \E ^{(0)}$, 
we get
\begin{equation}
\label{0todagf^!}
Z \mathrm{Car} ( f ^{!} (\E ))
=
Z \mathrm{Car} (f ^{! ^{(0)}} \overline{\E} ^{(0)}).
\end{equation}

\end{enumerate}

\end{empt}

\begin{prop}
\label{chi=d-times-chi}
Let $f\colon \PP ' \to \PP$ be a finite étale surjective morphism of degree $d$ of integral smooth formal $\V$-schemes. 
Let $(\E ,\phi)$ be a coherent $F\text{-}\D ^{\dag} _{\PP,\Q}$-module.
Then we get 
\begin{gather}
\label{ZCar-equality}
Z \mathrm{Car} (f _{+} f ^{!} \E) 
=
d
Z \mathrm{Car} (\E); 
\\
\label{chi-fetFormula}
\chi (\PP, f _{+} f ^{!} (\E))
= d\cdot \chi (\PP ,\E ).
\end{gather}

\end{prop}

\begin{proof}
From the equality of characteristic cycles \ref{Carf!f_+finet22} and both equalities of the paragraph \ref{Carf!+finetformal}, we get 
$Z \mathrm{Car} (f _{+} f ^{!} \E) 
=
d
Z \mathrm{Car} (\E)$.
Moreover, the equality \ref{chi-fetFormula} is a consequence of \ref{ZCar-equality} and of Berthelot's index theorem \cite[5.4.4]{Beintro2}.
\end{proof}

We will need at the end of the proof of Theorem \ref{lem-loc-acyc} the following Lemma.
\begin{lem}
\label{lem-O-coh-Car}
Let $\X$ be a smooth $\V$-formal scheme, $X $ be the reduction of $\X$ modulo $\pi $.
\begin{enumerate}
\item Let $\G$ be a coherent $\D ^{(0)} _{X}$-module. 
Choose a good filtration $(\G _n) _{n\in \N}$ of $\G$. 
Then the following assertions are equivalent
\begin{enumerate}
\item $\mathrm{Car} ^{(0)} (\G)  \subset T ^* _X X$. 

\item $\gr \G $ is $\O _X$-coherent (for the $\O _X$-module structure induced by 
$\O _X \hookrightarrow \gr \D ^{(0)} _{X}$).

\item $\G$ is $\O _X$-coherent (for the $\O _X$-module structure induced by 
$\O _X \hookrightarrow \D ^{(0)} _{X}$).
\end{enumerate}

\item Let $(\E ,\phi)$ be a coherent $F\text{-}\D ^{\dag} _{\X,\Q}$-module.
The following assertions are equivalent. 
\begin{enumerate}
\item $\mathrm{Car} (\E) \subset T ^* _X X$. 

\item $\E$ is $\O _{\X,\Q}$-coherent (for the $\O _{\X,\Q}$-module structure induced by 
$\O _{\X,\Q} \hookrightarrow \D ^{\dag} _{\X,\Q}$).

\end{enumerate}

\end{enumerate}
\end{lem}

\begin{proof}
Let us check the first part. 
Let us check that $(a)$ implies $(b)$. 
Since this is local, we can suppose that $X$ affine with  local coordinates
$t _1, \dots, t _d$.
Let $\xi _i$ be the global section of 
$\gr \D ^{(0)} _{X}$ which is the element associated  with
$\partial _i$, the derivation with respect to $t _i$. 
Since the ideal defining the closed immersion $\mathrm{Car} ^{(0)} (\G)\hookrightarrow T ^*  X$ is 
the radical of the annihilator of $\gr \G$,  
the inclusion $\mathrm{Car} ^{(0)} (\G)  \subset T ^* _X X$ implies that 
$\xi _{1} ^N, \dots, \xi _d ^N$ annihilate $\gr \G$ for some integer $N$ large enough. 
Hence, $\gr \G$ is a coherent $\gr \D ^{(0)} _{X} /(\xi _1,\dots, \xi _d) ^{Nd}$-module. 
Since $\gr \D ^{(0)} _{X} /(\xi _1,\dots, \xi _d) ^{Nd}$ is a finite $\O _X$-algebra 
(via the compostion $\O _X \to \gr \D ^{(0)} _{X} \to \gr \D ^{(0)} _{X} /(\xi _1,\dots, \xi _d) ^{Nd}$), we conclude
that $\gr \G$ is $\O _X$-coherent. 
Now, suppose $(b)$ satisfied. Then, by definition of a good filtration, 
this implies that $\G _n = \G$ for $n$ large enough. Hence, $\G$ is $\O _X$-module. 
Finally, suppose $(c)$. Then, the constant filtration $(\G _n=\G) _{n\in \N}$ is a good filtration (it might be more convenient to complete the filtration 
by  $\G _{n} =0$ if $n <0$). 
Then the action of  $\xi _i$ on $\gr \G =\G _0 /\G _{-1} =\G$ is zero (because the action of $\xi _i$ is induced
by maps of the form  $\G _i /\G _{i-1} \to \G _{i+1} /\G _{i}$, which are zero). 
Hence, $\mathrm{Car} ^{(0)} (\G)  \subset T ^* _X X$ (recall that the construction of 
$\mathrm{Car} ^{(0)} (\G) $ does not depend on the choice of the good filtration). 

Now, we deduce the second part from the first one. 
Let $\E ^{(0)}$ be the coherent $\widehat{\D} _{\X,\Q} ^{(0)}$-module
endowed with an isomorphism 
$\phi ^{(0)}\colon \widehat{\D} _{\X,\Q} ^{(s)} \otimes _{\widehat{\D} _{\X,\Q} ^{(0)}}\E ^{(0)} 
\riso F ^* \E ^{(0)}$
which induces canonically $\phi$.
Choose 
a coherent $\widehat{\D} _{\X} ^{(0)}$-module $\smash{\overset{\circ}{\E}} ^{(0)} $without $p$-torsion endowed with 
the isomorphism of the form
$\smash{\overset{\circ}{\E}} ^{(0)} _\Q \riso \E ^{(0)}$.
If $\E$ is $\O _{\X,\Q}$-coherent, then $\E ^{(0)}$ is $\O _{\X,\Q}$-coherent
(see \cite[2.2.14]{caro_courbe-nouveau})
Then, from \cite[3.1.3]{Be0}, we can choose $\smash{\overset{\circ}{\E}} ^{(0)} $ so that 
it is $\O _{\X}$-coherent.
Since $\mathrm{Car} (\smash{\overset{\circ}{\E}} ^{(0)}/\pi \smash{\overset{\circ}{\E}} ^{(0)}) 
= \mathrm{Car} (\E) $ and 
$\smash{\overset{\circ}{\E}} ^{(0)}/\pi \smash{\overset{\circ}{\E}} ^{(0)} $ is $\O _X$-coherent, this yields  from the first part
the inclusion
$\mathrm{Car} (\E) \subset T ^* _X X$.
Conversely, suppose $\mathrm{Car} (\E) \subset T ^* _X X$.
Then, from the first part 
$\smash{\overset{\circ}{\E}} ^{(0)}/\pi \smash{\overset{\circ}{\E}} ^{(0)} $ is $\O _X$-coherent. 
Hence, $\smash{\overset{\circ}{\E}} ^{(0)}$
 is $\O _\X$-coherent, which yields that $\E ^{(0)}$ is $\O _{\X,\Q}$-coherent. 
 With \cite[2.2.14]{caro_courbe-nouveau}, this implies that 
$ \E$ is $\O _{\X,\Q}$-coherent.
\end{proof}

\subsection{Generic smoothness up to Frobenius descent}

We prove below  in this section Proposition \ref{corolem-loc-smooth} which 
states that, up to some Frobenius descent (see Lemma \ref{FrobO-coh desc}),  
a morphism is generically smooth. This will be useful later
in the proof of Theorem \ref{lem-loc-acyc}.

\begin{empt}
[Universal homeomorphism]
\label{univ-homeo}
Let $f\colon X \to Y$ be a morphism of schemes. 

\begin{enumerate}
\item Following Definitions \cite[3.5.4]{EGAI} (and Remark \cite[3.5.11]{EGAI})
or \cite[2.4.2]{EGAIV2}, 
$f$ is by definition a universal homeomorphism (resp. is universally injective) if
for any morphism of schemes $g \colon Y ' \to Y$, the morphism 
$f _{Y'}\colon X \times _{Y} Y' \to Y'$ is a homeomorphism (resp. is injective). 

\item Some authors use the name of ``purely inseparable'' (e.g. \cite[5.3.13]{Liu-livre-02}) or ``radicial'' (e.g. \cite[3.5.4]{EGAI}) 
instead of ``universally injective''. From Definition \cite[3.5.4]{EGAI}, Proposition \cite[3.5.8]{EGAI} and Remark 
\cite[3.5.11]{EGAI}, the following conditions are equivalent : 
\begin{enumerate}[(a)]
\item $f$ is universally injective ; 
\item for any field $K$, the map $X (K) \to Y (K)$ is injective ; 
\item $f$ is injective and for any point $x$ of $X$ the monomorphism of the 
residue fields $k (f(x)) \to k (x)$ induced by $f$ is purely inseparable (some authors say ``radicial'' instead of ``purely inseparable''). 
\end{enumerate}

\item Suppose now that $f\colon X \to Y$ is a morphism of $k$-varieties.
Using Proposition \cite[2.4.5]{EGAIV2}, we check that $f$ is a universal homeomorphism if and only if 
$f$ is finite, surjectif and radicial. 
\end{enumerate}

\end{empt}

\begin{lem}
\label{Frobe-univ-homeo}
Let $X$ be a $k$-variety. Then the relative Frobenius
$F ^s _{X/k} \colon X \to X ^{\sigma}$,
the morphism $F ^s _{k} \colon  X ^{\sigma} \to X$
(induced from $F ^s _{k}$ by base change)
and the absolute Frobenius morphism 
$F ^s _{X/k} \colon X \to X $ (equal to $F ^s _{X} = F ^s _{k} \circ F ^s _{X/k} $)
are universal homeomorphisms. 
\end{lem}

\begin{proof}
From the characterization \ref{univ-homeo}.3, 
$F ^s _{k} \colon  \Spec k \to \Spec k$ is a universal homeomorphism. 
Hence, by stability of this property by base change 
we get that
$F ^s _{k} \colon  X ^{\sigma} \to X$ is a universal homeomorphism.
From Lemma \cite[3.2.25]{Liu-livre-02},
we check  that $F ^s _{X/k}$ is finite. 
Hence, so is by composition $F ^s _{X} $.
Since $F ^s _{X} $ induces the identity on the underlying topological space,
$F ^s _{X} $ is bijective. Moreover, the monomorphism of 
the residue fields $k (x) \to k(x)$ 
induced by $F ^s _{X} $ is the $s$th power of the Frobenius, hence it is radicial. 
From \ref{univ-homeo}.2.(c), this yields that $F ^s _{X} $ is radicial. 
From \ref{univ-homeo}.2.(b), this implies that $F ^s _{X/k}$ is also radicial.
With the characterization \ref{univ-homeo}.3, we get that 
$F ^s _{X/k}$ and $F ^s _{X} $ are universal homeomorphisms.
\end{proof}

\begin{dfn}
Let us clarify some terminology. 
Let $f\colon X \to Y$ be a smooth morphism of schemes.
Let $Z$ be a closed subscheme of $X$. 
We say that $Z$ is ``a strict normal crossing divisor relatively to $Y$'' (via $f$) if
for any point $x \in Z$, 
there exists an open affine set $V$ of $Y$ containing $y:= f (x)$,
there exists an open affine set $U$ of $X$ containing $x$ and included in $f ^{-1} (V)$,
there exists an etale $V$-morphism of the form $U \to \A ^{n}_{V} $ given by 
global sections $t _1,\dots, t _n$ such that 
$Z\cap U = V ( t _1 \cdots t _r)$ for some integer $r$.

Remark : When $Y$ is of the form $\Spec K$, with $K$ a perfect field, 
a ``strict normal crossing divisor of $X$ relatively to $Y$'' is the same than 
``a strict normal crossing divisor of $X$'' (the latter is an absolute notion only depending on $X$),
which might justify the terminology.
\end{dfn}

The following lemma is straightforward. 
\begin{lem}
\label{SNCDlem0}
Let $f\colon X \to Y$ be a smooth morphism of schemes.
Let $Z$ be a closed subscheme of $X$. 
Let $g \colon Y' \to Y$ be morphism of schemes, 
$X ':= X \times _{Y} Y'$, 
$Z ':= Z \times _{Y} Y'$.
If $Z$ is a strict normal crossing divisor of $X$ relatively to $Y$
then 
$Z'$ is a strict normal crossing divisor of $X'$ relatively to $Y'$.
\end{lem}

\begin{lem}
\label{SNCDlem1}
Let $f\colon X \to Y$ be a smooth morphism of smooth $k$-varieties with $Y$ integral.
Let $Z$ be a closed subvariety of $X$. 
Let $\eta$ be the generic point of $Y$, $k(\eta)$ be the function field of $Y$,
$X _\eta:= X \times _{Y}  \Spec k(\eta) $, $Z _\eta:= Z \times _{Y} \Spec k (\eta)$.
If $Z _\eta$ is a strict normal crossing divisor of $X _\eta$ relatively to $ \Spec k (\eta)$ then there exists
a dense open set $V$ of $Y$ such that, setting
$X _V := f ^{-1} (V)$ and $Z _V := Z \cap X _V $, the closed subvariety $Z _V$ of $X _V$ is a strict normal crossing  
of $X _V$ relatively to $V$.
\end{lem}

\begin{proof}
We can suppose $X$ integral.
By definition, there exists a covering 
$U _{1,\eta}, \dots, U _{m, \eta}$ by open affine $k(\eta)$-subvarieties of $X _\eta$ such that 
there exists a $k(\eta)$-morphism 
$ U _{i,\eta}\to \A ^{n} _{\eta}$ given by global section 
$t _{i,1},\dots, t _{i,n}$.
Consider the projective system of open affine dense $k$-subvarieties of $Y$ and remark that 
$\Spec k (\eta)$ is the projective limit of this system.
For any open affine $k$-subvariety $V$ of $Y$, 
put $X _V := f ^{-1} (V)$ and $Z _V := Z \cap X _V $.
By using \cite[8.8.2.(ii)]{EGAIV3}, there exists an open affine $k$-subvariety $V$ of $Y$ such that 
there exist a scheme $U _i$ of finite type over 
$V$ (recall that from \cite[1.6]{EGAIV1} to be of finite type or of finite presentation over a locally noetherian scheme 
is the same) and some $k (\eta)$-isomorphism
$U _i \times _{V} \Spec k (\eta)\riso U _{i,\eta}$ for any $i$.
By using Theorem \cite[8.8.2.(i)]{EGAIV3} and Theorem \cite[8.10.5.(iii)]{EGAIV3}, 
shrinking $V$ if necessary, 
we can suppose that there exists an open immersion 
$U _i \hookrightarrow X _V$ which induces
(via the isomorphism $U _i \times _{V} \Spec k (\eta)\riso U _{i,\eta}$)
the open immersion 
$U _{i,\eta} \hookrightarrow X _\eta$.
 By using Theorem \cite[8.10.5.(vi)]{EGAIV3}, we can suppose that 
$(U _i ) _{i=1,\dots, m}$ is an open covering of $X _V$.
By using Theorem \cite[8.8.2.(i)]{EGAIV3} and Theorem \cite[17.7.8]{EGAIV4}, 
shrinking $V$ is necessary, 
there exists an etale $V$-morphism
of the form
$U _i \to \A ^{n} _{V} $ which induces 
the etale $k(\eta)$-morphism $ U _{i,\eta}\to \A ^{n} _{\eta}$.

\end{proof}

\begin{lem}
\label{SNCDlem2}
Let $L/l$ be an algebraic extension of fields of characteristic $p$
such that $L$ is perfect. 
Let $X$ be a smooth variety over $l$, 
$Z$ be closed subvariety of $X$. 
If $ (Z \times _{\Spec l} \Spec L) _{\red}$
is a strict normal crossing divisor of 
$X \times _{\Spec l} \Spec L$
then there exists a finite extension $l'$ of $l$ included in $L$  
such that 
$ (Z \times _{\Spec l} \Spec l') _{\red}$
is a strict normal crossing divisor of 
$X \times _{\Spec l} \Spec l'$ relatively to $\Spec l'$.

\end{lem}

\begin{proof}
For any finite extension $l'$ of $l$ included in $L$,
set 
$ Z _{(l')}:= Z \times _{\Spec l} \Spec l'$,
$X _{(l')}:= X \times _{\Spec l} \Spec l'$
and
$Z ' _{(L)}:=  (Z \times _{\Spec l} \Spec L) _{\red}$.
By using \cite[8.8.2.(ii)]{EGAIV3},
there exists a finite extension $l'$ of $l$ included in $L$  
such that there exist a $l'$-scheme of finite type $Z' _{(l')}$ 
satisfying 
$Z '_{(L)} \riso Z ' _{(l')} \times _{\Spec (l')} \Spec (L)$.
From \cite[8.7.2]{EGAIV3}
we get that $Z ' _{(l')}$ is reduced.
Using \cite[8.8.2.(i)]{EGAIV3} and \cite[8.10.5.(iv) and (vi)]{EGAIV3},
increasing $l'$ if necessary, 
there exist a surjective closed immersion 
$Z ' _{(l')} \hookrightarrow Z _{(l')}$ inducing by extension
$Z ' _{(L)} \hookrightarrow Z _{(L)}$.
Since $Z ' _{(l')}$ is reduced, we get 
$Z ' _{(l')} :=  Z _{(l')\red}$.
Increasing $l'$ if necessary, proceeding as in the proof of 
\ref{SNCDlem1}, 
we check     
that $Z ' _{(l')}$ is a strict normal crossing divisor of 
$X _{(l')}$ relatively to $\Spec l'$.
\end{proof}

\begin{lem}
\label{lem-loc-smooth}
Let $a \colon X \to P$ be a dominant morphism of smooth integral $k$-varieties. 
Let $Z \hookrightarrow X$ be a proper closed subset.
Then there exist a dense  open subvariety $U$ of $P$, a universal homeomorphism $g \colon U' \to U$ of $k$-varieties 
with $U'$ normal, 
a projective, generically finite and etale $U'$-morphism of the form 
$f \colon \widetilde{V}' \to  (X \times _{P} U') _{\mathrm{red}}$ such that $\widetilde{V}'$ is integral and 
$\widetilde{V}'$ is smooth over $U'$, 
$f ^{-1} (Z \times _{P} U' ) _{\mathrm{red}}$ is the support of a strict normal crossing divisor in $\widetilde{V}'$ relatively to 
$U'$.
\end{lem}

\begin{proof}
1) Let $l$ be the field of fractions of $P$ and 
$\overline{l}$ be an algebraic closure of $l$, 
$L:= \smash{\overline{l}} ^{Gal (\overline{l}/l)}$ the fixed field by $Gal (\overline{l}/l)$.
Following \cite[V.6.11]{Lang-Algebra},  $L$ is perfect (in other words, since $\overline{l}$ is an algebraic closure of $L$, 
 $\overline{l}/L$ is separable) and $L/l$ is purely inseparable.
We put  $X _{(L)}: =X \times _{P} \Spec (L)$,
$Y _{(L)}: =( X _{(L)} ) _{\mathrm{red}}$ and 
$Z _{(L)}: =Z \times _{P} \Spec (L)$.

Using Theorems \cite[8.4.1]{EGAIV3} and \cite[8.10.5.(v)]{EGAIV3}, we get the 
$X \times _{P} \Spec (l)$ is irreducible and separated
(we use these Theorems in the following context: consider the projective system $( U _S) _S$ of open affine dense subvarieties of $P$
and next consider the projective system $( a ^{-1}(U _S)) _S$ of open integral subvarieties of $X$ whose projective limit is
$X \times _{P} \Spec (l)$). 
Since $X _{(L)} \to X \times _{P} \Spec (l)$ is a universal homeomorphism, 
we get that $X _{(L)}$ is also irreducible and separated. Hence, $Y _{(L)}$ is an integral $L$-variety, with $L$ a perfect field.
From the desingularisation de Jong's theorem (see \cite{dejong} or \cite[4.1]{Ber-alterationdejong}), 
this implies that
there exists a projective, generically finite and etale morphism
$\phi _L \colon Y '_{(L)} \to Y _{(L)}$ such that $Y '_{(L)}$ is integral, smooth over $\Spec L$ and
$Z ' _{(L)}:= \phi _L ^{-1} ( Z _{(L)} ) _{\mathrm{red}} $ is the support of a strict normal crossing divisor in 
$Y' _{(L)}$.

a) By using \cite[8.4.2]{EGAIV3}, \cite[8.7.2]{EGAIV3}, \cite[8.8.2.(ii)]{EGAIV3} and \cite[8.10.5.(v)]{EGAIV3}, 
there exists a finite (radicial) extension $l'$ of $l$ included in $L$  
such that there exist two integral $l'$-varieties $Y _{(l')}$ and $Y' _{(l')}$ 
satisfying 
$Y _{(L)} \riso Y _{(l')} \times _{\Spec (l')} \Spec (L)$
and
$Y' _{(L)} \riso Y' _{(l')} \times _{\Spec (l')} \Spec (L)$.

b) We put $X _{(l')}: =X \times _{P} \Spec (l')$ and $Z _{(l')}: =Z \times _{P} \Spec (l')$.
By increasing $l'$ is necessary, 
it follows from \cite[8.8.2.(i)]{EGAIV3} 
that 
there exists a morphism 
$\phi _{l'} := Y '_{(l')} \to Y _{(l')}$ 
(resp. $Y _{(l')} \to X _{(l')}$)
inducing $\phi _L$ 
(resp. the surjective closed immersion 
$Y _{(L)} \hookrightarrow X _{(L)}$).
By using \cite[17.7.8]{EGAIV4} and \cite[8.10.5]{EGAIV3} and Lemma \ref{SNCDlem2}, 
by increasing $l'$ is necessary, 
we can suppose that $Y _{(l')} \to X _{(l')}$ is a surjective closed immersion (i.e.
$Y _{(l')} = (X _{(l')} ) _{\mathrm{red}}$ since $Y _{(l')}$ is reduced), that
 $\phi _{l'}$ is projective, generically finite and etale morphism, 
 that $Y '_{(l')}$ is smooth over $\Spec l'$ and
$Z ' _{(l')}:= \phi _{l'} ^{-1} (Z _{(l')})  _{\mathrm{red}}$ is the support of a strict normal crossing divisor of 
$Y' _{(l')}$ relatively to $\Spec l'$.

2) Let $P'$ be the normalization of $P$ in $l'$ (see the Definition \cite[4.1.24]{Liu-livre-02}). 
Then the canonical morphism $g\colon P' \to P$ is a universal homeomorphism, i.e. 
is finite (e.g. use \cite[4.1.27]{Liu-livre-02}), surjective and radicial (e.g. use the exercise \cite[5.3.9.(a)]{Liu-livre-02}).
By using \cite[8.8.2.(ii)]{EGAIV3} (this time, we 
consider the projective system of open affine dense subvarieties of $P'$), 
there exists a dense open affine subvariety $U'$ of $P'$, two morphisms
$V ' \to U'$ and $\widetilde{V}' \to U'$ such that 
$Y _{(l')}\riso V'  \times  _{U'} \Spec (l')$
and
$Y '_{(l')}\riso \widetilde{V}'  \times  _{U'} \Spec (l')$.
Since $P'$ is noetherian, 
by using Propositions \cite[8.4.2]{EGAIV3} and \cite[8.7.2]{EGAIV3},
we get that $V'$ and $\widetilde{V}'$ are integral.
Hence, shrinking $U'$ is necessary, we can suppose
$V' =  (X \times _{P} U') _{\mathrm{red}}$.
By shrinking $U'$ is necessary, 
using \cite[8.8.2.(i)]{EGAIV3}, there exists 
a $U'$-morphism 
$f \colon \widetilde{V}' \to V ' $ which induces 
$\phi _{l'}$.
By shrinking $U'$ is necessary, 
by using \cite[17.7.8]{EGAIV4} and \cite[8.10.5]{EGAIV3}
and Lemma \ref{SNCDlem1}, 
we get the desired properties. 
\end{proof}

\begin{lem}
\label{rem-fact-Frob}
Let $g \colon U' \to U$ be a universal homeomorphism of integral $k$-varieties. 
We suppose $U$ normal. 
Then, for $s$ large enough, there exists a unique morphism
$h \colon U \to U ^{\prime \sigma}$ 
making commutative the following diagram
\begin{equation}
\label{rem-fact-Frob-diag}
\xymatrix{
{U '} 
\ar[r] ^-{g}
\ar[d] _-{F ^s _{U'/k}}
& 
{U} 
\ar[d] ^-{F ^s _{U/k}}
\ar[dl] _-{h}
\\ 
{U ^{\prime \sigma} } 
\ar[r] ^-{g ^\sigma}
& 
{U ^{\sigma}.} 
}
\end{equation}
Moreover, this morphism 
$h$ is a universal homeomorphism.

\end{lem}

\begin{proof}
From \ref{Frobe-univ-homeo}, we know that 
$F ^s _{U'/k}$ and 
$F ^s _{U/k}$ are universal homeomorphisms.
Hence, this is sufficient to check that there exist a unique morphism 
$h \colon U \to U ^{\prime \sigma}$
making commutative the diagram \ref{rem-fact-Frob-diag}. 
This is equivalent to check the existence and uniqueness of a morphism
$i \colon U \to U '$ making commutative the diagram
\begin{equation}
\label{rem-fact-Frob-diag2}
\xymatrix{
{U '} 
\ar[r] ^-{g}
\ar[d] _-{F ^s _{U'}}
& 
{U} 
\ar[d] ^-{F ^s _{U}}
\ar[dl] _-{i}
\\ 
{U ^{\prime } } 
\ar[r] ^-{g }
& 
{U .} 
}
\end{equation}
We can suppose $U$ affine. 
We set $U = \Spec A$, $U' = \Spec A'$,
$L:= \mathrm{Frac} A$,
$L':= \mathrm{Frac} A'$.
Since $g$ is surjective, $g ^* \colon A \to A '$ is injective. 
Since $A$ is normal and since $A \to A'$ is finite, 
then 
$A = A ' \cap L$.
For $s$ large enough, we can suppose
$(L') ^{p^s} \subset L$.
Hence the image of  
$F ^{s *} _{U'}\colon A ' \to A '$
is included in $A$. This yields
the desired morphism 
$A ' \to A$. 

\end{proof}

\begin{prop}
\label{corolem-loc-smooth}
Let $a \colon X \to P$ be a dominant morphism of smooth integral $k$-varieties. 
Let $Z \hookrightarrow X$ be a proper closed subset.
Then, for $s$ large enough, 
 there exists a dense open subvariety $U$ of $P$, 
 such that,  putting 
$ W:= (X ^{\sigma} \times _{P  ^{\sigma}} U) _{\mathrm{red}}$
(where $X ^{\sigma} \times _{P  ^{\sigma}} U$ means the base change
of $X ^{\sigma}$ by the composition of $F ^s _{U/k}\colon U \to U  ^{\sigma}$ with the open immersion 
$U  ^{\sigma}\subset P  ^{\sigma}$), there exists
a projective, generically finite and etale $U$-morphism of the form 
$\phi \colon W' \to W$ such that $W'$ is integral and smooth over $U$, 
$Z':= \phi  ^{-1} (Z ^\sigma  \times _{P ^\sigma} U ) _{\mathrm{red}}$ is the support of a strict normal crossing divisor in $W'$ relatively to 
$U$.
\end{prop}

\begin{proof}
Using the Lemmas \ref{lem-loc-smooth} (for the construction of $g$ and $f$)
and \ref{rem-fact-Frob} (for the construction of $h$), with their notation we get the diagram of morphisms of $k$-schemes
\begin{equation}
\notag
\xymatrix{
{W'} 
\ar[d] ^-{\phi}
\ar@{=}[r] ^-{}
\ar@{}[rd] ^-{}|\square
\ar[d] ^-{}
&
{W'} 
\ar[r] ^-{}
\ar@{}[rd] ^-{}|\square
\ar[d] ^-{}
&
{\widetilde{V} ^{\prime \sigma}} 
\ar[d] ^-{f ^{\sigma}}
&
{}
&
{}
\\
{T _{\red}}
\ar@{}[rd] ^-{}|\square
\ar@{^{(}->}[r] ^-{}
\ar@{=}[d] ^-{}
&
{T} 
\ar[r] ^-{}
\ar@{}[rd] ^-{}|\square
\ar@{^{(}->}[d] ^-{}
&
{(X ^{\sigma} \times _{P  ^{\sigma}} U ^{\prime \sigma}) _{\mathrm{red}}} 
\ar@{^{(}->}[d] ^-{}
&
{}
&
{}
\\
{W:= (X ^{\sigma} \times _{P  ^{\sigma}} U)  _\mathrm{red}}
\ar@{^{(}->}[r] ^-{}
&
{X ^{\sigma} \times _{P  ^{\sigma}} U} 
\ar[r] ^-{}
\ar@{}[rd] ^-{}|\square
\ar[d] ^-{}
&
{X ^{\sigma} \times _{P  ^{\sigma}} U ^{\prime \sigma}} 
\ar[r] ^-{}
\ar[d] ^-{}
\ar@{}[rd] ^-{}|\square
&
{X ^{\sigma} \times _{P  ^{\sigma}} U ^{ \sigma}}
\ar[r] ^-{}
\ar[d] ^-{}
\ar@{}[rd] ^-{}|\square
& 
{X ^{\sigma} } 
\ar[d] ^-{}
\\ 
&
{U}  
\ar[r] ^-{h}
\ar@/_0,6cm/[rr] ^-{F ^s _{U/k}}
& 
{U ^{\prime \sigma} } 
\ar[r] ^-{g ^{\sigma}}
& 
{U ^{\sigma} } 
\ar[r] ^-{}
& 
{P ^{\sigma} } 
}
\end{equation}
where $T := (X ^{\sigma} \times _{P  ^{\sigma}} U)
\times _{(X ^{\sigma} \times _{P  ^{\sigma}} U ^{\prime \sigma})} 
(X ^{\sigma} \times _{P  ^{\sigma}} U ^{\prime \sigma}) _{\mathrm{red}} $
and 
$W' := \widetilde{V} ^{\prime \sigma} \times _{U ^{\prime \sigma}} U$.
Since $W'$ is smooth over $U$ and $U$ is $k$-smooth, then $W'$ is $k$-smooth (and in particular integral since it is irreducible).
Hence $W' _{\red} = W' $ and then 
$W ' = W ' \times _T T _{\red}$ (use \cite[5.1.7]{EGAI}), which justify the cartesianity
of the left square of the top. Using \cite[5.1.7]{EGAI}, 
we get the equality 
$T _\mathrm{red} = (X ^{\sigma} \times _{P  ^{\sigma}} U)  _\mathrm{red}=:W$
and the cartesianity of the left square of the second row. 
Since $f ^\sigma$ is projective, generically finite and etale 
then so is 
$\phi \colon W' \to W$.
From \ref{SNCDlem0}, 
we get that $Z':= \phi ^{-1} (Z ^{\sigma} \times _{P  ^{\sigma}} U)  _\mathrm{red}$ is the support of a strict normal crossing divisor relatively to $U$.
\end{proof}

\begin{lem}
\label{FrobO-coh desc}
Let $f\colon \PP ' \to \PP$ be a finite, surjective morphism of smooth formal $\V$-schemes.
Let $\E$ be a coherent $\D ^{\dag} _{\PP,\Q}$-module. 
Then $\E$ is $\O _{\PP, \Q}$-coherent if and only if $f ^{!} (\E)$ is $\O _{\PP', \Q}$-coherent.
\end{lem}

\begin{proof}
Since $P$ and $P'$ are regular, then from \cite[4.3.11]{Liu-livre-02},
the morphism $P' \to P$ is flat. Since $\O _{\PP'}$ is $p$-adically complete and without $p$-torsion, 
then using Lemma \cite[2.1]{MonskyWashnitzer} (in the case where $I = (\pi)$), 
the morphism $f\colon\PP' \to \PP$ is also flat.
Since $f$ is also finite, then $f _* \O _{\PP '}$ is a locally free $\O _{\PP}$-module of finite type. 
Hence, we get that the canonical morphism
$f ^* \widehat{\D} ^{(m)} _{\PP}:=
\O _{\PP'} \otimes _{f ^{-1} \O _{\PP}} 
f ^{-1} \widehat{\D} ^{(m)} _{\PP}
\to 
\underleftarrow{\lim} _i \
\O _{P'} \otimes _{f ^{-1} \O _{P _i}} 
f ^{-1} \D ^{(m)} _{P _i}
=
 \widehat{\D} ^{(m)} _{\PP'\to \PP}$
is an isomorphism for any integer $m\geq0$.
Hence tensoring by $\Q$ over $\Z$ and passing the limits through 
the level, 
this yields that 
the canonical morphism
$f ^* \D ^{\dag} _{\PP,\Q}:=
\O _{\PP'} \otimes _{f ^{-1} \O _{\PP}} 
f ^{-1} \D ^{\dag} _{\PP,\Q}
\to 
\D ^{\dag} _{\PP'\to \PP,\Q}$
is an isomorphism.
Hence, 
this implies that
the canonical morphism $f ^* (\E) \to f ^{!} (\E)$
is an isomorphism, where
$f ^* \E:=
\O _{\PP'} \otimes _{f ^{-1} \O _{\PP}} 
f ^{-1} \E$,
and 
$f ^{!} (\E):= 
\D ^{\dag} _{\PP'\to \PP,\Q}
\otimes _{f ^{-1} \D ^{\dag} _{\PP,\Q}} ^\L
f ^{-1} \E$.
If $\E$ is $\O _{\PP, \Q}$-coherent 
this yields that $f ^{!} (\E)$ is $\O _{\PP', \Q}$-coherent.
Conversely, suppose $f ^{!} (\E)$ is $\O _{\PP', \Q}$-coherent.
Since the $\O _{\PP, \Q}$-coherence is local in $\PP$, we can suppose $\PP$ affine. Since $\E$ is a coherent $\D ^{\dag} _{\PP,\Q}$-module, 
this is sufficient to check that $\Gamma ( \PP, \E)$ is of finite type over $\Gamma ( \PP, \O _{\PP, \Q})$ 
(see \cite[2.2.13]{caro_courbe-nouveau}).
Since the extension
$\Gamma ( \PP, \O _{\PP, \Q}) \to \Gamma ( \PP', \O _{\PP', \Q})$ is faithfully flat (because $f \colon \PP' \to \PP$ is flat and surjective), 
since 
$\Gamma ( \PP ', f ^! \E) 
\riso 
\Gamma ( \PP ',\O _{\PP', \Q}) \otimes _{\Gamma ( \PP ,\O _{\PP, \Q})}
\Gamma ( \PP ,\E)$
is of finite type over $\Gamma ( \PP ',\O _{\PP', \Q})$,
we conclude.
\end{proof}

\subsection{The result}

\begin{lem}
\label{isoc-desc-finiteetale}
Let $\alpha \colon \widetilde{X} \to X$ be a finite étale surjective morphism of smooth $k$-varieties. 
Let $\E \in D ^{\mathrm{b}} _{\mathrm{ovhol}} (X/K)$ 
(resp. $\widetilde{\E} \in D ^{\mathrm{b}} _{\mathrm{ovhol}} (\widetilde{X}/K)$).
The property $\widetilde{\E} \in D ^{\mathrm{b}} _{\mathrm{isoc}} (\widetilde{X} /K)$
is equivalent to the property
$\alpha _{+}(\widetilde{\E} )\in D ^{\mathrm{b}} _{\mathrm{isoc}} (X /K)$.
The property $\E  \in D ^{\mathrm{b}} _{\mathrm{isoc}} (X /K)$
is equivalent to the property
$\alpha ^{+}(\E )\in D ^{\mathrm{b}} _{\mathrm{isoc}} (\widetilde{X} /K)$.

\end{lem}

\begin{proof}
Left to the reader.
\end{proof}

\begin{thm}
\label{lem-loc-acyc}
Let $\PP _{1}$ be a smooth separated formal $\V$-scheme, 
$\PP _{2}$ be a proper smooth formal $\V$-scheme,
$\PP:= \PP _1 \times \PP _2$ and
 $\mathrm{pr}\colon  \PP  \to \PP _1$ be the projection.
Let $\E$ be a complex of  $F \text{-}D ^{\mathrm{b}} _{\mathrm{ovhol}} (\D ^{\dag} _{\PP,\Q})$.

Then there exists an open dense formal subscheme $\U _{1}$ of $\PP _{1}$ 
such that, for any finite étale surjective morphisms of the form 
$\alpha _{1}\colon \widetilde{\PP} _{1} \to \PP _{1}$ and
$\alpha _{2}\colon \widetilde{\PP} _{2} \to \PP _{2}$, 
putting  
$\widetilde{\PP}= 
\widetilde{\PP}  _1 \times \widetilde{\PP} _2$, 
$\alpha \colon \widetilde{\PP} \to \PP$ 
and
$\widetilde{\E}:= \alpha ^{+} (\E) $,
we have
$(\mathrm{pr} \circ \alpha ) _{+} (\widetilde{\E}) | \U _{1} 
\in D ^{\mathrm{b}} _{\mathrm{coh}} (\O _{\U _{1},\Q})$.

\end{thm}

\begin{proof}
I)  We can suppose that $\alpha _1= Id$. Indeed, consider the following diagram
\begin{equation}
\notag
\xymatrix{
{\widetilde{\PP}  _1 \times \widetilde{\PP} _2} 
\ar[r] ^-{\alpha _2}
\ar@{}[rd] ^-{}|\square
\ar[d] ^-{\alpha _1}
&
{\widetilde{\PP}  _1 \times \PP _2} 
\ar[d] ^-{\alpha _1}
\ar[r] ^-{\mathrm{pr}}
\ar@{}[rd] ^-{}|\square
&
{\widetilde{\PP}  _1} 
\ar[d] ^-{\alpha _1}
\\ 
{\PP  _1 \times \widetilde{\PP} _2}
\ar[r] ^-{\alpha _2}
& 
{\PP _1 \times \PP _2} 
\ar[r] ^-{\mathrm{pr}}
& 
{\PP _1.}
}
\end{equation} 
Suppose there exists
an open dense formal subscheme $\U _{1}$ of $\PP _{1}$
such that 
$\mathrm{pr} _+ \circ \alpha _{2+} (\alpha _{2} ^{+} (\E))| \U _1
\in D ^{\mathrm{b}} _{\mathrm{coh}} (\O _{\U _{1},\Q})$.
Using base change isomorphism (see for instance \cite[1.3.10]{Abe-Caro-weights}), 
since $\alpha _1 ^{+} = \alpha _1 ^{!}$ (because $\alpha$ is finite etale),
we get
$\alpha _1 ^{+} \circ (\mathrm{pr} _+ \circ \alpha _{2+}) 
\riso 
( \mathrm{pr} _+\alpha _{2+} ) \circ \alpha _1 ^{+} $.
Hence, we get the first isomorphism:
$$\alpha _{1+} \alpha _1 ^{+}\mathrm{pr} _+ \circ \alpha _{2+} (\alpha _{2} ^{+} (\E))
\riso 
\alpha _{1+}  \mathrm{pr} _+\alpha _{2+} \alpha _1 ^{+} (\alpha _{2} ^{+} (\E))
\riso 
\alpha _{1+}  \mathrm{pr} _+\alpha _{2+} (\widetilde{\E})
\riso 
\mathrm{pr} _+ \alpha _{+}  (\widetilde{\E}).$$
From Lemma \ref{isoc-desc-finiteetale}, this implies that
 $(\mathrm{pr} \circ \alpha ) _+ (\widetilde{\E}) |\U _1
 \in 
 D ^{\mathrm{b}} _{\mathrm{coh}} (\O _{\U _{1},\Q})$.

II) We proceed by induction on the dimension of the support $X$ of $\E$. 
The case where $X \to P _1$ is not surjective is obvious (indeed, 
since $\alpha _{+} \alpha ^{+} (\E)$ has is support in $X$, 
we can choose $\U _1$ to be the  open dense subset of $\PP _1$ complementary to $\mathrm{pr} ( X)$).
Hence, we can suppose that $X \to P _1$ is surjective.
There exists a smooth dense open subvariety $Y$ of $X$ 
such that 
$\E  | Y \in F \text{-}D^{\mathrm{b}}_{\mathrm{isoc}} (Y,\PP/K)$
(see the notation \cite[1.2.14]{Abe-Caro-weights} and use \cite[3.1.1]{caro-2006-surcoh-surcv}). 
We put $Z := X \setminus Y$, endowed with its canonical structure of subvariety of $X$
(recall that varieties are reduced following the convention of the paper). 
Let $j \colon Y \subset X$ be the inclusion and $i \colon Z \hookrightarrow X$ the corresponding closed immersion.
By using the exact triangle of localisation of the form
$i _+ i ^! (\E) \to \E \to j _+ j ^{+} (\E) \to +1$ (see \cite[1.1.8.(ii)]{Abe-Caro-weights}), 
by devissage and by induction hypothesis, 
we can suppose $Y$ integral, that $\E$ is a module and that 
$\E \riso j _+ j ^{+} (\E)$ with 
$j ^{+} (\E )\in F \text{-}\mathrm{Isoc} ^{\dag \dag} (Y,\PP/K)$. 
By abuse of notation (to simplify them), $\PP_1$ will mean a dense open set $\U _1$ of $\PP _1$ 
(be careful that the open set has to be independent of the choice of $\alpha _2$), 
$X$ and $\PP$ will mean the base change of $X$ and $\PP$ by the inclusion $\U _1 \subset \PP _1$.

1) 
From de Jong desingularisation Theorem, 
there exists a surjective, projective, generically finite étale morphism $a \colon X ' \to X$,
with $X'$ integral and smooth such that $Z ':= a ^{-1} ( Z)$ is the support of a strict normal crossing divisor of $X'$.
Since $a$ is projective, there exists a closed immersion of the form $u' \colon X' \hookrightarrow \widehat{\P} ^{N} \times \PP$ (this is the product
in the category of formal schemes over $\V$)
such that the composition of $u'$ with the projection $f\colon \widehat{\P} ^{N} \times \PP \to \PP$ is equal to the composition of $a$ with the closed immersion
$X \hookrightarrow \PP$.
Since $\E \riso j _+ j ^{+} (\E)$, by setting 
$\E ' := a ^! (\E)$,  $Y ':= a ^{-1} (Y)$, $\PP' := \widehat{\P} ^{N} \times \PP$, 
we get 
$\E'  \riso j _+ j ^{+} (\E ')$ (use the base change isomorphism \cite[1.3.10]{Abe-Caro-weights})
and 
$j ^{+} (\E' )\in F \text{-}\mathrm{Isoc} ^{\dag \dag} (Y',\PP'/K)$
(with the convention given at the beginning of the paper, 
recall that $j$ means also the morphisms induced by base change form $j$).
We denote by $(X', M _{Z'})$
the smooth log whose
the underlying scheme is $X'$ and the log structure $M _{Z'}$ comes canonically from the strict normal crossing divisor $Z'$. Sometimes
we simply denote it by $(X', Z')$ if the notation is not confusing.
From Kedlaya's semistable reduction theorem (more precisely the global one, i.e. \cite[2.4.4]{kedlaya-semistableIV}),
we can suppose that the overconvergent isocrystal $\E ' | Y'$ on $Y'$ extends to a convergent log-$F$-isocrystal 
on $(X', M _{Z'})$.
Using the properties satisfied by $a$, using for instance \cite[6.3.1]{caro_devissge_surcoh},
we get  that $a _+ \circ a ^{!} (\widetilde{\E})$ is a direct factor of $\widetilde{\E}$.
Since by transitivity $\alpha ^+ a ^{!} \riso a ^{!}  \alpha ^+ $ (recall $\alpha ^+ = \alpha ^!$) and 
$a _+ \alpha _{+} \riso \alpha _{+} a _+ $ 
then $\mathrm{pr} _+a _+ \alpha _{+} \alpha ^+ (a ^{+} (\E))$ is a direct factor of $\mathrm{pr} _+\alpha _{+}\widetilde{\E}$.
By definition of cohomological operations (see the beginning of the paper),
$\mathrm{pr} _+a _+ \alpha _{+} \alpha ^+ (a ^{+} (\E))=
(\mathrm{pr}\circ f ) _+\alpha _{+} \alpha ^+ (\E')$, 
where in the latter term 
$\mathrm{pr}\circ f \colon 
\PP '= \widehat{\P} ^{N} \times \PP _2 \times \PP _1 \to \PP _1$
is the projection
and $\alpha \colon \widehat{\P} ^{N} \times \widetilde{\PP}  \to \widehat{\P} ^{N} \times \PP =\PP '$. 
This implies that
we can reduce to the case where $X$ is smooth, $Z$ is the support of a strict normal crossing divisor of $X$
and $\E |Y$ extends to a convergent log-$F$-isocrystal 
on $(X, M _Z)$ and we can forget the notation of part 1) of the proof.

2) 
From Proposition \ref{corolem-loc-smooth}, replacing $\PP _1$ by an open affine dense formal subscheme if necessary
and for $s$ large enough, putting $ W:= (X ^{\sigma} \times _{P  _1 ^{\sigma}} P _1) _{\mathrm{red}}$,
there exists 
a projective, surjective, generically finite and etale $P _1$-morphism of the form 
$\phi \colon W' \to W$ such that $W'$ is integral and smooth over $P _1$,
$Z':= \phi  ^{-1} (Z ^\sigma  \times _{P ^\sigma _1} P _1 ) _{\mathrm{red}}$ is the support of a strict normal crossing divisor in $W'$ relatively to 
$P _1$.
Let  $a \colon W= (X ^{\sigma} \times _{P  _1 ^{\sigma}} P _1) _{\mathrm{red}}\to X ^{\sigma}$ be 
canonical morphism.
 We set $Y' := \phi  ^{-1} (Y ^\sigma  \times _{P ^\sigma _1} P _1 ) _{\mathrm{red}}$.
Put $\psi := Y'\to (Y ^{\sigma} \times _{P  _1 ^{\sigma}} P _1)  _\mathrm{red}$
and $b \colon 
(Y ^{\sigma} \times _{P  _1 ^{\sigma}} P _1)  _\mathrm{red} \to Y ^{\sigma} $ the morphisms induced respectively by 
$\phi$ and $a$.
Since $\phi$ is projective, for some integer $N$ putting 
$\PP _3:= \widehat{\P} ^{N} \times \PP _{2} ^{\sigma}$ 
and taking $f\colon \PP _3 \to \PP _{2} ^{\sigma}$ to be the projection, 
we get the commutative diagram of the left: 
\begin{equation}
\label{2diag-psi-phi-a-b}
\xymatrix{
{Z'} 
\ar@{^{(}->}[r] ^-{}
\ar@{}[rd] ^-{}|{\square _{\mathrm{red}}}
\ar[d] ^-{}
&
{W'} 
\ar@{^{(}->}[r] ^-{}
\ar[d] ^-{\phi }
&
{\PP ' =\PP _3 \times \PP _1}
\ar[r] ^-{\mathrm{pr}'}
\ar[d] ^-{f }
& 
{\PP _1} 
\ar@{=}[d] ^-{}
\\ 
{(Z ^{\sigma} \times _{P  _1 ^{\sigma}} P _1)  _\mathrm{red}} 
\ar@{^{(}->}[r] ^-{}
\ar[d] ^-{}
\ar@{}[rd] ^-{}|{\square _{\mathrm{red}}}
&
{W=(X ^{\sigma} \times _{P  _1 ^{\sigma}} P _1)  _\mathrm{red}} 
\ar@{^{(}->}[r] ^-{}
\ar[d] ^-{a}
\ar@{}[rd] ^-{}|{\square _{\mathrm{red}}}
&
{\PP _2 ^{\sigma} \times \PP _1}
\ar[r] ^-{\mathrm{pr} ^{\sigma}}
\ar[d] ^-{F ^{s} _{\PP _1 /\V}}
\ar@{}[rd] ^-{}|\square
& 
{\PP _1} 
\ar[d] ^-{F ^{s} _{\PP _1 /\V}}
\\
{Z ^{\sigma }} 
\ar@{^{(}->}[r] ^-{}
&
{X ^{\sigma }} 
\ar@{^{(}->}[r] ^-{}
&
{\PP _2 ^{\sigma}  \times \PP _1 ^{\sigma} }
\ar[r] ^-{\mathrm{pr} ^{\sigma}}
& 
{\PP _1 ^{\sigma},} 
}
\xymatrix{
{Y'} 
\ar@{^{(}->}[r] ^-{j ^{\sigma}}
\ar@{}[rd] ^-{}|{\square }
\ar[d] ^-{\psi}
&
{W'} 
\ar[d] ^-{\phi }
\\ 
{(Y^{\sigma} \times _{P  _1 ^{\sigma}} P _1)  _\mathrm{red}} 
\ar@{^{(}->}[r] ^-{j ^{\sigma}}
\ar[d] ^-{b}
\ar@{}[rd] ^-{}|{\square }
&
{W}
\ar[d] ^-{a}
\\
{Y ^{\sigma }} 
\ar@{^{(}->}[r] ^-{j ^{\sigma}}
&
{X ^{\sigma }} 
}
\end{equation}
where the symbol ``$\square _{\mathrm{red}}$'' means the cartesianity in the category of reduced schemes,
where $\mathrm{pr}' \colon \PP _3 \to \mathrm{Spf}\, \V$ is the structural morphism.
Since $(Y ^{\sigma} \times _{P  _1 ^{\sigma}} P _1)  _\mathrm{red} = a ^{-1} (Y ^{\sigma})$, 
we get the morphism $j ^{\sigma} \colon (Y ^{\sigma} \times _{P  _1 ^{\sigma}} P _1)  _\mathrm{red} 
\to (X ^{\sigma} \times _{P  _1 ^{\sigma}} P _1)  _\mathrm{red}=W$. 
We have also $j ^{\sigma} \colon Y ' \to W'$. Hence, this justifies the cartesianity of the diagram of the right of \ref{2diag-psi-phi-a-b}.

We put
$\G := a ^{!} (\E ^{\sigma})= (F ^{s} _{\PP _1 /\V}) ^{!} (\E ^{\sigma})$ and $\G ':= \phi ^{!} (\G)= \R\underline{\Gamma} ^{\dag} _{W'} f ^{!} (\G)$.
Since $a \circ \phi$ induces the morphism of smooth log-schemes 
$(W' , M _{Z'}) \to (X ^{\sigma} , M _{Z ^{\sigma}})$, since $\E  ^{\sigma} | Y ^{\sigma}$ 
extends to a convergent log-$F$-isocrystal 
on $(X  ^{\sigma}, M _{Z  ^{\sigma}})$
then 
$\G' |Y'$ extends to a convergent log-$F$-isocrystal 
on $(W' , M _{Z'})$.

3) 
Let $Z '_1, \dots , Z '_{r'}$ be the irreducible components of $Z'$.
For any subset $I'$ of $\{ 1, \dots , r'\}$, 
we set $Z '_{I '} : =\cap _{i '\in I'} Z _{i'}$.
Then $|\mathrm{Car} (\G ' )|
\subset \cup _{I '\subset \{ 1, \dots , r'\}}T ^{*} _{Z '_{I'}} P '$.

{\it Proof.} Since the check is local on $P'$, we can suppose $P'$ affine with  
local coordinates
$t '_{1}, \dots , t ' _{d'}$ inducing local coordinates $\overline{t}' _1, \dots, \overline{t} '_{n'}$ of $W'$
and such that  
$Z '_{i'} = V (\overline{t} '_{i'})$ for $i'=1,\dots, r'$. 
From \cite{sga1},
there exists a smooth affine formal $\V$-scheme $\mathfrak{W}'$ whose special fiber is $W'$. 
Let $u \colon \mathfrak{W}' \hookrightarrow \PP'$ be a lifting of $W '\hookrightarrow P '$ and let 
$\FF ':= u ^{!} (\G ')$. 
From \cite[1.4.3.1]{Caro-Lagrangianity},
we have $|\mathrm{Car} (\FF' )|
\subset \cup _{I '\subset \{ 1, \dots , r'\}}T ^{*} _{Z '_{I'}} W '$.
Since $\G ' \riso u _{+} (\FF ')$ (this comes from Berthelot-Kashiwara's theorem), 
from \cite[5.3.3]{Beintro2}, 
we get $|\mathrm{Car} (\G') | = \mathscr{T} _{u} (|\mathrm{Car} (\FF ')  |) $.
Using 
\ref{lem-incluT*strat}, 
we get 
$\mathscr{T} _{u} (T ^{*} _{Z '_{I'}} W ')
=T ^{*} _{Z '_{I'}} P'$, which gives the desired result.

4) We put 
$\widetilde{\G} ' := \alpha _2 ^{\sigma+} (\G ')$.
We have $|\mathrm{Car} ( \mathrm{pr} ' _+ \alpha _{2+} ^{\sigma} (\widetilde{\G} ' ) )|
 \subset 
 T ^{*} _{P_1} P_1$.

{\it Proof.} 
Since $\mathrm{pr}' $ is proper, 
from \ref{dag-carf+redu},
we get the inclusion
$$|\mathrm{Car} ( \mathrm{pr} ' _+ \alpha _{2+} ^{\sigma}  (\widetilde{\G} ' ) )|
\subset 
\mathscr{T} _{\mathrm{pr} '}(|\mathrm{Car} (\alpha _{2+} ^{\sigma}  (\widetilde{\G} ' ))|).$$
Using \ref{ZCar-equality}, we get
$|\mathrm{Car} (\alpha _{2+} ^{\sigma}  \alpha _{2} ^{\sigma +} (\G ') |
=
|\mathrm{Car} (\G ')|$.
Hence, 
$$
\mathscr{T} _{\mathrm{pr} '}(|\mathrm{Car} (\alpha _{2+} ^{\sigma}  (\widetilde{\G} ' ))|)
=
\mathscr{T} _{\mathrm{pr} '}(|\mathrm{Car} (\G')|)
\subset 
\mathscr{T} _{\mathrm{pr} '} (\cup _{I '\subset \{ 1, \dots , r'\}}T ^{*} _{Z '_{I'}} P ')
= \cup _{I '\subset \{ 1, \dots , r'\}} ( \mathscr{T} _{\mathrm{pr} '} (T ^{*} _{Z '_{I'}} P ' )),$$
where the inclusion comes from the step 3).
Moreover, from 
\ref{lem-incluT*strat}, 
$\mathscr{T} _{u _{I'}} (T ^{*} _{Z '_{I'}} Z _{I'} ')= 
T ^{*} _{Z '_{I'}} P '$, where 
$u _{I'} \colon Z '_{I'} \hookrightarrow P'$ is the closed immersion. 
By transitivity of the application $\mathscr{T}$ (see \ref{stab-T_X X}), 
$ \mathscr{T} _{\mathrm{pr} '} (T ^{*} _{Z '_{I'}} P ' )=
\mathscr{T} _{\mathrm{pr} '} (\mathscr{T} _{u _{I'}} (T ^{*} _{Z '_{I'}} Z _{I'} '))
=
\mathscr{T} _{pr '\circ u _{I'}} (T ^{*} _{Z '_{I'}} Z _{I'} ')$.
Since $pr ' \circ u _{I'}\colon Z ' _{I'} \to P_1$ is smooth, 
using 
\ref{lem-stab-T_X X}.2,
we get the inclusion
$\mathscr{T} _{pr '\circ u _{I'}} (T ^{*} _{Z '_{I'}} Z _{I'} ')
\subset 
T ^{*} _{P_1} P_1$, which yields the desired result.

5) $\G $ is a direct factor of 
$ \phi _{+} (\G ')$ (which is by definition, if we look at the left diagram of \ref{2diag-psi-phi-a-b},  equal to $f _{+} (\G')$).

a) 
We have the isomorphism 
$b ^{!} \riso b ^{+}$. Indeed, from 
\cite[1.3.12]{Abe-Caro-weights}, since $b$ is a universal homeomorphism, 
the functors  $b ^{!}$ and $b _+$ induce quasi-inverse equivalences of categories (for
categories of overholonomic complexes). 
Since $b$ is proper, then $b _+ = b _!$ (i.e., via the biduality isomorphism, $b _+$ commutes with dual functors).
Hence, we get that $b ^!$ commutes also with dual functors.

b) Since $\theta := b \circ \psi $ is a morphism of smooth varieties and 
$\E ^{\sigma}|Y ^{\sigma}$ is an isocrystal, then 
$\theta ^{!} (\E ^{\sigma} |Y ^{\sigma}) \riso \theta ^{+} (\E ^{\sigma}|Y ^{\sigma}) $.
Hence, since $\theta$ is proper, we get the morphisms by adjunction (see \cite[1.3.14.(viii)]{Abe-Caro-weights})
$\E ^{\sigma} |Y ^{\sigma}\to 
\theta _{+}\theta ^{+} (\E ^{\sigma}|Y ^{\sigma}) 
\riso 
\theta _{+}\theta ^{!} (\E ^{\sigma}|Y ^{\sigma})
\to 
\E ^{\sigma}|Y ^{\sigma}$.
The composition is an isomorphism. 
Indeed, since $\E ^{\sigma} |Y ^{\sigma}$ is an isocrystal, we reduce to check it on a dense open subset of $Y ^{\sigma}$.
Hence, we can suppose that $\psi$ and $b$ are morphisms of smooth varieties. 
Using \cite[1.3.12]{Abe-Caro-weights} and the transitivity of the adjunction morphisms, 
we reduce to check such property for $\psi$.
Since $\psi$ is generically finite and etale, this is already known (e.g. see 
\cite[6.3.1]{caro_devissge_surcoh}).

c) We have just checked that 
$\E ^{\sigma} |Y ^{\sigma}$ is a direct factor of 
$\theta _{+}\theta ^{!} (\E ^{\sigma}|Y ^{\sigma})$.
This implies that 
$b ^{!}(\E ^{\sigma} |Y ^{\sigma} )$  is a direct factor of 
$b ^{!} \theta _{+}\theta ^{!} (\E ^{\sigma}|Y ^{\sigma})
\riso 
b ^{!} b _+ \psi _{+}\theta ^{!} (\E ^{\sigma}|Y ^{\sigma})
\underset{\cite[1.3.12]{Abe-Caro-weights}}{\riso}
\psi _{+} 
\theta ^{!} (\E ^{\sigma}|Y ^{\sigma})
\riso 
\psi _{+} (\G ' |Y ')$.
We get 
$j ^{\sigma }  _+ b ^{!}(\E ^{\sigma} |Y ^{\sigma} )$ is a direct factor of 
$j ^{\sigma }  _+ \psi _{+} (\G ' |Y ')$.
By base change isomorphism (e.g. see \cite[1.3.10]{Abe-Caro-weights}), 
by using the cartesianity of the right diagram of \ref{2diag-psi-phi-a-b},
we get the isomorphism
$(a \circ \phi ) ^! j ^{\sigma } _+
\riso 
j ^{\sigma } _+ \theta ^!$.
By applying the functor 
$(a \circ \phi ) ^! $ to the isomorphism
$\E ^{\sigma} \riso j ^{\sigma} _+  j ^{\sigma !} \E ^{\sigma} $
we obtain
$\G ^{\prime} \riso j ^{\sigma} _+  j ^{\sigma !} \G ^{\prime}$.
This yields 
$\phi _+ (\G') \riso \phi _{+} j ^{\sigma }  _+  (\G ' |Y ')
\riso 
j ^{\sigma }  _+ \psi _{+} (\G ' |Y ')$
and 
$\G \riso a ^{!} j ^{\sigma }  _+ (\E ^{\sigma} |Y ^{\sigma} )
\riso 
j ^{\sigma }  _+ b ^{!}(\E ^{\sigma} |Y ^{\sigma} )$, 
which gives the desired result.

6) Putting $\widetilde{\G}: =\alpha _{2+} ^{\sigma} (\G)$,
(using some base change isomorphism) the step 5) implies 
that
 $\widetilde{\G}$ is a direct factor of 
$ \phi _{+} (\widetilde{\G} ')= f _{+} (\widetilde{\G} ')$.
Since
$\mathrm{pr} ' _+ \alpha _{2+} ^{\sigma}  (\widetilde{\G} ' ) 
\riso 
\mathrm{pr} ^{\sigma } _+ \alpha _{2+} ^{\sigma} \phi _+  (\widetilde{\G} ' )$,
we obtain
$|\mathrm{Car} ( \mathrm{pr} ^{\sigma } _+ \alpha _{2+} ^{\sigma} (\widetilde{\G} ))| 
\subset 
|\mathrm{Car} ( \mathrm{pr} ' _+ \alpha _{2+} ^{\sigma}  (\widetilde{\G} ' ) )| $.
From part 4), this yields
$|\mathrm{Car} ( \mathrm{pr} ^{\sigma } _+ \alpha _{2+}  ^{\sigma}(\widetilde{\G} ))|
\subset T ^{*} _{P_1} P_1$.
Let $\widetilde{\mathrm{pr}}\colon \widetilde{\PP} _2 \to \Spf \V$ be the structural morphism.  
Since 
$\mathrm{pr} ^{\sigma } _+ \alpha _{2+} ^{\sigma} (\widetilde{\G} )
=
\widetilde{\mathrm{pr}} ^{\sigma } _+ (\widetilde{\G} )$, 
using the second part of the Lemma \ref{lem-O-coh-Car},
this inclusion is equivalent to say that
$\widetilde{\mathrm{pr}} ^{\sigma } _+ (\widetilde{\G} )
\in D ^{\mathrm{b}} _{\mathrm{coh}} (\O _{\PP _{1},\Q})$.
Since 
$\widetilde{\G} \riso  a ^{!}(\widetilde{\E} ^{\sigma} ) $, 
we obtain 
$\widetilde{\mathrm{pr}} ^{\sigma } _+ (\widetilde{\G} ) 
\riso 
\widetilde{\mathrm{pr}} ^{\sigma } _+ (a ^{!}(\widetilde{\E} ^{\sigma} ))
\riso
(F ^{s} _{\PP _1 /\V}) ^{!} \widetilde{\mathrm{pr}} ^{\sigma } _+ (\widetilde{\E} ^{\sigma} )$.
From Lemma \ref{FrobO-coh desc}, this implies 
that 
$\widetilde{\mathrm{pr}} ^{\sigma } _+ (\widetilde{\E} ^{\sigma} )
\in D ^{\mathrm{b}} _{\mathrm{coh}} (\O _{\PP ^{\sigma }_{1},\Q})$, which yields
$\widetilde{\mathrm{pr}} _+ (\widetilde{\E} )
\in D ^{\mathrm{b}} _{\mathrm{coh}} (\O _{\PP _{1},\Q})$.

\end{proof}

\section{Betti numbers estimates}

\subsection{The curve case}

\begin{lem}
\label{fsmoothreldimvanish}
Let $f \colon Y \to X$ be a smooth morphism of integral smooth $k$-varieties
of relative dimension $d$ (i.e. $d = \dim Y - \dim X$).
If $\FF \in D ^{\geq 0} _{\mathrm{ovhol}} (Y/K)$ then 
$f _+ (\FF) \in D ^{\geq -d} _{\mathrm{ovhol}} (X/K)$.
If $\G \in D ^{\leq 0} _{\mathrm{ovhol}} (Y/K)$ then 
$f _! (\G) \in D ^{\leq d} _{\mathrm{ovhol}} (X/K)$.

\end{lem}

\begin{proof}
Since the second statement follows by duality
(recall $f _{!}=\DD _X \circ f _{+} \circ \DD _Y$ and dual functors exchange
$D ^{\geq n}$ with $D ^{\leq -n}$), let us check the first one.
As explained in the convention of the paper, since our varieties are realizable, 
there exist 
a smooth morphism $\phi \colon \QQ \to \PP$ 
of proper smooth formal $\V$-schemes, 
immersions
$\iota \colon X \hookrightarrow \PP$,
$\iota '\colon Y \hookrightarrow \QQ$ 
so that $\phi \circ \iota '= \iota  \circ f$.
The morphism $f$ is the composition of 
an immersion of the form $u '\colon Y \hookrightarrow \phi ^{-1} (X)$ 
followed by the morphism $\phi ^{-1} (X) \to X$ induced by $\phi$.
Since $u ' _{+} \colon D ^{\geq 0} _{\mathrm{ovhol}} (Y/K) \to D ^{\geq 0} _{\mathrm{ovhol}} (\phi ^{-1} (X)/K)$, 
we reduce to the case where 
$ Y= \phi ^{-1} (X)$.
Let $\U$ be an open set of $\PP$ such that 
$\iota$ factors through a closed immersion 
$X \hookrightarrow \U$. 
Put $\mathfrak{V}:= \phi ^{-1} (\U)$ and $\psi \colon \mathfrak{V}\to \U$ be the morphism induced by $\phi$. 
 By definition, 
$\FF \in D ^{\geq 0} _{\mathrm{ovhol}} (Y/K)= D^{\geq 0} _{\mathrm{ovhol}} (Y,\QQ/K)$,
$f _{+} (\FF)= \phi _{+} (\FF)$
and
$\phi _{+} (\FF) \in D^{\geq -d} _{\mathrm{ovhol}} (X,\PP/K)$
means that 
$\phi _{+} (\FF) |\U \in D^{\geq -d} _{\mathrm{ovhol}} (\D ^\dag _{\U,\Q})$.
Since this last property is local in $\U$, we can suppose $\U$ affine. 
Hence, there exists a closed immersion of smooth formal $\V$-schemes 
$u\colon \X \hookrightarrow \U$ which lifts $X \hookrightarrow \U$. 
Set $\Y:= \mathfrak{V} \times _{\U} \X$, which is a smooth lifting of $Y$.
Put $v \colon \Y \hookrightarrow \mathfrak{V}$ the projection, which is a closed immersion
and $\theta \colon \Y \to  \X$ the second projection. 
We have
$\psi _{+} (\FF |\mathfrak{V}) \riso \phi _{+} (\FF) |\U $
and 
$\FF |\mathfrak{V} \in D^{\geq 0} _{\mathrm{ovhol}} (Y,\mathfrak{V}/K) \cong 
D^{\geq 0} _{\mathrm{ovhol}}(\D ^\dag _{\Y,\Q})$, 
where the latter equivalence of categories are given 
by the quasi-inverse functors
$v _+$ and $v ^!$ (this is a form of Berthelot-Kashiwara's theorem).
Hence, we reduce to check that 
the functor $\theta _+$ induces the factorisation
$\theta _{+} \colon 
D^{\geq 0} _{\mathrm{ovhol}}(\D ^\dag _{\Y,\Q})
\to 
D^{\geq -d} _{\mathrm{ovhol}}(\D ^\dag _{\X,\Q})$,
which is obvious because, since $\theta$ is a smooth morphism of smooth formals $\V$-schemes
of relative dimension $d$, then we have
$\theta _{+} (\E) =
\R \theta _{*} ( \E \otimes _{\O _{\Y}} \Omega ^\bullet _{\Y/\X}) [d]$
for any $\E \in D^{\geq 0} _{\mathrm{ovhol}}(\D ^\dag _{\Y,\Q})$.
\end{proof}

\begin{lem}
\label{coro-fsmoothreldimvanish}
Let $X$ be an integral $k$-variety of dimension $d$.
If $\FF \in D ^{\geq 0} _{\mathrm{ovhol}} (X/K)$ then 
$f _+ (\FF) \in D ^{\geq -d} _{\mathrm{ovhol}} (\Spec k /K)$.
If $\G \in D ^{\leq 0} _{\mathrm{ovhol}} (X/K)$ then 
$f _! (\G) \in D ^{\leq d} _{\mathrm{ovhol}} (\Spec k /K)$.
\end{lem}

\begin{proof}
As for the proof of \ref{fsmoothreldimvanish}, we reduce to check the first statement. 
We proceed by induction on the dimension of $X$.
Choose an open smooth dense subvariety $U$ of $X$.
Put $Z := X \setminus U$.
Let $j \colon U \hookrightarrow X$,
$i \colon Z \hookrightarrow X$ be the corresponding morphisms of $k$-varieties. 
Since
$i ^{!} (\FF) \in D ^{\geq 0} _{\mathrm{ovhol}} (Z/K)$, 
then, by induction hypothesis, 
we get $p _{Z+}i ^{!} (\FF) \in D ^{\geq -d} _{\mathrm{ovhol}} (\Spec k/K)$.
From Lemma \ref{fsmoothreldimvanish}, 
we get $p _{U+} j ^{!} (\FF) \in D ^{\geq -d} _{\mathrm{ovhol}} (\Spec k/K)$. 
Applying $p _{X+}$ to the exact triangle of localisation
$i _+ i ^! (\FF) \to \FF \to j _+ j ^! (\FF) \to +1$ (see \cite[1.1.8.(ii)]{Abe-Caro-weights}), we get 
the exact triangle 
$p _{Z+} i ^! (\FF) \to p _{X+}(\FF) \to p _{U+} j ^! (\FF) \to +1$.
\end{proof}

\begin{lem}
\label{ZXIsoc}
Let $u\colon Z \hookrightarrow X$ be a closed immersion of $k$-varieties. 
Suppose $X$ smooth and integral.
Let $\E \in F \text{-}\mathrm{Isoc} ^{\dag \dag} (X/K)$. 
Then $u  ^{!} (\E)\in F \text{-}D ^{\geq r} _{\mathrm{ovhol}} (Z/K)$ (see the notation of \cite[1.2]{Abe-Caro-weights}),
with $r = \dim X - \dim Z$.
\end{lem}

\begin{proof}
First, suppose that $Z$ is smooth. 
Then, we get the exact functor 
$u  ^{!}  [r] \colon F \text{-}\mathrm{Isoc} ^{\dag \dag} (X/K) \to F \text{-}\mathrm{Isoc} ^{\dag \dag} (Z/K)$
and the Lemma follows. 
More generally, we prove the Lemma by induction on the dimension of $Z$.
When the dimension of $Z$ is $0$ then $Z$ is a finite etale over $\Spec k$
(recall a reduced $k$-scheme of finite type of dimension $0$ is finite etale over $\Spec k$, 
and by our convention $k$-varieties are assumed to be reduced)
and then this case has already been checked. 
Suppose $\dim Z >1$. Then there exists a dense open smooth subset $Z _0$ of $Z$. Put
$T := Z \setminus Z _0$. Let $j\colon Z _0 \to Z $ and $i\colon T \to Z$ be the corresponding immersions.
Put $\FF := u ^! (\E)$. 
Then we conclude by using the induction hypothesis and 
Consider the exact triangle 
of localisation 
$i _+ i ^! (\FF) \to \FF \to j _+ j ^! (\FF) \to +1$ (see \cite[1.1.8.(ii)]{Abe-Caro-weights}). 
From the smooth case, 
$ j ^! (\FF) [r] \in F \text{-}\mathrm{Isoc} ^{\dag \dag} (Z _0/K)$
and then 
$j _+ j ^! (\FF)  \in F \text{-}D ^{\geq r} _{\mathrm{ovhol}} (Z/K)$. 
By induction hypothesis, we get 
$i ^! (\FF) \in F \text{-}D ^{\geq r'} _{\mathrm{ovhol}} (T/K)$, 
with $r ' := \dim X - \dim T \geq r$.
Since $i _+$ is exact, 
$i _+ i ^! (\FF)\in F \text{-}D ^{\geq r'} _{\mathrm{ovhol}} (Z/K)\subset
F \text{-}D ^{\geq r} _{\mathrm{ovhol}} (Z/K)$
\end{proof}

\begin{ntn}
\label{ntn-rk}
Let $X$ be an integral variety and 
$\E \in F \text{-}\mathrm{Ovhol} (X/K)$. 
Then, there exists a smooth dense open subvariety $Y$ of $X$ 
such that 
$\E |Y \in F \text{-}\mathrm{Isoc} ^{\dag \dag} (Y/K)$
(see the notation \cite[1.2.14]{Abe-Caro-weights} and use \cite[3.1.1]{caro-2006-surcoh-surcv}). 
Then, by definition, $\mathrm{rk} (\E)$ means the rank of the corresponding overconvergent isocrystal associated to $\E |Y$ 
(which does not depend on the choice of such open dense subvariety $Y$).
\end{ntn}

\begin{lem}
\label{estiH-d}
We suppose that $k$ is infinite.
Let $f \colon Y \to X$ be a smooth morphism of integral smooth $k$-varieties.
We suppose there exists a $k$-valued point $x$ of $X$ such that
$Y _x := f ^{-1} (x)$ is an integral $k$-variety of dimension $d$.
Let $\FF \in F \text{-}\mathrm{Isoc} ^{\dag \dag} (Y/K)$ such that 
$f _+ (\FF) \in F \text{-} D ^{\mathrm{b}} _{\mathrm{isoc}} (X/K)$.
Then we have the inequalities: 
\begin{gather}
\label{estiH-d+}
\mathrm{rk} H ^{-d} f _{+} (\FF) \leq \mathrm{rk} (\FF);
\\
\label{estiH-d!}
\mathrm{rk} H ^{d} f _{!} (\FF) \leq \mathrm{rk} (\FF).
\end{gather}

\end{lem}

\begin{proof}
0) 
We remark that
$\mathrm{rk} (\FF)
=
\mathrm{rk} (\DD _Y (\FF))$
(recall \cite[3.12]{Abe-Frob-Poincare-dual}). 
Since $\DD _X H ^{-d} f _{+} (\FF)
=
H ^{d}  \DD _X f _{+} (\FF)
=
H ^{d} f _{!}  \DD _Y  (\FF)$ (recall $f _{!} (\FF) := \DD _X (f _{+} (\DD _Y (\FF ))) $
and 
$\DD _Y \circ \DD _Y \riso Id$),
we get similarly
$\mathrm{rk} H ^{-d} f _{+} (\FF)
=
\mathrm{rk}H ^{d} f _{!}  (\DD _Y  (\FF))$.
Hence,  we reduce to check the first inequality. 
Let $i _x \colon x \hookrightarrow X$ be the canonical closed immersion.
Recall that from the  convention of the paper 
$i _x \colon Y _x \hookrightarrow Y$ 
(resp. $f \colon Y _x \to x$) means the morphism induced from $i _x$ (resp. $f$)
by base change. 
The functors $i _x ^{!} [\dim X]  \colon F \text{-}\mathrm{Isoc} ^{\dag \dag} (X/K)
\to 
F \text{-}\mathrm{Isoc} ^{\dag \dag} (x/K)$
and 
$i _x ^{!} [\dim X]  \colon F \text{-}\mathrm{Isoc} ^{\dag \dag} (Y/K)
\to 
F \text{-}\mathrm{Isoc} ^{\dag \dag} ( Y _x/K)$
are exact and preserve the rank (and in particular an isocrystal $\G$ is null if and only if 
$i _x ^{!} [\dim X] (\G) = 0$). 
We have the base change isomorphism
$i _x ^{!} [\dim Y] \circ f _+ (\FF)
\riso 
f _+\circ  i _x ^{!} [\dim Y]  (\FF)$ (e.g. see \cite[1.3.10]{Abe-Caro-weights}).
Hence, 
we reduce to the case where $X = \Spec k$, i.e. 
$f= p _Y$  (see the notation of the paper).

1) We check that for any open dense subset $V$ of $Y$, we have  
$H ^{-d} p _{Y +} (\FF) \riso H ^{-d} p _{V +} (\FF |V)$.  
Indeed, let $V$ be a open dense subset of $Y$ and put $Z := Y\setminus V$.
We denote by $j \colon V  \hookrightarrow Y$ the canonical open immersion
and 
$i \colon Z \hookrightarrow Y$ the canonical closed immersion. 
Since $V $ is dense in $Y$, then 
$d _Z:=\dim Z < d$.
Since $\FF$ is an isocrystal, then from 
\ref{ZXIsoc} we have
$i  ^{!} (\FF)\in F \text{-}D ^{\geq d-d _{Z}} _{\mathrm{ovhol}} (Z/K)$.
Hence, via 
\ref{coro-fsmoothreldimvanish}, 
this implies 
$p _{Z +} i ^{!} (\FF)  \in F \text{-}D ^{\geq d-2d _{Z}} _{\mathrm{ovhol}} (\Spec k/K)$.
Since $d _{Z} \leq d -1$, we get $d-2d _{Z} \geq -d +2$. 
This yields 
$H ^{-d} p _{Z+} i ^{!} (\FF)=0$
and $H ^{-d +1} p _{Z +} i ^{!} (\FF)=0$.
Applying $p _{Y+}$ to the exact triangle of localisation
$i _+ i ^! (\FF) \to \FF \to j _+ j ^! (\FF) \to +1$ (see \cite[1.1.8.(ii)]{Abe-Caro-weights}), we get 
the exact triangle 
$p _{Z+} i ^! (\FF) \to p _{Y+}(\FF) \to p _{V+}  (\FF |V) \to +1$.
By considering the long exact sequence associated with the latter exact triangle, 
we conclude.

2) We check the lemma in the case where $d =1$. 
From 1), we can suppose that $Y$ is affine. 
Choose a smooth compactification $\overline{Y}$ of $Y$ and put 
$D := \overline{Y} \setminus Y$. Choose a closed point $y$ of $D$. 
With the notation \cite[7.1.1]{crewfini}, 
we associate to $\FF$ an $A ^\dag _Y$-module module $M$ endowed with a connexion.  
Then $H ^{-1} p _{Y +} (\FF)$ is equal to the horizontal sections of $M$. 
Let $A (y)$ be the Robba ring (or local algebra following the terminology of \cite{crewfini})
corresponding to $y$ (see the notation of \cite[7.3]{crewfini}). 
Using \cite[6.2]{crewfini} we get that the dimension over $K$ of the $K$-vector space of the horizontal sections of 
$M \otimes _{A ^\dag _Y} A (y)$ (which is bigger than that of the horizontal sections of $M$) is less or equal to the rank of $M$ (which is also
the rank of $\FF$).

3) Now we prove the lemma by induction on $d$. 
Suppose $d \geq 2$.
From part 1) of the proof, using \cite{Kedlaya-coveraffinebis}, 
we can suppose that there exists 
a finite etale morphism of the form
$\phi \colon Y \to \A ^{d} _k$. Let $g $ be the composite of $\phi$ with the projection 
$ \A ^{1} _k \times \A ^{d-1} _k \to \A ^{d-1} _k$.
There exists a dense open subvariety $U$ of $\A ^{d-1} _k$
such that $g _+ (\FF) |U \in F \text{-} D ^{\mathrm{b}} _{\mathrm{isoc}} (U/K)$
(see the notation \cite[1.2.14]{Abe-Caro-weights} and use \cite[3.1.1]{caro-2006-surcoh-surcv}). 
Let $V := g ^{-1} (U)$ and $h\colon V \to U$ 
the induced smooth morphism of relative dimension $1$. 
Since $k$ is infinite and $U$ is dense in $\A ^{d-1} _k$, 
there exists a $k$-valued point $x$ of  $U$.
Since $g$ is surjective, $g ^{-1} (x)$ is a smooth variety of dimension $1$.
Shrinking $V$ (from now $V$ is only a open dense subset of $g ^{-1} (U)$) if necessary,
we can assume that $h ^{-1} (x)$ is an integral smooth variety of dimension $1$
(use again part 1) of the proof).
Proceeding as in part 0) and using part 2) of the proof, 
we get 
$\mathrm{rk}\,\mathcal{H} ^{-1} h  _+ (\FF |V)
\leq  
\mathrm{rk}\, (\FF |V)=\mathrm{rk}\, (\FF)$.
By using the induction hypothesis, 
we obtain 
$\mathrm{rk}\,H ^{-d +1} p _{U +} (\mathcal{H} ^{-1} h  _+ (\FF |V))
\leq 
\mathrm{rk}\,\mathcal{H} ^{-1} h  _+ (\FF |V)$.
Using Lemma \ref{fsmoothreldimvanish}, the functors
$p _{U +} [-d+1]$ and $h _{+} [-1]$ are left exact.
Hence, we get $H ^{-d} p _{V +} (\FF |V)
\riso 
H ^{-d +1} p _{U +} (\mathcal{H} ^{-1} h  _+ (\FF |V))$
and we are done.
\end{proof}

\begin{lem}
\label{i!j!}
We suppose that $k$ is algebraically closed.
Let $X$ be a smooth irreducible curve, 
$j \colon U \hookrightarrow X$ an open immersion such that  
$Z: = X \setminus U$ is a closed point. 
Let $\FF \in F \text{-} \mathrm{Isoc} ^{\dag \dag} (U/K)$. 
Let $i \colon Z \hookrightarrow X$ be the closed immersion. 
Then, for $n=0,1$, we have the inequality
\begin{gather}
\label{i!j!1}
\dim _K H ^{n} i ^{!} j _{!} (\FF) 
\leq 
\mathrm{rk} ( \FF).
\end{gather}

\end{lem}

\begin{proof}
Let $\theta _j  \colon j _{!} (\FF ) \to j _{+} (\FF )$ be the canonical morphism. 
Let $\mathcal{C}$ be the mapping cone of $\theta _j $. By definition,
$\mathcal{H} ^{-1} (\mathcal{C})
\riso
\ker (\theta _j )$
and 
$\mathcal{H} ^{0} (\mathcal{C})
\riso
\coker (\theta _j )$.
Since $j ^{!}(\theta _j)$ is an isomorphism (use $j ^! = j ^+$), then 
$j ^{!} (\mathcal{C}) =0$. By using the triangle of localization 
$i _+ i ^! (\mathcal{C})\to \mathcal{C} \to j _+ j ^{!} (\mathcal{C}) \to +1$, 
this implies that the canonical morphism 
$i _+ i ^! (\mathcal{C})\to \mathcal{C} $ is an isomorphism, i.e .
the cone of $\theta _j$ has its support in $Z$ (see the terminology of \cite[1.3.2.(iii)]{Abe-Caro-weights}).
Since $i ^! j _{+}= 0$, since the cone of $\theta _j$ has its support in $Z$,
then by using Berthelot-Kashiwara theorem (in the form of  \cite[1.3.2.(iii)]{Abe-Caro-weights}),
we get $\ker (\theta _j )
\riso 
i _{+} H ^{0} i ^{!} j _{!}(\FF )$
and 
$\mathrm{coker} (\theta _j )
\riso 
i _{+} H ^{1} i ^{!} j _{!} (\FF )$.
From the last isomorphism of Corollary \cite[1.4.3]{Abe-Caro-weights}, 
we obtain 
$\DD _{X}i _{+} H ^{0} i ^{!} j _{!} (\FF)
\riso 
i _{+}H ^{1} i ^{!} j _{!} (\DD _{U} (\FF) )$.
Hence, 
by using again Berthelot-Kashiwara theorem (in the form of  \cite[1.3.2.(iii)]{Abe-Caro-weights}),
and the relative duality isomorphism (i.e. the isomorphism  \cite[1.3.14.(vi)]{Abe-Caro-weights}),
we get 
$ H ^{0} i ^{!} j _{!} (\FF)
\riso 
\DD _{Z} H ^{1} i ^{!} j _{!} (\DD _{U} (\FF) )$
and
we reduce to check the case $n=1$. 

We have the exact sequence
$0 \to j _{!+} (\FF) \to  j _{+} (\FF) \to i _{+} H ^{1} i ^{!} j _{!} (\FF ) \to 0$,
where $j _{!+} (\FF)$ is the intermediate extension as defined in \cite[1.4.1]{Abe-Caro-weights}, i.e. 
$ j _{!+} (\FF) $ is the image of $\theta _j$.
From Kedlaya's semistable theorem 
\cite{kedlaya_semi-stable}, there exists a finite surjective morphism
$f \colon P' \to X$, with $P'$ smooth integral, such that 
$f ^{!} (\FF)$ comes from a convergent isocrystal on $P'$ with logarithmic poles along $f ^{-1}(Z)$.
Since $P'$ and $X$ are smooth,
since $f$ is finite and surjective then $f$ is flat
(e.g. see \cite[IV.15.4.2]{EGAIV4}). 
Let $X'$ be an open dense subset of $P'$ such that $Z':= f ^{-1} (Z) \cap X'$ is a closed point. 
We get the (quasi-finite) flat morphism $a \colon X' \to X$ and the open immersion 
$j \colon U':= a ^{-1} (U) \to X'$ and $a \colon U' \to U$. 
Since $a$ is flat and quasi-finite, then $a ^{!}$ is exact. Hence, we get the exact sequence
$0 \to a ^{!} j _{!+} (\FF) \to  a ^{!} j_{+} (\FF) \to a ^{!} i _{+} H ^{1} i ^{!} j _{!} (\FF ) \to 0$.
From the base change isomorphism (e.g. see \cite[1.3.10]{Abe-Caro-weights}), 
we get 
$a ^{!} j_{+} (\FF)  \riso  j_{+}  a ^{!}(\FF) $.
Hence,
$a ^{!} j _{!+} (\FF) $ is a subobject of $j_{+}  a ^{!}(\FF) $.
Moreover, 
$j ^{!}a ^{!} j _{!+} (\FF)  \riso a ^{!}  j ^{!}j _{!+} (\FF)  \riso a ^{!} (\FF) $.
This yields, from \cite[1.4.8]{Abe-Caro-weights}, 
that the inclusion $j _{!+} (a ^{!} ( \FF))  \hookrightarrow j_{+} ( a ^{!}  (\FF))$
factors through the composition 
$a ^{!} j _{!+} (\FF)  \hookrightarrow a ^{!} j_{+} (\FF) \riso j_{+} a ^{!}  (\FF)$.
Then, we get the epimorphism
$$i _{+} H ^{1} i ^{!} j _{!} (a ^{!} \FF ) \riso 
j_{+} ( a ^{!}  (\FF)) / j _{!+} (a ^{!} ( \FF)) 
\twoheadrightarrow 
a ^{!} j_{+} (\FF) / a ^{!} j _{!+} (\FF)
\riso 
a ^{!} i _{+} H ^{1} i ^{!} j _{!} (\FF )
\underset{\tiny\cite[1.3.10]{Abe-Caro-weights}}{\riso}
i _+  H ^{1} i ^{!} j _{!} (\FF ),$$ 
where for the last isomorphism we use also that
$a $ induces the isomorphism $a \colon a ^{-1} (Z) \riso Z$ (because 
$k$ is algebraically closed).
By applying $i ^!$ (and by using Berthelot-Kashiwara theorem), 
we get 
$\dim _K H ^{1} i ^{!} j _{!} (\FF) 
\leq 
\dim _K 
H ^{1} i ^{!} j _{!} (a ^{!} \FF )$.
Since $a ^{!} \FF$ is log-extendable, then
from \cite[3.4.19.1]{Abe-Caro-weights} we obtain the inequality
$\dim _K 
H ^{1} i ^{!} j _{!} (a ^{!} \FF ) 
\leq 
\mathrm{rk} (a ^{!} \FF ) 
=
\mathrm{rk} (\FF ) .$
\end{proof}

Since the proof of the main result on Betti estimate (see \ref{BBD4.5.1}) in the case of curves is easier 
(e.g. remark that we do not need in this case the Lemma \ref{lem-loc-acyc})
and since its proof is made by induction, we first check separately this curve case via the following Proposition.

\begin{prop}
[Curve case]
\label{Betti:curvecase}
Suppose $k$ is algebraically closed. 
Let $X _1$ be a projective, smooth and connected curve, 
$\E \in F \text{-}D ^{ \leq 0}  (X _1/K)$ 
(see the notation of \cite[1.2]{Abe-Caro-weights}).
There exists a constant $c (\E)$ such that, for any finite étale morphism of degree $d _1$ of the form
$\alpha _1 \colon \widetilde{X} _1\to X _1$ with $\widetilde{X} _1$ connected, 
by putting  
$\widetilde{\E}:= \alpha _1 ^{+} (\E) $,
we have
\begin{enumerate}
\item $\dim _K H ^{1} p _{\widetilde{X} _1 +} (\widetilde{\E}) \leq  c (\E)$ ;
\item For any integer $r \leq 0$, $\dim _K H ^{r} p _{\widetilde{X} _1 +} (\widetilde{\E}) \leq  c (\E) d _1$.
\end{enumerate}
\end{prop}

\begin{proof} 
There exists an open dense affine subvariety $U _{1}$ of $X _{1}$ 
such that $\E | U _1
\in F \text{-}D ^{\mathrm{b}} _{\mathrm{isoc}} (U _{1}/K)$
(see the notation \cite[1.2.14]{Abe-Caro-weights} and use \cite[3.1.1]{caro-2006-surcoh-surcv}).  
Let $Z _1$ be the closed subvariety $X _{1} \setminus U _{1}$, 
$j \colon U _{1} \hookrightarrow X _{1}$
and 
$i \colon Z _1\hookrightarrow X _1$ be the immersions.
We put
$\widetilde{U} _1 := \alpha _{1} ^{-1} (U _1)$,
$\widetilde{Z} _1 := \alpha _{1} ^{-1} (Z _1)$ i.e we get the cartesian squares:

\begin{equation}
\notag
\xymatrix{
{\widetilde{Z} _1} 
\ar[r] ^-{i}
\ar@{}[rd] ^-{}|\square
\ar[d] ^-{\alpha _1}
&
{\widetilde{X} _1} 
\ar[d] ^-{\alpha _1}
\ar@{}[rd] ^-{}|\square
&
{\widetilde{U} _1} 
\ar[d] ^-{\alpha _1}
\ar[l] _-{j}
\\ 
{Z _1}  
\ar[r] ^-{i}
& 
{X _1} 
& 
{U _1.}
\ar[l] _-{j}
}
\end{equation} 
 
By considering the exact triangle 
$j _{!} j ^{!} (\E) \to \E \to i _{+} i ^{+} (\E) \to +1$ (see \cite[1.1.8.(ii)]{Abe-Caro-weights})
we reduce to check the proposition 
for  
$\E= j _{!} j ^{!} (\E) $ or 
$\E= i _{+} i ^{+} (\E)$ (and because 
the functors 
$j _{!} j ^{!}$ and 
$i _{+} i ^{+}$ preserve
$D ^{ \leq 0}$).

1) In the case where $\E= i _{+} i ^{+} (\E)$, we can suppose that $Z _1$ is a point.
We put
$\G:= i ^{+} (\E)$. 
Since $\widetilde{Z} _1$ is $d _{1}$ copies of $Z _1$, 
then we get 
$\alpha _{1+}\alpha _1 ^{+} \G \riso \G ^{d _1}$.
Since $i _+ \riso i _!$ and $\alpha _1 ^+ \riso \alpha _1 ^!$, then we get from the base change isomorphism
(e.g. see \cite[1.3.10]{Abe-Caro-weights}):
$\alpha _1 ^+ i _+ \riso i _+ \alpha _1 ^+$ .
Since $\E= i _{+} i ^{+} (\E)$ , 
this implies 
$ \alpha _{1 +} (\widetilde{\E})
= 
\alpha _{1+ } \alpha _1 ^+ (\E)
\riso 
i _+  \alpha _{1+ } \alpha _1 ^+ i ^+ (\E)
=
i _+  \alpha _{1+ } \alpha _1 ^+ (\G)
\riso 
i _+   \G ^{d _1}$.
Hence, 
$p _{\widetilde{X} _1 +} (\widetilde{\E}) \riso 
p _{X _1 +}  (  i _{+} \G ^{d _1})
\riso 
\G ^{d _1}$, which gives the desired result.

2) Suppose now 
$\E= j _{!} j ^{!} (\E) $.
We put 
$\FF= j ^{!} (\E) $, 
$\widetilde{\FF}= j ^{!} (\widetilde{\E})$.
Using the spectral sequence
$E _{2} ^{r,s} = H ^{r} p _{\widetilde{X} _{1}+}
(\mathcal{H} ^{s}  ( \widetilde{\E}))
\Rightarrow 
H ^{r+s} p _{\widetilde{X} _1+}  ( \widetilde{\E})$,
we reduce to the case where 
$\FF \in F \text{-}\mathrm{Isoc} ^{\dag \dag} (U _{1}/K)$
(and then 
$\widetilde{\FF} \in F \text{-}\mathrm{Isoc} ^{\dag \dag} (\widetilde{U} _{1}/K)$).
Since $X _1$ is proper and smooth integral of dimension $1$, 
then 
$p _{\widetilde{X} _1 +} (\widetilde{\E})
\riso 
p _{\widetilde{X} _1 !} (\widetilde{\E})$
is a complex concentrated in degree $-1,0,1$ (use Lemma \ref{fsmoothreldimvanish}).
Since $U _1$ is affine, then $p _{\widetilde{U} _1!}$
is left $t$-exact
(see \cite[1.3.13.(i)]{Abe-Caro-weights}).
Since 
$\widetilde{\E}\riso j _{!} (\widetilde{\FF})$, 
then $p _{\widetilde{X} _1 +} (\widetilde{\E})
\riso 
p _{\widetilde{X} _1 !} (\widetilde{\E})
\riso 
p _{\widetilde{U} _1!} (\widetilde{\FF})$.
Hence, 
we get
$H ^{-1} p _{\widetilde{X} _1 +} (\widetilde{\E})=0$.
From \ref{estiH-d!}, we get
$\dim _K H ^{1} p _{\widetilde{X} _1 +} (\widetilde{\E})
=
\dim _K H ^{1} p _{\widetilde{U} _1!} (\widetilde{\FF})
\leq 
\mathrm{rk} (\widetilde{\FF})
=
\mathrm{rk} (\FF)$.
It remains to estimate 
$|\chi ( \widetilde{X} _1, \widetilde{\E})|$.
Since $p _{\widetilde{X} _1 +} (\widetilde{\E}) \riso 
p _{X _1 +} (\alpha_{1 +} \alpha_1 ^{+} (\E))$, we get the equality
$\chi ( \widetilde{X} _1, \widetilde{\E}) 
=
\chi (X_1, \alpha_{1 +} \alpha_1 ^{+} (\E))$.
From Lemma \ref{chi=d-times-chi}
(recall also that from \cite[III.6.10]{sga1}, 
there exist some smooth proper formal $\V$-schemes $\X _1 $ and $\widetilde{\X} _1$ which are respectively a lifting of 
$X _1$ and $\widetilde{X} _1$, so we are in the geometrical context of Lemma \ref{chi=d-times-chi}),
we have the formula
$\chi (X_1, \alpha_{1 +} \alpha_1 ^{+} (\E))
= d _1\cdot \chi (X _1,\E ).$
Hence, we can choose in that case 
$c (\E) = \max \{| \chi (X _1,\E )| ;  \mathrm{rk} (\FF)\}$. 

\end{proof}

\subsection{The result and some applications}
In this subsection, we will need the Fourier transform (see \ref{Fourier}). 
Hence, we assume here that there exists  $\pi _0\in K$ such that $\pi _0^{p-1}=-p$ and such that $\sigma (\pi _0) =\pi _0$. 
Let $q := p^s$, $W (\F _q)$ be the ring of Witt vectors of $\F _q$.  
We get a complete discrete valuation ring of residue field $\F _q$ by setting
$\V _0 := W (\F _q) [X] /(X ^{p-1} +p)$  (indeed, 
$X ^{p-1} +p$ is an Eisenstein polynomial). 
The class of $X$ is a uniformiser of $\V _0$. 
Remark that the canonical lifting of the $s$th Frobenius power of $\F _q$ is the identity of $\V _0$.
We define the extension $\rho \colon \V _0 \to \V$ by sending the class of $X$ to $\pi _0$. 
Since $\sigma (\pi _0) =\pi _0$, this homomorphism $\rho$ is compatible with Frobenius liftings i.e. 
$\sigma \circ \rho = \rho $. 
Let $K _0$ be the field of fraction of $\V _0$. 
We know the $K _0 = \mathrm{Frac} (W (\F _q)) ( \mu _p)$, where 
$\mu _p \subset \overline{\Q} _p$ are the group of $p$-rooth of unity (see \cite[1.3]{MR773087}). 

We fix a non trivial additive character $\psi \colon \F _q \to \mu _p \subset K _0$.

\begin{empt}
\label{Artin-Shreier}
We denote by $\mathcal{L}_\psi$ the Artin-Schreier
isocrystal in $F\text{-}\mathrm{Isoc}^{\dag\dag}(\mathbb{A} _{\F _q}  ^1/K _0)$
(see Proposition \cite[1.5]{MR773087}). 
The extension $\V _0 \to \V$ 
induces the morphism $\widehat{\P} ^{1} _{\V} \to \widehat{\P} ^{1} _{\V _0}$.
We obtain the morphism of ringed spaces $ f \colon
( \widehat{\P} ^{1} _{\V} , \O _{\widehat{\P} ^{1} _{\V}} (\hdag \infty) _\Q)
\to 
(\widehat{\P} ^{1} _{\V _0} , \O _{\widehat{\P} ^{1} _{\V _0}}(\hdag \infty) _\Q)$, 
where $\infty$ is the closed point $\P ^1 _{\F _q} \setminus \A ^1 _{\F _q}$
and respectively
$\P ^1 _{k} \setminus \A ^1 _{k}$.
Recall (see the convention of the paper) that
$F\text{-}\mathrm{Isoc}^{\dag\dag}(\mathbb{A} _{\F _q}  ^1/K _0)$ 
($F\text{-}\mathrm{Isoc}^{\dag\dag}(\mathbb{A} _{k}  ^1/K )$ )
is equivalent to the category of 
coherent 
$\O _{\widehat{\P} ^{1} _{\V _0}} (\hdag \infty) _\Q$-modules $\E$ 
(resp. coherent 
$\O _{\widehat{\P} ^{1} _{\V}} (\hdag \infty) _\Q$-modules $\E$) 
endowed with an integrable connexion and 
a Frobenius structure  i.e. an isomorphism of the form $F ^* (\E) \riso \E$.
Hence, we get  the functor
$f ^* \colon F\text{-}\mathrm{Isoc}^{\dag\dag}(\mathbb{A} _{\F _q}  ^1/K _0)
\to 
F\text{-}\mathrm{Isoc}^{\dag\dag}(\mathbb{A} _{k}  ^1/K )$.
We still denote by  $\mathcal{L}_\psi$ the 
object of $F\text{-}\mathrm{Isoc}^{\dag\dag}(\mathbb{A} _{k}  ^1/K )$ which is the image by $f ^*$ 
of the Artin-Schreier
isocrystal $\mathcal{L}_\psi$ in $F\text{-}\mathrm{Isoc}^{\dag\dag}(\mathbb{A} _{\F _q}  ^1/K _0)$.

\end{empt}

\begin{empt}
[Fourier transform]
\label{Fourier}
Let $S$ be a $k$-variety. 
Let us briefly review the geometric Fourier transform defined by Noot-Huyghe in 
\cite{{Noot-Huyghe-fourierI}}, but only in the specific case of $\A ^{1} _S /S$.
Let $\mu  \colon \mathbb{A} _k ^1 \times (\mathbb{A} _k ^1)' \to
\mathbb{A} _k ^1$ be the canonical duality bracket given by $t\mapsto  x y$, 
where $(\mathbb{A}_S^1)'$ is the ``dual affine space over $S$'', which is
nothing but $\mathbb{A}^1_S$ (we have 
$\mathbb{A}^1_S =\Spec \O _S [x]$ and $(\mathbb{A}^1_S )'=\Spec \O _S [y]$). 
We denote the composition by
$\mu _S \colon \mathbb{A} _S ^1 \times _S (\mathbb{A} _S ^1)' \to
\mathbb{A} _k ^1 \times (\mathbb{A} _k ^1)' \to
\mathbb{A} _k ^1$.

 Now, consider the following diagram:
$(\mathbb{A} _S ^1)'\xleftarrow{p_2}\mathbb{A} _S ^1 \times _S (\mathbb{A} _S ^1)' \xrightarrow{p_1}\mathbb{A} _S^1$.
Similarly to Katz and Laumon in \cite[7.1.4, 7.1.5]{KatzLaumon} (in fact, here is the particular case where $r=1$),
for any $\E\in
F\text{-}D^{\mathrm{b}}_{\mathrm{ovhol}}(\mathbb{A} _S^1)$, the geometric
Fourier transform $\mathscr{F}_\psi(\E)$ is defined to
be 
\begin{equation}
\label{Fourierdef}
\mathscr{F}_{\psi}(\E) := p_{2+}\bigl(p_1^!\E \widetilde{\otimes}_{\mathbb{A}^{2} _S}
\mu _S ^{!} \mathcal{L}_\psi [-1]\bigr)
\end{equation}
(cf.\
\cite[3.2.1]{Noot-Huyghe-fourierI}\footnote{Notice that our twisted
tensor product and hers are the same.}).
Here $\widetilde{\otimes}$ is compatible with Laumon's notation (see \cite[7.0.1, page 192]{KatzLaumon})
and was defined in the context of arithmetic $\D$-modules in \cite[1.1.6]{Abe-Caro-weights}.

\end{empt}

\begin{empt}
\label{Fourier1} 
An important property for us of Fourier transform is the following. 
The functor $\mathscr{F} _{\psi}[1]$ is acyclic, i.e. 
if $\E\in
F\text{-}\mathrm{Ovhol}(\mathbb{A} _S^1/K)$ then 
$\mathscr{F} _{\psi}(\E) [1]\in F\text{-}\mathrm{Ovhol}((\mathbb{A} _S^1) ^{\prime}/K)$
(cf.\ \cite[Theorem 5.3.1]{Noot-Huyghe-fourierI}).

Remark: this might be simpler in our case to define the Fourier transform by setting 
$\mathscr{F}_{\psi}(\E) := p_{2+}\bigl(p_1^!\E \widetilde{\otimes}_{\mathbb{A}^{2} _S}
\mu _S ^{!} \mathcal{L}_\psi\bigr)$ 
(and then $\mathscr{F} _{\psi}(\E)\in F\text{-}\mathrm{Ovhol}((\mathbb{A} _S^1) ^{\prime}/K)$) but, 
to avoid confusion with the standard notation, we stick with the convention 
of \cite[3.2.2]{AbeMarmora} or 
\cite[3.2.1]{Noot-Huyghe-fourierI} (for this latter reference, remark that there is a typo in 
\cite[3.1.1]{Noot-Huyghe-fourierI} : $K _\pi = \delta ^* L _\pi [2N-2]$ and not $K _\pi = \delta ^* L _\pi [2-2N]$).
\end{empt}

\begin{lemm}
\label{Fourierf^!f_+}
Let $f \colon T \to S$ be a morphism of $k$-varieties. 
Let $\E\in
F\text{-}D^{\mathrm{b}}_{\mathrm{ovhol}}(\mathbb{A} _S^1 /K)$
and 
$\FF\in
F\text{-}D^{\mathrm{b}}_{\mathrm{ovhol}}(\mathbb{A} _T ^1 /K)$.
We have the canonical isomorphisms
\begin{gather} 
\label{Fourierf^!}
f ^{!} \mathscr{F} _{\psi} (\E)
\riso 
\mathscr{F} _{\psi} (f ^{!} \E) ; 
\\
\label{Fourierf_+}
f _{+} \mathscr{F} _{\psi} (\FF)
\riso 
\mathscr{F} _{\psi} (f _{+} \FF).
\end{gather}

\end{lemm}

\begin{proof}
We have the canonical isomorphisms:
$$f ^{!} \mathscr{F} _{\psi} (\E)
=
f ^{!}  p_{2+}\bigl(p_1^!\E \widetilde{\otimes}_{\mathbb{A}^{2} _S}
\mu _S ^{!} \mathcal{L}_\psi [-1]\bigr)
\underset{\tiny\cite[1.3.10]{Abe-Caro-weights}}{\riso}
p_{2+} f ^{!}  \bigl(p_1^!\E \widetilde{\otimes}_{\mathbb{A}^{2} _S}
\mu _S ^{!} \mathcal{L}_\psi [-1]\bigr)
\underset{\tiny\cite[1.1.9.1]{Abe-Caro-weights}}{\riso}
p_{2+} \bigl( f ^{!}  p_1^!\E \widetilde{\otimes}_{\mathbb{A}^{2} _T}
f ^{!}  \mu _S ^{!} \mathcal{L}_\psi [-1]\bigr).
$$
Since $\mu _T = \mu _S \circ f $, by transitivity of the extraordinary inverse image, 
we obtain the isomorphism
$f ^{!}  \mu _S ^{!} \mathcal{L}_\psi [-1] \riso  \mu _T ^{!} \mathcal{L}_\psi [-1]$. 
Hence, 
$p_{2+} \bigl( f ^{!}  p_1^!\E \widetilde{\otimes}_{\mathbb{A}^{2} _T}
f ^{!}  \mu _S ^{!} \mathcal{L}_\psi [-1]\bigr)
\riso 
p_{2+} \bigl( p_1^! (f ^{!}  \E )\widetilde{\otimes}_{\mathbb{A}^{2} _T}
 \mu _T ^{!} \mathcal{L}_\psi [-1]\bigr)
 = \mathscr{F} _{\psi} (f ^{!} \E)$, which gives \ref{Fourierf^!}.
 Moreover, by transitivity of the push-forward, we get the first isomorphism: 
\begin{gather}
\notag
f _{+} \mathscr{F} _{\psi} (\FF)=
f _{+}  p_{2+} \bigl( p_1^! (\FF )\widetilde{\otimes}_{\mathbb{A}^{2} _T}
 \mu _T ^{!} \mathcal{L}_\psi [-1]\bigr)
 \riso
 p_{2+}   f _{+}  \bigl( p_1^! (\FF )\widetilde{\otimes}_{\mathbb{A}^{2} _T}
 f ^{!}  \mu _S ^{!}  \mathcal{L}_\psi [-1]\bigr)
\underset{\tiny\cite[A.6]{Abe-Caro-weights}}{\riso}
   p_{2+}    \bigl(  f _{+}  p_1^! (\FF )\widetilde{\otimes}_{\mathbb{A}^{2} _S}
\mu _S ^{!}  \mathcal{L}_\psi [-1]\bigr)
\\ \notag
\underset{\tiny\cite[1.3.10]{Abe-Caro-weights}}{\riso}
p_{2+}    \bigl(  p_1^! (f _{+}  \FF )\widetilde{\otimes}_{\mathbb{A}^{2} _S}
\mu _S ^{!}  \mathcal{L}_\psi [-1]\bigr)
=
\mathscr{F} _{\psi} (f _{+} \FF).
\end{gather}

\end{proof}

We will use the following remark during the proof of the main theorem.
\begin{rem}
Let $\mathcal{R} _K$ be the Robba ring over $K$ (e.g. 
see \cite[15.1.4]{Kedlaya-padicDiffEq}).
Let  $M$ be a differential module on $\mathcal{R} _K$, i.e. 
a free $\mathcal{R} _K$-module of finite type endowed with an integrable connexion 
(e.g. see the beginning of \cite[3]{christol-Mebkhout2002}).
We suppose $M$ solvable (e.g. see Definition \cite[8.7]{christol-Mebkhout2002}).
We get the differential slope decomposition 
$M = \oplus M _{\beta}$, where $M _\beta$
 is purely of differential  slope $\beta$
 (see Theorem \cite[2.4-1]{Christol-Mebkhout_IV}).
 By definition
 $\mathrm{Irr} (M) := \sum _{\beta \geq 0} \beta \cdot\mathrm{rk} ( M _\beta)$ (see both definitions in \cite[2.3.1]{AbeMarmora} and 
 the formula \cite[2.3.2.2]{AbeMarmora} which compare both definitions).
 Hence, we get that
\begin{equation}
\label{rem-Irr-esti-rk}
\mathrm{Irr} (M) \leq \mathrm{rk} (M) + \mathrm{Irr} ( M _{]1, \infty[}),
\end{equation}
  where $M _{]1, \infty[}:=  \oplus _{\beta \in ]1, \infty[} M _{\beta}$.

\end{rem}

\begin{thm}
\label{BBD4.5.1}
Suppose $k$ is algebraically closed. 
Let $(X _a) _{1 \leq a \leq n}$ be projective, smooth and connected curves, 
$X= \prod _{a=1} ^{n} X _a$, 
$\E \in F \text{-}D ^{ \leq 0}  (X/K)$ 
(see the notation of \cite[1.2]{Abe-Caro-weights}).
There exists a constant $c (\E)$ such that, for any finite étale morphism of degree $d _{a}$ of the form
$\alpha _{a} \colon \widetilde{X} _{a} \to X _{a}$ with $\widetilde{X} _{a}$ connected, 
by putting  
$\widetilde{X}= 
\prod _{a=1} ^{n} \widetilde{X} _a$, 
$\alpha \colon \widetilde{X} \to X$
and 
$\widetilde{\E}:= \alpha ^{+} (\E) $,
we have
\begin{enumerate}
\item For any integer $r$, $\dim _K H ^{r} p _{\widetilde{X} +} (\widetilde{\E}) \leq  c (\E) \prod _{a=1} ^{n} d _{a}$.
\item For any integer $r\geq 1$, $\dim _K H ^{r} p _{\widetilde{X} +} (\widetilde{\E}) \leq  c (\E)\max \{ \prod _{a\in A} d _{a} \, | \, A \subset \{ 1, \dots, n\} \text{ and } |A| = n-r \} $.
\end{enumerate}
\end{thm}

\begin{proof}
We proceed by induction on $n \in \N$.
The case $n=1$ has already been checked in \ref{Betti:curvecase}.
Suppose $n \geq 2$.
Let  
$\alpha _{a} \colon \widetilde{X} _{a} \to X _{a}$ 
be some finite étale morphism of degree $d _{a}$ with $\widetilde{X} _{a}$ connected, 
$\widetilde{X}= 
\prod _{a=1} ^{n} \widetilde{X} _a$,
$\alpha \colon \widetilde{X} \to X$
and 
$\widetilde{\E}:= \alpha ^{+} (\E) $.
We put 
$Y:= \prod _{a \not = 1} X _{a}$, 
$\widetilde{Y}:= \prod _{a \not = 1} \widetilde{X} _{a}$, 
$\beta \colon \widetilde{Y} \to Y$.
Let 
$\mathrm{pr}
\colon 
Y \to \Spec k $
and 
$\widetilde{\mathrm{pr}}
\colon 
\widetilde{Y} \to \Spec k $ 
be the projections
(recall that from the convention of this paper, for instance, 
$\widetilde{\mathrm{pr}}$ means also the projection $\widetilde{\mathrm{pr}}
\colon 
\widetilde{X}= \widetilde{X} _1 \times \widetilde{Y} \to \widetilde{X} _1$ etc.). 
From Lemma \ref{lem-loc-acyc} (recall also that from \cite[III.6.10]{sga1}, since $X _i$ and $\widetilde{X} _i$ are 
smooth curves, 
there exist some smooth proper formal $\V$-schemes $\X _i $ and $\widetilde{\X} _i$ which lift them, which reduce us to the geometrical situation of 
Lemma \ref{lem-loc-acyc}) 
there exists an affine open dense subvariety $U _{1}$ 
(independent of the choice of $\alpha _i$) of $X _{1}$ 
such that $\mathrm{pr} _{+} (\alpha _+ \widetilde{\E}) | U _{1} 
\in D ^{\mathrm{b}} _{\mathrm{isoc}} (U _{1}/K)$ (use \cite[2.2.12]{caro_courbe-nouveau} to check this latter property). 
Let $Z _1$ be the closed subvariety $X _{1} \setminus U _{1}$, 
$\widetilde{U} _1: =\alpha ^{-1} _{1} (U _1)$,
$\widetilde{Z} _1 := \alpha _{1} ^{-1} (Z _1)$.
Let $j \colon U _1\hookrightarrow X_1$ 
$i \colon Z _1\hookrightarrow X _1$ be the inclusions.
\medskip 

\noindent Step 0)
We have
$j ^{!}\widetilde{\mathrm{pr}} _{+} (\widetilde{\E})  =\widetilde{\mathrm{pr}} _{+} (\widetilde{\E}) | \widetilde{U} _{1} 
\in F \text{-} D ^{\mathrm{b}} _{\mathrm{isoc}} (\widetilde{U} _{1}/K)$. 
Indeed, from \ref{isoc-desc-finiteetale}, this is equivalent to 
prove
$\alpha _{1+}(\widetilde{\mathrm{pr}} _{+} (\widetilde{\E}) ) | U _{1}\in D ^{\mathrm{b}} _{\mathrm{isoc}} (U _{1}/K)$.
Then, we get the desired property from the isomorphism
$\alpha _{1+}(\widetilde{\mathrm{pr}} _{+} (\widetilde{\E}) )
\riso 
\mathrm{pr} _{+} (\alpha _+ \widetilde{\E})$ (checked by transitivity of the push-forwards).

\medskip 

\noindent Step I) With the notation \ref{ntn-rk}, we check 
that there exists a constant 
$c $ (only depending on $\E$) such that
\begin{itemize}
\item For any integer $s$, $\mathrm{rk} \,  \mathcal{H} ^{s} \widetilde{pr} _{+} ( \widetilde{\E}) \leq  c  \prod _{b=2} ^{n} d _{b}$.
\item For any integer $s\geq 1$, $\mathrm{rk} \,\mathcal{H} ^{s} \widetilde{pr} _{+} ( \widetilde{\E})
\leq  c \max \{ \prod _{b\in B} d _{b} \, | \, B \subset \{ 2, \dots, n\} \text{ and }|B| = n-1-s \} $.
\end{itemize}

{\it Proof.} Let $t$ be a closed point of $U _1$, $\widetilde{t}$ be a closed point of $\widetilde{U} _1$ such that
$\alpha _1 (\widetilde{t}) =t$. Let $i _t \colon t \hookrightarrow X _1$, 
$i _{\widetilde{t}} \colon \widetilde{t} \hookrightarrow \widetilde{X} _{1}$,
$\iota _{\widetilde{t}} \colon \widetilde{t} \hookrightarrow \widetilde{U} _{1}$ be the closed immersions.
Since the functor
$\iota _{\widetilde{t}} ^{!} [1]$ is acyclic on 
$F \text{-}D ^{\mathrm{b}} _{\mathrm{isoc}} (\widetilde{U} _{1}/K)$,
since $i _{\widetilde{t}} ^{!} [1] \riso (\iota _{\widetilde{t}} ^{!} [1]) \circ j ^{!}$, 
we obtain 
$i _{\widetilde{t}} ^{!} [1] (\mathcal{H} ^{s} \widetilde{pr} _{+} ( \widetilde{\E}))
\riso
\mathcal{H} ^{s} ( i _{\widetilde{t}} ^{!} \widetilde{pr} _{+} ( \widetilde{\E} [1] ))$.
Moreover, for such $\widetilde{t}$, we have
$\mathrm{rk} (\mathcal{H} ^{s} \widetilde{pr} _{+} ( \widetilde{\E}))
=
\dim _K i _{\widetilde{t}} ^{!} [1] \mathcal{H} ^{s} \widetilde{pr} _{+} ( \widetilde{\E})$.

We put $\E _{1}:=  i _{t}^{!}(\E) [1]$ and $\widetilde{\E} _1 := \beta ^{+} (\E _1)$.
Since $t\times Y$ is a smooth divisor of $X$, then $i _{t}^{!}[1]$ is right exact.
Hence $\E _{1} \in F \text{-}D ^{ \leq 0}  (Y/K)$ (we identify $Y$ with $t\times Y$).
To simplify the notation, 
we avoid to mention the isomorphism $\widetilde{t} \riso t$ induced by $\alpha _1$
in other words, we identify $t$ and $\widetilde{t}$ via this isomorphism). 
We get 
$i _t = \alpha _1 \circ i _{\widetilde{t}}$
and then 
$ \alpha \circ  i _{\widetilde{t}} = i _{t} \circ \beta \colon \widetilde{t} \times \widetilde{Y} \to X$.
Since $ \beta ^{+}=  \beta ^{!}$, 
we get by transitivity of the extraordinary inverse image the isomorphism
\begin{equation}
\label{EtoE1}
i _{\widetilde{t}}^{!}\widetilde{\E} [1]
 =
  i _{\widetilde{t}}^{!} \alpha ^{+} (\E) [1]
\riso
 \beta ^{+}  i _{t}^{!}(\E) [1]
=\beta ^{+} (\E _1)= 
\widetilde{\E} _1 .
\end{equation}
 Hence, 
\begin{equation}
\label{EtoE1bis}
i _{\widetilde{t}}^{!} \widetilde{pr} _ {+} (\widetilde{\E})[1] 
\underset{\cite[1.3.10]{Abe-Caro-weights}}{\riso}
 \widetilde{pr} _ {+} i _{\widetilde{t}}^{!} (\widetilde{\E})[1] 
\underset{\ref{EtoE1}}{\riso} 
 \widetilde{pr} _ {+} 
 (\widetilde{\E} _1)
 =
p _{\widetilde{Y}+}( \widetilde{\E} _1) ,
\end{equation}
where we have identified $\widetilde{Y}$ with 
$\widetilde{t} \times \widetilde{Y}$ in the last equality.
By composition, we obtain
$$
i _{\widetilde{t}} ^{!} [1] (\mathcal{H} ^{s} \widetilde{pr} _{+} ( \widetilde{\E}))
\riso
\mathcal{H} ^{s} ( i _{\widetilde{t}} ^{!} \widetilde{pr} _{+} ( \widetilde{\E} [1] ))
\underset{\ref{EtoE1bis}}{\riso} 
\mathcal{H} ^{s}
p _{\widetilde{Y}+}( \widetilde{\E} _1) $$
and then 
$\mathrm{rk} (\mathcal{H} ^{s} \widetilde{pr} _{+} ( \widetilde{\E}))
=
\dim _K 
\mathcal{H} ^{s}
p _{\widetilde{Y}+}( \widetilde{\E} _1)$.
We conclude by applying the induction hypothesis to $\E _1$ (notice that we do need Theorem \ref{lem-loc-acyc} : 
since $U _1$ is independent of the choice of $\alpha _i$ then so is $t$ and then $\E _1$).

\medskip

\noindent Step II) 

\noindent 1) By considering the exact triangle 
$j _{!} j ^{!} (\E) \to \E \to i _{+} i ^{+} (\E) \to +1$ (see \cite[1.1.8.(ii)]{Abe-Caro-weights})
we reduce to check the proposition 
for  
$\E= j _{!} j ^{!} (\E) $ or 
$\E= i _{+} i ^{+} (\E)$ (and because 
the functors 
$j _{!} j ^{!}$ and 
$i _{+} i ^{+}$ preserve
$D ^{ \leq 0}$).

\noindent 2) Suppose $\E= i _{+} i ^{+} (\E)$.
We can suppose that $Z _1$ is irreducible (i.e. since 
 $k$ is algebraically closed, $Z _1 \riso \Spec k$).
Consider the diagram with cartesian squares:
\begin{equation}
\notag
\xymatrix{
{\widetilde{Z} _1 \times \widetilde{Y}} 
\ar[r] ^-{\alpha _1}
\ar@{}[rd] ^-{}|\square
\ar@{^{(}->}[d] ^-{i}
&
{Z _1 \times \widetilde{Y}} 
\ar[r] ^-{\beta}
\ar@{^{(}->}[d] ^-{i}
\ar@{}[rd] ^-{}|\square
&
{Z _1 \times Y}
\ar[r] ^-{\mathrm{pr}}
\ar@{^{(}->}[d] ^-{i}
\ar@{}[rd] ^-{}|\square
& 
{Z _1} 
\ar@{^{(}->}[d] ^-{i}
\\ 
{\widetilde{X}}  
\ar[r] ^-{\alpha _1}
& 
{X _1 \times \widetilde{Y}} 
\ar[r] ^-{\beta}
& 
{X }
\ar[r] ^-{\mathrm{pr}}
& 
{X_1 .}
}
\end{equation}
We put
$\G:= i ^{+} (\E)$, 
$\widetilde{\G}:= \beta ^{+} (\G)$.
Since $\widetilde{Z} _1$ is $d _{1}$ copies of $Z _1$, 
then we get 
$\alpha _{1+}\alpha _1 ^{+}  \widetilde{\G} \riso (\widetilde{\G} )^{d _1}$.
This implies
\begin{equation}
\label{induc-hyp1}
p _{\widetilde{X} +} (\widetilde{\E})
=
p _{\widetilde{X} +} (\alpha ^{+} ( i _{+} \G ) )
\underset{\cite[1.3.10]{Abe-Caro-weights}}{\riso}
p _{\widetilde{X} +} ( i _{+} \alpha _1 ^{+}   \beta ^{+} \G)
\riso 
p _{\widetilde{Y} +}  \alpha _{1+}\alpha _1 ^{+}   (\widetilde{\G})
\riso 
(p _{\widetilde{Y} +} (\widetilde{\G})) ^{d _1},
\end{equation}
where in the second isomorphism we have identified 
$\widetilde{Y}$ with $Z _1 \times \widetilde{Y}$.
We conclude by applying the induction hypothesis to 
$\G$.
\medskip 

\noindent 3) Suppose now 
$\E= j _{!} j ^{!} (\E) $.

a) We check that
there exists a constant 
$c $ (only depending on $\E$) such that
\begin{itemize}
\item For any $s$, $ \dim _K H ^{1} p _{\widetilde{X} _{1}+} ( \mathcal{H} ^{s}\widetilde{pr} _{+} ( \widetilde{\E})) \leq  c  \prod _{b=2} ^{n} d _{b}$.
\item For any $s\geq 1$, $\dim _K H ^{1} p _{\widetilde{X} _{1}+} ( \mathcal{H} ^{s}\widetilde{pr} _{+} ( \widetilde{\E}))
\leq  c \max \{ \prod _{b\in B} d _{b} \, | \, B \subset \{ 2, \dots, n\} \text{ and }|B| = n-1-s \} $.
\end{itemize}

Proof: We put 
$\FF= j ^{!} (\E) $, 
$\widetilde{\FF}:= \alpha ^{+} (\FF)$. 
By transitivity of the extraordinary push-forward,
we get the first isomorphism
\begin{equation}
\label{StepII.3iso}
j _{!}  \widetilde{pr} _{+} ( \widetilde{\FF})
\riso 
\widetilde{pr} _{+} j _! ( \widetilde{\FF})
\underset{\cite[1.3.10]{Abe-Caro-weights}}{\riso} 
\widetilde{pr} _{+} ( \widetilde{\E}).
\end{equation}
 Moreover, since $j ^! j _! \riso Id$, we get the first isomorphism
 $ \widetilde{pr} _{+} ( \widetilde{\FF})
\riso 
j ^! j _! \widetilde{pr} _{+} ( \widetilde{\FF})
\underset{\ref{StepII.3iso}}{\riso} 
j ^{!} \widetilde{pr} _{+} (\widetilde{\E}) $.
Hence, from the Step 0, this implies
 $ \widetilde{pr} _{+} ( \widetilde{\FF})
\in D ^{\mathrm{b}} _{\mathrm{isoc}} (\widetilde{U} _{1}/K)$.

Since $j _!$ is exact, we get from \ref{StepII.3iso} the isomorphism
$j _{!}  \mathcal{H} ^{s}  \widetilde{pr} _{+} ( \widetilde{\FF})
\riso 
 \mathcal{H} ^{s}\widetilde{pr} _{+} ( \widetilde{\E})$.
 By applying the functor 
 $p _{\widetilde{X} _{1}+}$ to this last isomorphism,
 we get by transitivity of the extraordinary push-forward 
 \begin{equation}
 \label{StepII.3.1so2}
 p _{\widetilde{U} _{1}!} ( \mathcal{H} ^{s}\widetilde{pr} _{+} ( \widetilde{\FF}))
\riso
p _{\widetilde{X} _{1}+} ( \mathcal{H} ^{s}\widetilde{pr} _{+} ( \widetilde{\E})).
 \end{equation}
Applying $H ^{1} $ to \ref{StepII.3.1so2}, we get the first equality: 
$$\dim _K H ^{1} p _{\widetilde{X} _{1}+} ( \mathcal{H} ^{s}\widetilde{pr} _{+} ( \widetilde{\E}))
= 
\dim _K H ^{1} p _{\widetilde{U} _{1}!} ( \mathcal{H} ^{s}\widetilde{pr} _{+} ( \widetilde{\FF}))
\leq 
\mathrm{rk} 
 \mathcal{H} ^{s}\widetilde{pr} _{+} ( \widetilde{\FF})
 =
\mathrm{rk} 
 \mathcal{H} ^{s}\widetilde{pr} _{+} ( \widetilde{\E}),$$
the inequality in the middle is a consequence of \ref{estiH-d!}.
From the step I), we get the desired estimate.

b) 
We have the spectral sequence 
\begin{equation}
\notag
E _{2} ^{r,s} = H ^{r} p _{\widetilde{X} _{1}+}
\mathcal{H} ^{s} \widetilde{pr} _{+} ( \widetilde{\E})
\Rightarrow 
H ^{r+s} p _{\widetilde{X}+}  ( \widetilde{\E}).
\end{equation}
Since $\widetilde{U} _{1}$ is affine of dimension $1$, 
using the isomorphism 
\ref{StepII.3.1so2}, 
we get
$E _{2} ^{r,s}= 0$ when $r \not \in \{ 0,1\}$ (use also \ref{fsmoothreldimvanish} and the respective case of 
\cite[1.3.13.(i)]{Abe-Caro-weights}).
Hence, by using the step a) (still valid if we vary the order of 
$X _1,\dots, X _n$),
it remains to check that there exists a constant $c (\E)$ such that 
\begin{itemize}
\item For any $s$, $|\chi ( \widetilde{X} _1,  \mathcal{H} ^{s}\widetilde{pr} _{+} ( \widetilde{\E})) |\leq  c (\E) \prod _{a=1} ^{n} d _{a}$.
\item For any $s\geq 1$, $|\chi ( \widetilde{X} _1,  \mathcal{H} ^{s}\widetilde{pr} _{+} ( \widetilde{\E}))|
\leq  c (\E)\max \{ \prod _{a\in A} d _{a} \, | \, A \subset \{ 1, \dots, n\} \text{ and }  |A| = n-s \} $.
\end{itemize}

i) In this step, we reduce to the case where $\alpha _1$ is the identity.
For this purpose, consider the following diagram
\begin{equation}
\notag
\xymatrix{
{\widetilde{X} _1 \times \widetilde{Y}} 
\ar[r] ^-{\alpha _1}
\ar[d] ^-{\widetilde{pr}}
\ar@{}[rd] ^-{}|\square
&
{X _1 \times \widetilde{Y}}
\ar[r] ^-{\beta}
\ar[d] ^-{\widetilde{pr}}
& 
{X} 
\ar[ld] ^-{\mathrm{pr}}
\\ 
{\widetilde{X} _1} 
\ar[r] ^-{\alpha _1}
& 
{X _1 .}
& 
{}
}
\end{equation}
We have 
$\chi ( \widetilde{X} _1,  \mathcal{H} ^{s}\widetilde{pr} _{+} ( \widetilde{\E}))
=
\chi ( X _1, \alpha _{1+} \mathcal{H} ^{s}\widetilde{pr} _{+} ( \widetilde{\E}))$.
By using the transitivity of the pull-back, 
we get the first isomorphism: 
$\alpha _{1+} \widetilde{pr} _{+} ( \widetilde{\E})
\riso
\alpha _{1+} \widetilde{pr} _{+} \alpha _{1} ^{+} ( \beta ^{+} (\E))
\underset{\cite[1.3.10]{Abe-Caro-weights}}{\riso}
\alpha _{1+} \alpha _{1} ^{+}  \widetilde{pr} _{+} ( \beta ^{+} (\E))
$.
Since $\alpha _{1+} $ and $\alpha _{1} ^{+}$ are exact,
this implies the isomorphism
$\alpha _{1+} \mathcal{H} ^{s} \widetilde{pr} _{+} ( \widetilde{\E})
\riso
\alpha _{1+} \alpha _{1} ^{+} \mathcal{H} ^{s}  \widetilde{pr} _{+} ( \beta ^{+} (\E))$.
Hence, 
$$\chi ( X _1, \alpha _{1+} \mathcal{H} ^{s}\widetilde{pr} _{+} ( \widetilde{\E}))
=
\chi ( X _1, \alpha _{1+} \alpha _{1} ^{+} \mathcal{H} ^{s}  \widetilde{pr} _{+} ( \beta ^{+} (\E))).$$
From Lemma \ref{chi=d-times-chi} 
(recall also that from \cite[III.6.10]{sga1}, 
there exist some smooth proper formal $\V$-scheme $\X _1 $  which is a lifting of 
$X _1$), 
we have 
$\chi ( X _1, \alpha _{1+} \alpha _{1} ^{+} \mathcal{H} ^{s}  \widetilde{pr} _{+} ( \beta ^{+} (\E)))=
d _1\chi ( X _1, \mathcal{H} ^{s}  \widetilde{pr} _{+} ( \beta ^{+} (\E)))$.
Hence, we have checked that 
$\chi ( \widetilde{X} _1,  \mathcal{H} ^{s}\widetilde{pr} _{+} ( \widetilde{\E}))
=
d _1\chi ( X _1, \mathcal{H} ^{s}  \widetilde{pr} _{+} ( \beta ^{+} (\E)))$, 
which yields the desired result.

\medskip

ii) We suppose from now that $\alpha _1$ is the identity.
We prove that we can reduce to the case where $X _1= \P _k ^{1}$  and $U  _1 = \A ^{1} _{k}$.
Indeed, from Kedlaya's main theorem of \cite{Kedlaya-coveraffinebis}, by shrinking $U _1$ is necessary,
there exists a finite morphism 
$f \colon X _1 \to \P _k ^{1}$ such that $U _1 =f ^{-1} (\A _k ^{1})$
and the induced morphism $g \colon U _1 \to \A _k ^{1}$ is etale.
We get the cartesian squares: 
\begin{equation}
\xymatrix{
{\widetilde{X}} 
\ar@{}[rd] ^-{}|\square
\ar[r] _-{\alpha}
\ar@/^0.8pc/[rr] ^-{\widetilde{pr}}
\ar[d] ^-{f}
& 
{X} 
\ar[d] ^-{f}
\ar@{}[rd] ^-{}|\square
\ar[r] _-{pr}
& 
{X _1} 
\ar[d] ^-{f}
\ar@{}[rd] ^-{}|\square
& 
{U _1} 
\ar[d] ^-{g}
\ar[l] ^-{j}
\\ 
{\P _k ^{1} \times \widetilde{Y}} 
\ar[r] ^-{\alpha}
\ar@/_0.8pc/[rr] _-{\widetilde{pr}}
& 
{\P _k ^{1} \times Y} 
\ar[r] ^-{pr}
&
{\P _k ^{1} } 
&
{\A _k ^{1} .} 
\ar[l] _-{j}
}
\end{equation}
Set $\E ' := f _{+} (\E)=
f _{+} j _{!} j ^{!} (\E) \riso j _! g _+ j ^{!} (\E) 
\underset{\cite[1.3.10]{Abe-Caro-weights}}{\riso}
j _! j ^{!} f _+ (\E)  
=
j _! j ^{!} (\E ') $.
Set 
$\widetilde{\E} ' := \alpha ^{+} (\E') $.
We have the isomorphisms
$j ^{!} \mathrm{pr} _+ \alpha _+ (\widetilde{\E}')
\underset{\tiny \cite[1.3.10]{Abe-Caro-weights}}{\riso}
j ^{!} \mathrm{pr} _+ \alpha _+ f _+ (\widetilde{\E})
\riso 
j ^{!} f _+  \mathrm{pr} _+ \alpha _+ (\widetilde{\E})
\underset{\tiny \cite[1.3.10]{Abe-Caro-weights}}{\riso}
g _+  j ^{!}  \mathrm{pr} _+ \alpha _+ (\widetilde{\E})$.
Since
$j ^{!}  \mathrm{pr} _+ \alpha _+ (\widetilde{\E}) \in D ^{\mathrm{b}} _{\mathrm{isoc}} (U _{1}/K)$,
since $g$ is finite and etale,
using Lemma \ref{isoc-desc-finiteetale}, 
this yields that
$j ^{!} \mathrm{pr} _+ \alpha _+ (\widetilde{\E}')
\in 
 D ^{\mathrm{b}} _{\mathrm{isoc}} (\A ^1 _k/K)$.

To finish this step (ii), it remains to compare the Euler-Poincare characteristic. Since 
$f _+ j _! \riso j _{!} g_+$, 
since $g _+$ (because $g$ is finite and etale)
and $j _!$ are exact, 
then we get 
\begin{equation}
\label{StepII3b.(ii)iso}
\mathcal{H} ^{s} f _{+} \widetilde{pr} _{+} ( \widetilde{\E})
\underset{\ref{StepII.3iso}}{\riso} 
 \mathcal{H} ^{s} f _{+} j _!  \widetilde{pr} _{+} ( \widetilde{\FF})
\riso 
j _! g _+ \mathcal{H} ^{s}  \widetilde{pr} _{+}  ( \widetilde{\FF})
\riso 
f _+ j _! \mathcal{H} ^{s}  \widetilde{pr} _{+} ( \widetilde{\FF})
\underset{\ref{StepII.3iso}}{\riso} 
 f _{+}  \mathcal{H} ^{s}\widetilde{pr} _{+} ( \widetilde{\E}).
\end{equation}
This yields
$ \mathcal{H} ^{s}  \widetilde{pr} _{+} ( \widetilde{\E} ')
\underset{\cite[1.3.10]{Abe-Caro-weights}}{\riso}
 \mathcal{H} ^{s} f _{+} \widetilde{pr} _{+} ( \widetilde{\E})
 \underset{\ref{StepII3b.(ii)iso}}{\riso}
 f _{+}  \mathcal{H} ^{s}\widetilde{pr} _{+} ( \widetilde{\E}) $
 and then we get the last equality
$$\chi ( X _1,  \mathcal{H} ^{s}\widetilde{pr} _{+} ( \widetilde{\E}))
=
\chi ( \P _k ^{1}, f _{+}  \mathcal{H} ^{s}\widetilde{pr} _{+} ( \widetilde{\E}))
=
\chi ( \P _k ^{1}, \mathcal{H} ^{s}\widetilde{pr} _{+} ( \widetilde{\E}')).$$

\medskip 

iii) We suppose from now $X  _1 = \P ^{1} _{k}$ and $U  _1 = \A ^{1} _{k}$.
Recall that $\alpha _1= id$ and that we have checked in Step II.3.a) 
that 
$\widetilde{pr} _{+}  ( \widetilde{\FF})\in D ^{\mathrm{b}} _{\mathrm{isoc}} (U _{1}/K)$.
Hence we can apply Lemma \ref{i!j!} : for any $m \in \{0,1\}$ we have the inequality
$$\mathcal{H} ^{m} (i ^{!}  j _! \mathcal{H} ^{s}  \widetilde{pr} _{+} ( \widetilde{\FF}))
\leq 
\mathrm{rk} ( \mathcal{H} ^{s}  \widetilde{pr} _{+}  ( \widetilde{\FF}))
\underset{\ref{StepII.3iso}}{=}
\mathrm{rk} ( \mathcal{H} ^{s}  \widetilde{pr} _{+} ( \widetilde{\E})).$$
From the step I, this latter is well estimated. 
This implies that 
$\chi ( X _1, i _+ i ^{!} j _! \mathcal{H} ^{s}\widetilde{pr} _{+} ( \widetilde{\FF}))$ 
is well estimated. 

From \ref{StepII.3iso}, we obtain
 $\chi ( X _1,  \mathcal{H} ^{s}\widetilde{pr} _{+} ( \widetilde{\E}))
= 
\chi ( X _1, j _! \mathcal{H} ^{s}\widetilde{pr} _{+} ( \widetilde{\FF}))$.
Moreover, 
by using the exact triangle 
$i _+ i ^!\to \mathrm{id} \to j _+ j ^!  \to +1$ for 
$j _! \mathcal{H} ^{s}\widetilde{pr} _{+} ( \widetilde{\FF})$
 (see \cite[1.1.8.(ii)]{Abe-Caro-weights}), 
 since $j ^! j _! = \mathrm{id}$,
 we get the equality
$ \chi ( X _1, j _! \mathcal{H} ^{s}\widetilde{pr} _{+} ( \widetilde{\FF}))
=
\chi ( X _1, i _+ i ^{!} j _! \mathcal{H} ^{s}\widetilde{pr} _{+} ( \widetilde{\FF}))
+
\chi ( X _1, j _+ \mathcal{H} ^{s}\widetilde{pr} _{+} ( \widetilde{\FF}))$.
Hence, we reduce to estimate
$\chi ( X _1, j _+ \mathcal{H} ^{s}\widetilde{pr} _{+} ( \widetilde{\FF}))$. 

From Christol-Mebkhout's Theorem \cite[5.0-10]{christol-MebkhoutIV} (as described in the introduction), 
we have the following $p$-adic Euler-Poincare formula:
$$\chi ( X _1, j _+ \mathcal{H} ^{s}\widetilde{pr} _{+} ( \widetilde{\FF}))
=
\chi ( U _1, \mathcal{H} ^{s}\widetilde{pr} _{+} ( \widetilde{\FF}))
=
\mathrm{rk} \, (\mathcal{H} ^{s}\widetilde{pr} _{+} ( \widetilde{\FF})) 
\chi ( U _1)
- \mathrm{Irr} _{\infty} (\mathcal{H} ^{s}\widetilde{pr} _{+} ( \widetilde{\FF})) ,$$
where $\infty$ is the complement of $U _1$ in $X _1$, i.e. of $\A ^1 _k$ in $\P ^{1} _{k}$. 

To simplify notation, we put 
$\mathscr{F} := \mathscr{F} _{\psi} [1]$ (see the notation \ref{Fourier}) and then from \ref{Fourier1} the image by $\mathscr{F}$ of a module is a module.
Since $\G ^s :=\mathcal{H} ^{s} \widetilde{pr} _{+} ( \widetilde{\FF}) \in F\text{-}\mathrm{Isoc} (U _1/K)$,  
then it has no singular points (see Definition \cite[2.4.2]{AbeMarmora}). Hence, 
Abe-Marmora's formula \cite[4.1.6.(i)]{AbeMarmora} can be formulated of the form (see also the notation  \cite[2.4.1, 4.1.1]{AbeMarmora}): 
\begin{equation}
\label{Irr1}
 -\mathrm{rk} (\mathscr{F} ( \G ^s))
= 
\mathrm{rk} ( (\G ^s |\eta _\infty) _{]1,\infty  [})
- \mathrm{Irr}  ( (\G ^s |\eta _\infty) _{]1,\infty  [}). 
\end{equation}
With \cite[2.3.2.2]{AbeMarmora} (resp. \ref{rem-Irr-esti-rk}), 
we get the equality (resp. inequality):
\begin{equation}
\label{Irr2}
\mathrm{Irr} _{\infty} (\G ^s)
= \mathrm{Irr}  ( \G ^s |\eta _\infty)
\leq 
\mathrm{rk}  ( \G ^s |\eta _\infty)
+
\mathrm{Irr}  ( (\G ^s |\eta _\infty) _{]1,\infty  [}).
\end{equation}
Since $\mathrm{rk}  ( \G ^s)= \mathrm{rk}  ( \G ^s |\eta _\infty)$
and $\mathrm{rk} ( (\G ^s |\eta _\infty) _{]1,\infty  [}) \leq \mathrm{rk}  ( \G ^s)$, 
we get from \ref{Irr1} and \ref{Irr2} the inequality :
$$\mathrm{Irr} _{\infty} (\mathcal{H} ^{s}\widetilde{pr} _{+} ( \widetilde{\FF})) \leq 
\mathrm{rk} (\mathscr{F}  ( \mathcal{H} ^{s} \widetilde{pr} _{+} ( \widetilde{\FF}) ) )
+
2\mathrm{rk} ( \mathcal{H} ^{s} \widetilde{pr} _{+} ( \widetilde{\FF}) ).$$ 
Hence, we reduce to check the step iv).

iv) In this step, we estimate
$\mathrm{rk} (\mathscr{F}  ( \mathcal{H} ^{s} \widetilde{pr} _{+} ( \widetilde{\FF}) ) )$.
Since $\beta = \alpha$,
we get the diagram
\begin{equation}
\xymatrix{
{\widetilde{X}} 
\ar@{}[rd] ^-{}|\square
\ar[r] _-{\alpha}
\ar@/^0.8pc/[rr] ^-{\widetilde{pr}}
& 
{X} 
\ar@{}[rd] ^-{}|\square
\ar[r] _-{pr}
& 
{X _1} 
\\ 
{\widetilde{U}} 
\ar[u] ^-{j}
\ar[r] ^-{\alpha}
\ar@/_0.8pc/[rr] _-{\widetilde{pr}}
& 
{U} 
\ar[u] ^-{j}
\ar[r] ^-{pr}
&
{U  _1 ,} 
\ar[u] ^-{j}
}
\end{equation}
where $U = \A ^1 _Y$, $\widetilde{U}= \A ^1 _{\widetilde{Y}}$.
Set $\M := j _{+}\mathscr{F}  (j ^{!}\E)$,
where $\mathscr{F} := \mathscr{F} _{\psi} [1] \colon 
F\text{-}D^{\leq 0} _{\mathrm{ovhol}}(\mathbb{A} _Y^1)
\to 
F\text{-}D^{\leq 0} _{\mathrm{ovhol}}(\mathbb{A} _Y^1)$ (see the notation \ref{Fourier} 
and  \ref{Fourier1} for the acyclicity). 
From the step I) applied to 
$\M $,
there exists a constant 
$c $ (only depending on $\E$) such that
\begin{itemize}
\item For any $s$, $\mathrm{rk} (\mathcal{H} ^{s} \widetilde{pr} _{+} ( \alpha ^{+ }\M)) \leq  c  \prod _{b=2} ^{n} d _{b}$.
\item For any $s\geq 1$, $\mathrm{rk} (\mathcal{H} ^{s} \widetilde{pr} _{+} ( \alpha ^{+ }\M ))
\leq  c \max \{ \prod _{b\in B} d _{b} \, | \, B \subset \{ 2, \dots, n\} \text{ and }|B| = n-1-s \} $.
\end{itemize}
It remains to check that 
$\mathrm{rk} (\mathscr{F}  ( \mathcal{H} ^{s} \widetilde{pr} _{+} ( \widetilde{\FF}) ) )
=
\mathrm{rk} (\mathcal{H} ^{s} \widetilde{pr} _{+} ( \alpha ^{+ }\M) )$. 
By base change (recall that $\alpha ^!= \alpha ^+$) and next by using \ref{Fourierf^!f_+}, we have
$$\widetilde{pr} _{+} ( \alpha ^{+ }\M)
= 
\widetilde{pr} _{+} ( \alpha ^{+ } j _{+}\mathscr{F}  (\FF))
\riso 
\widetilde{pr} _{+}  j _{+}  \alpha ^{+ }   (\mathscr{F}  (\FF))
\riso
j _{+}  \widetilde{pr} _{+}   \alpha ^{+ }   (\mathscr{F}  (\FF))
\riso 
j _{+}  \widetilde{pr} _{+}  (\mathscr{F}  (\widetilde{\FF}))
\riso 
j _{+}   (\mathscr{F}  ( \widetilde{pr} _{+} \widetilde{\FF})),
$$
where $\mathscr{F} := \mathscr{F} _{\psi} [1] \colon 
F\text{-}D^{\mathrm{b}}_{\mathrm{ovhol}}(\mathbb{A} _S^1)
\to 
F\text{-}D^{\mathrm{b}}_{\mathrm{ovhol}}(\mathbb{A} _S^1)$ is the shifted Fourier transform
for respectively $S=Y$, $S = \widetilde{Y}$ or $S =\Spec k$.
Since $\mathscr{F} $ and $j _+$
are acyclic  (see \ref{Fourier1} ), 
then we get 
$\mathcal{H} ^{s} \widetilde{pr} _{+} ( \alpha ^{+ }\M)
\riso 
j _{+}   (\mathscr{F}  ( \mathcal{H} ^{s}\widetilde{pr} _{+} \widetilde{\FF}))$.
Since 
$\mathrm{rk} (\mathscr{F}  ( \mathcal{H} ^{s}\widetilde{pr} _{+} \widetilde{\FF})
=
\mathrm{rk} (j _{+}   (\mathscr{F}  ( \mathcal{H} ^{s}\widetilde{pr} _{+} \widetilde{\FF}))$ (recall the notation of \ref{ntn-rk}),
we can conclude. 

\end{proof}

From Theorem \ref{BBD4.5.1}, the reader can 
check the $p$-adic analogues of corollaries \cite[4.5.2--5]{BBD} by copying the proofs. 
Moreover, from \cite{Abe-Caro-weights}, we have a theory of weight in the framework of arithmetic $\D$-modules. 
For instance, we have checked the stability of the weight under Grothendieck six operations, i.e. the $p$-adic analogue of Deligne famous work
in \cite{deligne-weil-II}), which is also explained in \cite[5.1.14]{BBD}.
In \cite[5.2.1]{BBD}, a reverse implication was proved.
The reader can check that we can copy the proof without further problems (i.e., we only have to check that we have nothing new to check, e.g. we have 
already \ref{estiH-d} or the purity of the middle extension of some pure unipotent $F$-isocrystal as given in \cite[3.6.3]{Abe-Caro-weights}).
For the reader, let us write this $p$-adic version and its important corollary \cite[5.3.1]{BBD} (this corollary is proved in \cite{Abe-Caro-weights} in 
another way, but Theorem \ref{5.2.1BBD} below is a new result). 
\begin{thm}
[{\cite[5.2.1]{BBD}}]
\label{5.2.1BBD}
We suppose $k= \F _{p ^s}$ is finite and that $F$ means the $sth$ power of Frobenius. 
Choose an isomorphism of the form $\iota \colon \overline{\Q} _p \riso \C$. 
Let $X$ be a $k$-variety and $\E \in F \text{-}\mathrm{Ovhol} (X/K)$.
We suppose that, for any etale morphism 
$\alpha \colon U \to X$ with $U$ affine, the $K$-vector space
 $H ^{0} (p _{U +}( \alpha ^{+}(\E))$ is $\iota$-mixed of weight $\geq w$. 
 Then $\E$ is $\iota$-mixed of weight $\geq w$.

\end{thm}

\begin{coro}
[{\cite[5.3.1]{BBD}}]
With the notation \ref{5.2.1BBD}, if $\E$ is $\iota$-mixed of weight $\geq w$ (resp. $\leq w$), then any subquotient of $\E$ is 
$\iota$-mixed of weight $\geq w$ (resp. $\leq w$).
\end{coro}
Finally, except 
\cite[5.4.7--8]{BBD}, the reader can check easily the other results of 
the chapter $5$ of \cite{BBD} by translating the proofs in our $p$-adic context.

\small
\bibliographystyle{alpha}

\begin{thebibliography}{CM01b}

\bibitem[Abe14]{Abe-Frob-Poincare-dual}
Tomoyuki Abe.
\newblock Explicit calculation of {F}robenius isomorphisms and {P}oincar\'e
  duality in the theory of arithmetic {$\scr{D}$}-modules.
\newblock {\em Rend. Semin. Mat. Univ. Padova}, 131:89--149, 2014.

\bibitem[AC13a]{Abe-Caro-BEq}
Tomoyuki Abe and Daniel Caro.
\newblock On beilinson's equivalence for $p$-adic cohomology.
\newblock 09 2013.

\bibitem[AC13b]{Abe-Caro-weights}
Tomoyuki Abe and Daniel Caro.
\newblock Theory of weights in p-adic cohomology.
\newblock 03 2013.

\bibitem[AM]{AbeMarmora}
Tomoyuki Abe and Adriano Marmora.
\newblock Product formula for p-adic epsilon factors.
\newblock {\em to appear in J. Inst. Math. Jussieu}.

\bibitem[BBD82]{BBD}
A.~A. Be{\u\i}linson, J.~Bernstein, and P.~Deligne.
\newblock Faisceaux pervers.
\newblock In {\em Analysis and topology on singular spaces, {I} ({L}uminy,
  1981)}, volume 100 of {\em Ast{\'e}risque}, pages 5--171. Soc. Math. France,
  Paris, 1982.

\bibitem[Ber84]{MR773087}
Pierre Berthelot.
\newblock Cohomologie rigide et th\'eorie de {D}work: le cas des sommes
  exponentielles.
\newblock {\em Ast\'erisque}, (119-120):3, 17--49, 1984.
\newblock $p$-adic cohomology.

\bibitem[Ber90]{Be0}
Pierre Berthelot.
\newblock {Cohomologie rigide et th\'eorie des $\mathcal{D}$-modules}.
\newblock In {\em $p$-adic analysis (Trento, 1989)}, pages 80--124. Springer,
  Berlin, 1990.

\bibitem[Ber97]{Ber-alterationdejong}
Pierre Berthelot.
\newblock Alt\'erations de vari\'et\'es alg\'ebriques (d'apr\`es {A}. {J}. de
  {J}ong).
\newblock {\em Ast\'erisque}, (241):Exp.\ No.\ 815, 5, 273--311, 1997.
\newblock S\'eminaire Bourbaki, Vol.\ 1995/96.

\bibitem[Ber00]{Be2}
Pierre Berthelot.
\newblock {$\mathcal{D}$}-modules arithm\'etiques. {I}{I}. {D}escente par
  {F}robenius.
\newblock {\em M\'em. Soc. Math. Fr. (N.S.)}, (81):vi+136, 2000.

\bibitem[Ber02]{Beintro2}
Pierre Berthelot.
\newblock {Introduction \`a la th\'eorie arithm\'etique des
  {$\mathcal{D}$}-modules}.
\newblock {\em Ast\'erisque}, (279):1--80, 2002.
\newblock Cohomologies {$p$}-adiques et applications arithm\'etiques, {II}.

\bibitem[Car]{Caro-Lagrangianity}
Daniel Caro.
\newblock Lagrangianity for log extendable overconvergent {$F$}-isocrystals.
\newblock {\em To appear in Math. Zeitschrift}.

\bibitem[Car06a]{caro_devissge_surcoh}
Daniel Caro.
\newblock {D{\'e}vissages des $F$-complexes de $\mathcal{D}$-modules
  arithm{\'e}tiques en $F$-isocristaux surconvergents}.
\newblock {\em Invent. Math.}, 166(2):397--456, 2006.

\bibitem[Car06b]{caro_courbe-nouveau}
Daniel Caro.
\newblock Fonctions {L} associ{\'e}es aux {$\mathcal{D}$}-modules
  arithm{\'e}tiques. {C}as des courbes.
\newblock {\em Compositio Mathematica}, 142(01):169--206, 2006.

\bibitem[Car07]{caro-2006-surcoh-surcv}
Daniel Caro.
\newblock {Overconvergent F-isocrystals and differential overcoherence}.
\newblock {\em Invent. Math.}, 170(3):507--539, 2007.

\bibitem[CM01a]{Christol-Mebkhout_IV}
G.~Christol and Z.~Mebkhout.
\newblock Sur le th\'eor\`eme de l'indice des \'equations diff\'erentielles
  {$p$}-adiques. {IV}.
\newblock {\em Invent. Math.}, 143(3):629--672, 2001.

\bibitem[CM01b]{christol-MebkhoutIV}
G.~Christol and Z.~Mebkhout.
\newblock Sur le th\'eor\`eme de l'indice des \'equations diff\'erentielles
  {$p$}-adiques. {IV}.
\newblock {\em Invent. Math.}, 143(3):629--672, 2001.

\bibitem[CM02]{christol-Mebkhout2002}
Gilles Christol and Zoghman Mebkhout.
\newblock \'{E}quations diff\'erentielles {$p$}-adiques et coefficients
  {$p$}-adiques sur les courbes.
\newblock {\em Ast\'erisque}, (279):125--183, 2002.
\newblock Cohomologies $p$-adiques et applications arithm\'etiques, II.

\bibitem[Cre98]{crewfini}
Richard Crew.
\newblock Finiteness theorems for the cohomology of an overconvergent
  isocrystal on a curve.
\newblock {\em Ann. Sci. \'Ecole Norm. Sup. (4)}, 31(6):717--763, 1998.

\bibitem[CT12]{caro-Tsuzuki}
Daniel Caro and Nobuo Tsuzuki.
\newblock Overholonomicity of overconvergent {$F$}-isocrystals over smooth
  varieties.
\newblock {\em Ann. of Math. (2)}, 176(2):747--813, 2012.

\bibitem[Del80]{deligne-weil-II}
Pierre Deligne.
\newblock La conjecture de {W}eil. {II}.
\newblock {\em Inst. Hautes \'Etudes Sci. Publ. Math.}, (52):137--252, 1980.

\bibitem[dJ96]{dejong}
A.~J. de~Jong.
\newblock Smoothness, semi-stability and alterations.
\newblock {\em Inst. Hautes \'Etudes Sci. Publ. Math.}, (83):51--93, 1996.

\bibitem[Ful98]{Fulton-Intersection}
William Fulton.
\newblock {\em Intersection theory}, volume~2 of {\em Ergebnisse der Mathematik
  und ihrer Grenzgebiete. 3. Folge. A Series of Modern Surveys in Mathematics
  [Results in Mathematics and Related Areas. 3rd Series. A Series of Modern
  Surveys in Mathematics]}.
\newblock Springer-Verlag, Berlin, second edition, 1998.

\bibitem[EGA I]{EGAI}
A.~Grothendieck.
\newblock \'{E}l\'ements de g\'eom\'etrie alg\'ebrique. {I}. {L}e langage des
  sch\'emas.
\newblock {\em Inst. Hautes \'Etudes Sci. Publ. Math.}, (4):228, 1960.

\bibitem[EGA II]{EGAII}
A.~Grothendieck.
\newblock \'{E}l\'ements de g\'eom\'etrie alg\'ebrique. {II}. \'{E}tude globale
  \'el\'ementaire de quelques classes de morphismes.
\newblock {\em Inst. Hautes \'Etudes Sci. Publ. Math.}, (8):222, 1961.

\bibitem[EGA IV1]{EGAIV1}
A.~Grothendieck.
\newblock \'{E}l\'ements de g\'eom\'etrie alg\'ebrique. {IV}. \'{E}tude locale
  des sch\'emas et des morphismes de sch\'emas. {I}.
\newblock {\em Inst. Hautes \'Etudes Sci. Publ. Math.}, (20):259, 1964.

\bibitem[EGA IV2]{EGAIV2}
A.~Grothendieck.
\newblock \'{E}l\'ements de g\'eom\'etrie alg\'ebrique. {IV}. \'{E}tude locale
  des sch\'emas et des morphismes de sch\'emas. {II}.
\newblock {\em Inst. Hautes \'Etudes Sci. Publ. Math.}, (24):231, 1965.

\bibitem[EGA IV3]{EGAIV3}
A.~Grothendieck.
\newblock \'{E}l\'ements de g\'eom\'etrie alg\'ebrique. {IV}. \'{E}tude locale
  des sch\'emas et des morphismes de sch\'emas. {III}.
\newblock {\em Inst. Hautes \'Etudes Sci. Publ. Math.}, (28):255, 1966.

\bibitem[EGA IV4]{EGAIV4}
A.~Grothendieck.
\newblock \'{E}l\'ements de g\'eom\'etrie alg\'ebrique. {IV}. \'{E}tude locale
  des sch\'emas et des morphismes de sch\'emas {IV}.
\newblock {\em Inst. Hautes \'Etudes Sci. Publ. Math.}, (32):361, 1967.

\bibitem[Ked03]{kedlaya_semi-stable}
Kiran~S. Kedlaya.
\newblock Semistable reduction for overconvergent {$F$}-isocrystals on a curve.
\newblock {\em Math. Res. Lett.}, 10(2-3):151--159, 2003.

\bibitem[Ked05]{Kedlaya-coveraffinebis}
Kiran~S. Kedlaya.
\newblock More \'etale covers of affine spaces in positive characteristic.
\newblock {\em J. Algebraic Geom.}, 14(1):187--192, 2005.

\bibitem[Ked10]{Kedlaya-padicDiffEq}
Kiran~S. Kedlaya.
\newblock {\em {$p$}-adic differential equations}, volume 125 of {\em Cambridge
  Studies in Advanced Mathematics}.
\newblock Cambridge University Press, Cambridge, 2010.

\bibitem[Ked11]{kedlaya-semistableIV}
Kiran~S. Kedlaya.
\newblock Semistable reduction for overconvergent {$F$}-isocrystals, {IV}:
  local semistable reduction at nonmonomial valuations.
\newblock {\em Compos. Math.}, 147(2):467--523, 2011.

\bibitem[KL85]{KatzLaumon}
Nicholas~M. Katz and G{\'e}rard Laumon.
\newblock Transformation de {F}ourier et majoration de sommes exponentielles.
\newblock {\em Inst. Hautes \'Etudes Sci. Publ. Math.}, (62):361--418, 1985.

\bibitem[Lan02]{Lang-Algebra}
Serge Lang.
\newblock {\em Algebra}, volume 211 of {\em Graduate Texts in Mathematics}.
\newblock Springer-Verlag, New York, third edition, 2002.

\bibitem[Lau85]{Laumon-TransfCanon}
G.~Laumon.
\newblock Transformations canoniques et sp\'ecialisation pour les {${\scr
  D}$}-modules filtr\'es.
\newblock {\em Ast\'erisque}, (130):56--129, 1985.
\newblock Differential systems and singularities (Luminy, 1983).

\bibitem[Liu02]{Liu-livre-02}
Qing Liu.
\newblock {\em Algebraic geometry and arithmetic curves}, volume~6 of {\em
  Oxford Graduate Texts in Mathematics}.
\newblock Oxford University Press, Oxford, 2002.
\newblock Translated from the French by Reinie Ern{{\'e}}, Oxford Science
  Publications.

\bibitem[MW68]{MonskyWashnitzer}
P.~Monsky and G.~Washnitzer.
\newblock Formal cohomology. {I}.
\newblock {\em Ann. of Math. (2)}, 88:181--217, 1968.

\bibitem[NH04]{Noot-Huyghe-fourierI}
Christine Noot-Huyghe.
\newblock Transformation de {F}ourier des {$\mathcal{D}$}-modules
  arithm\'etiques. {I}.
\newblock In {\em Geometric aspects of Dwork theory. Vol. I, II}, pages
  857--907. Walter de Gruyter GmbH \& Co. KG, Berlin, 2004.

\bibitem[SGA1]{sga1}
{\em Rev\^etements \'etales et groupe fondamental ({SGA} 1)}.
\newblock Documents Math\'ematiques (Paris) [Mathematical Documents (Paris)],
  3. Soci\'et\'e Math\'ematique de France, Paris, 2003.
\newblock S\'eminaire de g\'eom\'etrie alg\'ebrique du Bois Marie 1960--61.
  [Geometric Algebra Seminar of Bois Marie 1960-61], Directed by A.
  Grothendieck, With two papers by M. Raynaud, Updated and annotated reprint of
  the 1971 original [Lecture Notes in Math., 224, Springer, Berlin; ].

\bibitem[Vir00]{virrion}
Anne Virrion.
\newblock Dualit\'e locale et holonomie pour les {$\mathcal{D}$}-modules
  arithm\'etiques.
\newblock {\em Bull. Soc. Math. France}, 128(1):1--68, 2000.

\end{thebibliography}
\def\cprime{$'$}

\bigskip
\noindent Daniel Caro\\
Laboratoire de Math\'ematiques Nicolas Oresme\\
Universit\'e de Caen
Campus 2\\
14032 Caen Cedex\\
France.\\
email: daniel.caro@unicaen.fr

\end{document}